\documentclass[11pt]{amsart}
\usepackage[utf8]{inputenc}
\usepackage{accents}
	
\usepackage{subfiles}
\usepackage{lipsum}
\usepackage[font=itshape]{quoting}
\usepackage{tabu, fullpage,diagbox, booktabs}
\usepackage{placeins}
\usepackage[margin=1in]{geometry}
\usepackage[dvipsnames]{xcolor}   
\usepackage{graphicx,comment}			

\usepackage{csquotes} 
\usepackage{dsfont}
\usepackage{mathrsfs}
\usepackage{amsthm}
\usepackage{amsmath}
\usepackage{fdsymbol}

\usepackage{stmaryrd}
\usepackage{tikz,framed}
\usepackage{tikz-cd}
\usepackage{enumitem}
\usepackage{bm}
\usepackage{mathtools}
\usepackage[all]{xy}
\usepackage{caption}

\usepackage{url}
\usepackage{float}
\usepackage{todonotes}
\usepackage{bbm}
\usepackage{easyeqn}

\usepackage{textcomp}
\usepackage[colorlinks,
			citecolor=citation, 
			urlcolor=citation, 
			linkcolor=reference, 
            backref = page]{hyperref}

\definecolor{citation}{rgb}{0,.40,.80}
\definecolor{reference}{rgb}{.80,0,.40}
\definecolor{todogray}{HTML}{E4E4E4}

\newcommand{\PKHcomment}[1]{{\color{red}PKH: #1}}

\newcommand{\XYHcomment}[1]{\textcolor{brown}{[XYH: #1]}}
\newcommand{\MScomment}[1]{\textcolor{OliveGreen}{[MS: #1]}}

\tikzset{
	commutative diagrams/.cd, 
	arrow style=tikz, 
	diagrams={>=stealth}
}

\theoremstyle{definition}

\makeatletter
\newcommand{\colim@}[2]{%
	\vtop{\m@th\ialign{##\cr
			\hfil$#1\operator@font colim$\hfil\cr
			\noalign{\nointerlineskip\kern1.5\ex@}#2\cr
			\noalign{\nointerlineskip\kern-\ex@}\cr}}%
}
\newcommand{\colim}{%
	\mathop{\mathpalette\colim@{\rightarrowfill@\textstyle}}\nmlimits@
}
\makeatother

\makeatletter
\def\@tocline#1#2#3#4#5#6#7{\relax
	\ifnum #1>\c@tocdepth 
	\else
	\par \addpenalty\@secpenalty\addvspace{#2}%
	\begingroup \hyphenpenalty\@M
	\@ifempty{#4}{%
		\@tempdima\csname r@tocindent\number#1\endcsname\relax
	}{%
		\@tempdima#4\relax
	}%
	\parindent\z@ \leftskip#3\relax \advance\leftskip\@tempdima\relax
	\rightskip\@pnumwidth plus4em \parfillskip-\@pnumwidth
	#5\leavevmode\hskip-\@tempdima
	\ifcase #1
	\or\or \hskip 1em \or \hskip 2em \else \hskip 3em \fi%
	#6\nobreak\relax
	\dotfill\hbox to\@pnumwidth{\@tocpagenum{#7}}\par
	\nobreak
	\endgroup
	\fi}
\makeatother

\usetikzlibrary{calc}
\usetikzlibrary{fadings}
\usetikzlibrary{decorations.pathmorphing}
\usetikzlibrary{decorations.pathreplacing}

\newcounter{marginnote}
\DeclareMathAlphabet{\mathpzc}{OT1}{pzc}{m}{it}

\hypersetup{
	colorlinks   = true,          
	urlcolor     = blue,          
	linkcolor    = purple,          
	citecolor   = blue             
}

\theoremstyle{definition}
\newtheorem{maintheorem}{Theorem}
\newtheorem{maincorollary}[maintheorem]{Corollary}

\newtheorem{theorem}{Theorem}[subsection]

\newtheorem{claim}[theorem]{Claim}

\newtheorem{corollary}[theorem]{Corollary}
\newtheorem{lemma}[theorem]{Lemma}
\newtheorem{proposition}[theorem]{Proposition}

\newtheorem{remark}[theorem]{Remark}

\newtheorem*{runningexample*}{Running example}

\newtheorem{construction}[theorem]{Construction}

\newtheorem{definition}[theorem]{Definition}
\newtheorem{example}[theorem]{Example}

\newtheorem{proposition-definition}[theorem]{Proposition-Definition}

\newtheorem{outcome}{Outcome}

\usepackage{etoolbox}
\newtheorem*{theorem*}{Theorem}

\DeclareMathOperator{\arrowup}{\!\Uparrow \!}
\DeclareMathOperator{\arrowdown}{\!\Downarrow\!}

\newcommand{\cA}{\mathcal{A}}

\newcommand{\cC}{\mathcal{C}}

\newcommand{\cE}{\mathcal{E}}
\newcommand{\cF}{\mathcal{F}}
\newcommand{\cG}{\mathcal{G}}

\newcommand{\cI}{\mathcal{I}}

\newcommand{\cL}{\mathcal{L}}
\newcommand{\cM}{\mathcal{M}}
\newcommand{\cN}{\mathcal{N}}
\newcommand{\cO}{\mathcal{O}}

\newcommand{\cT}{\mathcal{T}}
\newcommand{\cU}{\mathcal{U}}
\newcommand{\cV}{\mathcal{V}}

\newcommand{\cX}{\mathcal{X}}
\newcommand{\cY}{\mathcal{Y}}

\newcommand{\LogO}[1]{\mathcal{O}_{{#1}_\text{l\'et}}}

 \newcommand{\bcd}{\begin{center}\begin{tikzcd}}
	\newcommand{\ecd}{\end{tikzcd}\end{center}}

\usepackage{accents}

\newcommand{\ul}[1]{\underline{#1}}

\DeclareMathAlphabet{\mathpzc}{OT1}{pzc}{m}{it}

\NewDocumentCommand{\compatibilitydatum}{m m m m m m O{} O{} O{}}{
	\begin{equation*} \begin{tikzcd}[ampersand replacement=\&]
	\: \arrow{r} \& {#1} \arrow{r} \arrow{d}{#7} \& {#2} \arrow{r} \arrow{d}{#8} \& {#3} \arrow{r}{[1]} \arrow{d}{#9} \& \: \\
	\: \arrow{r} \& {#4} \arrow{r} \& {#5} \arrow{r} \& {#6} \arrow{r} \& \:
	\end{tikzcd} \end{equation*}}

\NewDocumentCommand{\commutingsquare}{m m m m o O{} O{} O{} O{}}{
	\begin{equation}\begin{tikzcd}[ampersand replacement=\&] \label{#5}
	#1 \arrow{r}{#6} \arrow{d}{#7} \& #2 \arrow{d}{#8} \\
	#3 \arrow{r}{#9} \& #4
	\end{tikzcd}\IfValueTF{#5}{\label{#5}}{} \end{equation}}

\NewDocumentCommand{\cartesiansquarelabel}{m m m m m O{} O{} O{} O{}}{
	\begin{tikzcd}[ampersand replacement=\&]
	#1 \arrow{r}{#6} \arrow{d}{#7} \arrow[dr, phantom, "\square"] \& #2 \arrow{d}{#8} \\
	#3 \arrow{r}{#9} \& #4
	\end{tikzcd}\IfValueTF{#5}{\label{#5}}{}
}

\NewDocumentCommand{\triangleofspaces}{m m m O{} O{} O{}}{
	\begin{tikzcd} [ampersand replacement=\&]
	#1 \arrow{r}{#4} \arrow[bend right]{rr}{#5} \& #2 \arrow{r}{#6} \& #3
	\end{tikzcd}}

 \newcommand{\cal}{\mathcal}

\def\cA{{\cal A}}

\def\cC{{\cal C}}

\def\cE{{\cal E}}
\def\cF{{\cal F}}

\def\cL{{\cal L}}
\def\cM{{\cal M}}
\def\cN{{\cal N}}
\def\cO{{\cal O}}

\def\cT{{\cal T}}
\def\cU{{\cal U}}
\def\cV{{\cal V}}

\def\cX{{\cal X}}
\def\cY{{\cal Y}}

\docsvlist{A,B,C,D,E,F,G,H,I,J,K,L,M,N,O,P,Q,R,S,T,U,V,W,X,Y,Z} 

\docsvlist{A,B,C,D,E,F,G,H,I,J,K,L,M,N,O,P,Q,R,S,T,U,V,W,X,Y,Z} 

\newcommand{\logqcoh}[1]{\cL \text{QCoh}(#1)}
\newcommand{\logcoh}[1]{\cL \text{Coh}(#1)}

\newcommand{\rt}[1]{{#1}_{\text{l\'et}}^r}
\newcommand{\longleftrightarrows}{\mathrel{\substack{\textstyle\longrightarrow \\[-0.3ex] \textstyle\longleftarrow}}}
\makeatletter
\newcommand{\longtwoheadrightarrow}{\mathrel{\begin{tikzpicture}[baseline=-0.5ex, evar/.style={font=\scriptsize}]
    \draw [->>, line width=0.06ex] (0,0) -- (1.5em,0);
\end{tikzpicture}}}
\makeatother

\usetikzlibrary{arrows.meta}

\usepackage[normalem]{ulem}
\usepackage{cleveref}

\title{Coherent Sheaves in Logarithmic Geometry}

    \author{Hannah Dell}
	\address{Mathematisches Institut, Rheinische Friedrich-Wilhelms-Universität Bonn, Bonn, Germany}
	\email{\href{mailto:hdell@math.uni-bonn.de}{hdell@math.uni-bonn.de}}
	\author{Xianyu Hu}
	\address{Fakult\"at f\"ur Mathematik,
	Technische Universit\"at M\"unchen, M\"unchen, Germany}
	\email{\href{mailto:xianyu.hu@tum.de}{xianyu.hu@tum.de}}
    \author{Patrick Kennedy-Hunt}
	\address{Max Planck Institute for Mathematics, Bonn, Germany}
	\email{\href{mailto:pfk21@cam.ac.uk}{pfk21@cam.ac.uk}}
    \author{Kabeer Manali Rahul}
    \address{Max Planck Institute for Mathematics, Bonn, Germany}
    \email{\href{mailto:manalirahul@mpim-bonn.mpg.de}{manalirahul@mpim-bonn.mpg.de}}
	\author{Maximilian Schimpf}
	\address{Max Planck Institute for Mathematics, Bonn, Germany}
	\email{\href{mailto:mschimpf@mpim-bonn.mpg.de}{schimpf@mpim-bonn.mpg.de}}
\thanks{This document was obtained by merging two manuscripts which independently produced similar results. One project was joint work of Dell, Kennedy-Hunt, and Manali Rahul. The other project was joint work of Hu and Schimpf.}
\begin{document}
\begin{abstract}
This paper introduces an abelian category of logarithmic coherent sheaves that arranges coherent sheaves across all expansions and root stacks of a simple normal crossing degeneration. Formally, logarithmic coherent sheaves are coherent sheaves in the full logarithmic \'etale topology.
We develop a suite of tools that reduces the evaluation of the basic functors of homological algebra to the conventional calculation on a computable logarithmic alteration. A second paper will establish good properties of the associated logarithmic derived category.

We thus offer a unified perspective on logarithmic moduli spaces of coherent sheaves: The logarithmic Quot spaces motivated by Maulik and Ranganathan's logarithmic Donaldson--Thomas theory, the logarithmic Picard group constructed by Molcho and Wise, and moduli spaces of logarithmic parabolic sheaves as developed by Borne, Talpo, and Vistoli. In establishing the connection with logarithmic Picard groups, we offer a new interpretation of chip firing as the combinatorial shadow to a logarithmic version of S-equivalence.
\end{abstract}

\maketitle

\tableofcontents

\section*{Introduction}
Moduli spaces of coherent sheaves on a smooth projective variety $X$ behave unpredictably as $X$ undergoes a simple normal crossing degeneration to a singular scheme $X_0$. For example, the relative Picard scheme may not be universally closed \cite{AltmanKleiman1979,AltmanKleiman1980,DSouza1979Compactification,Ishida1978,OdaSeshadri1979}; for $X$ a surface, the relative Hilbert scheme of points is not a simple normal crossing degeneration \cite{GulbrandsenHalleHulek2019,Wu-Li-Good-degeneration-of-Quot-schemes-and-coherent-systems,shafi2025logarithmic,tschanz2023expansions,tschanz2024good}; and degeneration formulae as used in enumerative geometry are not available \cite{KuhnGluing,Wu-Li-Good-degeneration-of-Quot-schemes-and-coherent-systems,maulik2024logarithmic,maulik2025logarithmic}. 

A single geometric idea, first employed by Gieseker \cite{Gieseker1984ADO}, has become a central tool for constructing good degenerations of these moduli spaces: study a larger moduli stack that tracks coherent sheaves across all \textit{expansions}\footnote{See Definition \ref{definition-log-modification}.} of $X_0$ (see, for instance, \cite{GiesekerLi1994,halpern2025quantum,kennedy2023logarithmic,kennedy2025logarithmic,KuhnGluing,Wu-Li-Good-degeneration-of-Quot-schemes-and-coherent-systems,martens2016compactifications,maulik2025logarithmic,nadler2023automorphicgluingfunctorbetti,Nagaraj-Seshadri-Degenerations-of-the-moduli-I,Nagaraj-Seshadri-Degenerations-of-the-moduli-II,shafi2025logarithmic}). However, a setting for homological algebra that captures this geometry has not emerged. This paper addresses the following basic question.
\begin{center} 
\noindent\fbox{%
\parbox{\linewidth - 2\fboxsep}{%
$$\text{Is there a well-behaved abelian category of coherent sheaves across all expansions of $X_0$?}$$ }%
} 
\end{center}

We address this question by studying the category of \textit{logarithmic coherent sheaves,} which we define to be coherent sheaves in the logarithmic \'etale topology. This definition arranges conventional coherent sheaves across all expansions and logarithmic root stacks of $X_0$. In particular, if $X_0$ has trivial logarithmic structure, then logarithmic coherent sheaves are the same as conventional coherent sheaves. Here, root stacks are included because one needs to apply toric semistable reduction in order to avoid non-reduced expansions. In doing so, two basic issues must be overcome. First, identifying an abelian category whose objects are precisely those of geometric origin requires us to address well-known subtleties concerning logarithmic \'etale descent \cite{molcho2023remarks}. Second, it is not clear a priori how to compute the natural derived functors (sections, pullbacks, pushforwards, tensor products, and hom sheaves).

An earlier notion of coherent sheaves on logarithmic schemes, called \textit{parabolic sheaves}, was pioneered in \cite{Parabolic-sheaves-on-logarithmic-schemes,MatsukiOlsson2005,Olsson2007TwistedCurves,talpo2018infinite}. We were unable to recover the geometry of expansions using the framework of parabolic sheaves. However, we expect that all of the moduli spaces of sheaves on expansions discussed above are captured by moduli stacks of logarithmic coherent sheaves, which contain open substacks of logarithmically flat parabolic sheaves \cite{talpo2017moduli}.


Indeed, up to subtleties detailed in \cite{Kennedy-Hunt-Poiret-Song:Log-Vector-bundle}, logarithmic Quot spaces \cite{kennedy2023logarithmic,maulik2024logarithmic}, and in particular the logarithmic Donaldson--Thomas theory of ideal sheaves and stable pairs \cite{maulik2024logarithmic,maulik2025logarithmic, maulik2026logarithmic}, are moduli spaces of quotients of a fixed logarithmic coherent sheaf. 
Whilst this theory is now well-developed, 
higher rank logarithmic Donaldson--Thomas theory is an important missing part of the story. 
The results we present suggest that such a theory should exist, and offer first steps towards its construction. 
Access to a logarithmic derived category offers a route to studying logarithmic moduli of complexes, whose conventional counterparts have been used to celebrated effect \cite{feyzbakhsh2023rank,feyzbakhsh2023curve,feyzbakhsh2024rank}.

In order to use the category of logarithmic coherent sheaves to study moduli spaces under degeneration, we develop a suite of tools to compute elementary operations.

\begin{outcome}[The yoga of ups and downs\footnote{The calculus proceeds via an up functor and a down functor. See Sections \ref{section: log coh intro} and \ref{subsection ob log etale ringed topoi}.}]\label{OutcomeI}
Develop a calculus to compute
    \[
        L\pi^\ast,\;R\pi_\ast,\;  \otimes^\mathbb{L}, \; \text{and } R\mathcal{H}om,
    \]
    by reducing to the corresponding conventional (i.e. non-logarithmic) computation on a root stack of an expansion. 
\end{outcome}

In particular, the above functors preserve logarithmic (quasi)-coherence as in the conventional setting. In the companion paper \cite{Logarithmic-derived-category}, we study the associated \textit{logarithmic derived category}, and show that it satisfies Grothendieck-Verdier duality along proper maps, and a weak version of categorical smoothness along logarithmically smooth morphisms. Moduli of logarithmic coherent sheaves are systematically developed elsewhere \cite{Kennedy-Hunt-Poiret-Song:Log-Vector-bundle}.

In contrast to logarithmic Quot spaces, the connection between the moduli stack of invertible logarithmic coherent sheaves on a logarithmically smooth family of curves $C/S$, denoted $\mathfrak{M}_1(C/S)$, and logarithmic Picard groups \cite{MolchoWise} is subtle. In particular, points of the logarithmic Picard group only parametrize equivalence classes of invertible logarithmic coherent sheaves up to the action of piecewise linear functions \cite{ModStckTropCurve,HolmesSchwarz2022,MolchoRanganathan2024}.
\begin{outcome}\label{Outcome:BeholdLogPic}
    We construct a morphism $$\varpi\colon \mathfrak{M}_1(C/S)\longrightarrow \operatorname{LogPic}(C/S)$$ and establish that $\varpi$ satisfies, in an appropriate category, the universal mapping property shared by fine, coarse, good, and adequate moduli spaces. Thus, $\operatorname{LogPic}(C/S)$ is a categorical moduli space for $\mathfrak{M}_1(C/S)$.
\end{outcome}
In Section \ref{section:log picard groups} we unpack the perspective that Outcome \ref{Outcome:BeholdLogPic} reinterprets the combinatorics of chip firing as \textit{a logarithmic version of S-equivalence}. The two differing notions (compare \cite{KavehManonTropVB,KavehManon2025,KhanMaclagan2026} with \cite{GrossKaurWernerTropical,GrossZakharov2022}) of tropical vector bundles considered elsewhere can be understood as generalizations of the objects parameterized by the tropicalizations of $\mathfrak{M}_1$ and $\operatorname{LogPic}$ respectively to the higher rank setting; see Section~\ref{sec:theInterval} for a summary. Together, Outcomes \ref{OutcomeI} and \ref{Outcome:BeholdLogPic} indicate that logarithmic coherent sheaves, as defined in this paper, provide a natural replacement for coherent sheaves in logarithmic geometry.

We now detail Outcomes \ref{OutcomeI} and \ref{Outcome:BeholdLogPic} assuming that $X$ is a Noetherian, fine, and saturated logarithmic scheme over a field $k$ with underlying scheme $\underline{X}$. Theorem~\ref{main theorem on log coherent}~-~\ref{main thm: main theorem push forward} address Outcome \ref{OutcomeI}, while Theorem \ref{mainthm:LogPic} concerns Outcome \ref{Outcome:BeholdLogPic}. We adopt the terminology that a \textit{logarithmic alteration} $X'\rightarrow X$ is any morphism that can be expressed as a composition of expansion and root stack morphisms (see Section \ref{section: log alteration}). 

\subsection{Logarithmic coherent sheaves}\label{section: log coh intro}

Our tool for organizing sheaves across all logarithmic alterations of $X$ is the \textit{logarithmic \'etale topology} \cite{m-type-topology,Overview-of-the-work-Kato-School,Log-etale-cohomology-Chikara-Nakayama}. As observed in \cite{MolchoWise,molcho2023remarks}, this topology can be defined by declaring that logarithmic alterations are covers, in addition to the (strict) \'etale covers inherited from conventional algebraic geometry. We warn the reader that logarithmic alterations are not flat and can contract entire irreducible components.

As detailed in Section \ref{section-on-log-etale-topoi}, the small logarithmic \'etale site $X_\text{l\'et}$ of $X$ carries a natural structure sheaf $\LogO{X}$. The bridge between conventional $\mathcal{O}_{\underline{X}{}'_\text{\'et}}$-modules on a logarithmic alteration $X'\rightarrow X$ and $\LogO{X}$-modules is provided by an adjoint pair of functors. These functors, introduced in Section~\ref{subsection ob log etale ringed topoi}, are the \textit{up functor} and its right adjoint, the \textit{down functor}:
$$\Uparrow\colon \operatorname{Mod}(\mathcal{O}_{\underline{X}{}'_\text{\'et}})\longleftrightarrows \operatorname{Mod}(\LogO{X})\colon \Downarrow_{X'}.$$
Whilst the up functor depends on the choice of $X'$, we omit any reference to $X'$ from our notation since it is clear from context. 

We call an $\LogO{X}$-module $\mathcal{F}$ \textit{coherent} if, logarithmic \'etale locally, it admits a presentation of the form
$$\LogO{X}^{\oplus k}\longrightarrow \LogO{X}^{\oplus \ell} \longrightarrow \mathcal{F} \longrightarrow 0$$ and denote $\logcoh{X}\subset \text{Mod}(\LogO{X})$ the full subcategory spanned by such logarithmic coherent sheaves. The following result is Proposition~\ref{proposition on finite presentation} together with Theorem~\ref{thm: FPisCoherent}.
\begin{maintheorem}\label{main theorem on log coherent}
    The category $\logcoh{X}$ is abelian. Moreover, an $\LogO{X}$-module $\mathcal{F}$ is coherent if and only if there is a conventional coherent sheaf $F$ on some logarithmic alteration $X'\rightarrow X$ such that $$\mathcal{F}\cong \;\Uparrow \! \! F.$$
\end{maintheorem}
A closely related theorem is recorded in the PhD thesis of Thompson \cite{on-toric-log-schemes}, who acknowledges substantial contributions from Gabber \cite{GabberToricFlattening} (see Section \ref{sec:origin Static}).

The central difficulty in proving Theorem \ref{main theorem on log coherent} is establishing that the category of logarithmic coherent sheaves is closed under kernels. Indeed, $\Uparrow$ does not commute with taking kernels since logarithmic alterations are not always flat. We isolate the subtleties that must be addressed in order to establish the yoga of ups and downs described in Outcome \ref{OutcomeI}.
\begin{enumerate}
    \item The up functor is neither left exact, nor full, nor faithful; see Example \ref{example: up not exact}.
    \item For $F$ a coherent sheaf, it is not a priori clear that the $\mathcal{O}_{\underline{X}}$-module $\Downarrow\!\!_X\arrowup F$ is coherent, even in the case where $F$ is the structure sheaf \cite[Question 1.7.1]{molcho2023remarks}. 
\end{enumerate}
We proceed by reducing to a class of conventional coherent sheaves which do not suffer from these issues.

\subsection{Static quasi-coherent sheaves} 
A conventional quasi-coherent sheaf on $\underline{X}$ is said to be \textit{static} if it has Tor dimension at most one over the Artin fan of $X$. See Section \ref{relative Tor dimension} for a precise definition of Tor dimension. This notion is closely related to \cite[Definition 2.1.2]{on-toric-log-schemes}, see Section \ref{sec:origin Static}. More concretely, a conventional quasi-coherent sheaf $F$ on an affine smooth toric variety $X$ is static if and only if for some (equivalently any) short exact sequence of the form $$0 \longrightarrow \mathcal{K}\longrightarrow \mathcal{O}_{\underline{X}}^{\oplus \ell}\longrightarrow F \longrightarrow 0,$$
the equations of the toric boundary divisors of $X$ form a $\mathcal{K}$-regular sequence. We denote the (not necessarily abelian) category of static quasi-coherent sheaves by $\text{QCoh}_{\text{sta}}(\underline{X})$. The next result follows from Remark~\ref{remark: up is exact on statics}, Proposition \ref{prop: Ext groups static}, Corollary~\ref{corollary-DM-stacks-static-pullback}, and Corollary~\ref{corollary-sheaf-cohomology-of-static-sheaves}.
\begin{maintheorem}\label{mainthm: Static Good}
    After restricting its domain, the up functor
    $$
        \arrowup\colon \operatorname{QCoh}_{\text{sta}}(\underline{X}) \longrightarrow \operatorname{Mod}(\LogO{X})
    $$
    becomes fully faithful and preserves exact sequences bounded on the right. Moreover, for any static conventional quasi-coherent sheaf on $\ul{X}$ and any logarithmic alteration $q\colon X'\rightarrow X$ we have $$\Downarrow_{X'} \arrowup F = \ul{q}^\ast F.$$ Finally, there is an equality of sheaf cohomology groups:
    $$H^{i}(\underline{X},F) = H^{i}(X_{\text{l\'et}}, \arrowup F).$$
\end{maintheorem}

In particular, if $X$ is affine and $F$ is a static conventional quasi-coherent sheaf on $\ul{X}$, then $H^i(X_{\text{l\'et}},\arrowup F)=0$ for all $i>0$. We view this as the first piece of evidence for the following perspective developed throughout this introduction: Staticity is a natural enhancement of affineness in logarithmic geometry.

If $X$ is of finite type over $k$, we will now detail a combinatorial characterization of staticity that allows us to reduce basic computations concerning logarithmic coherent sheaves to the static case. For this, let $F$ be a coherent sheaf on $X$. Since being static is a local property and since $X$ locally admits a strict closed immersion into an affine toric variety, we may assume that $X$ is itself affine and toric.
Recall that the \textit{Gr\"obner stratification} \cite{kennedy2023logarithmic,kennedy2025logarithmic} of a sheaf $K$ on $X$ with cocharacter lattice $M_X$ is the data of a locally closed stratification of $M_X\otimes \mathbb{Q}$. A fan structure $\Sigma$ on $M_X$ is said to \textit{refine} a \textit{Gr\"obner stratification} if each stratum is a union of interiors of cones in $\Sigma$. 
\begin{maintheorem}[Proposition~\ref{prop:computing statification}]\label{main theorem: Grobner strat}
     Let $X$ be an affine toric variety and let $F$ be a coherent sheaf on $\ul{X}$ which admits a presentation of the form $$0 \longrightarrow {K}\longrightarrow \mathcal{O}_{\ul{X}}^{\oplus \ell} \longrightarrow F \longrightarrow 0.$$
     The pullback of $F$ along a toric modification $q\colon \ul{X}'\rightarrow \ul{X}$ is static if and only if the fan of $\ul{X}'$ refines the Gr\"obner stratification of $K$.
\end{maintheorem}
In particular, Theorem \ref{main theorem: Grobner strat} implies that there is a logarithmic modification $q\colon X'\rightarrow X$ such that $\ul{q}^*F$ is static. More generally, by \cite[Corollary 3.3.5]{on-toric-log-schemes}, so long as $X$ is fine and Noetherian, any coherent sheaf can be rendered static after a logarithmic modification, but the proof in op. cit. is non-constructive.

\begin{maincorollary}\label{main corr:addressing Molcho Wise}
    There exists a logarithmic modification $$X'\longrightarrow X$$ such that the structure sheaf of $X'$ satisfies logarithmic \'etale descent. In particular, the logarithmic \'etale sheafification of $\mathcal{O}_X$ is a coherent $\mathcal{O}_X$-module.
\end{maincorollary}

Corollary \ref{main corr:addressing Molcho Wise} answers a question of Molcho and Wise \cite[Question 1.7.1]{molcho2023remarks} in the Noetherian case. The definition of static applied to the structure sheaf addresses \cite[Question 1.7.2]{molcho2023remarks}. Example \ref{example: statification of structure sheaf 1} and Example \ref{example: statification of structure sheaf 2} together resolve \cite[Question 1.7.3]{molcho2023remarks}. 

\subsection{The yoga of ups and downs} 
Before explaining the yoga of Outcome \ref{OutcomeI}, we record that derived functors preserve (quasi-)coherence in the expected way. We call an $\mathcal{O}_{X_{\text{l\'et}}}\text{-module}$ $\mathcal{F}$ \textit{quasi-coherent} if $\arrowdown_{X'}\!\mathcal{F}$
    is a conventional quasi-coherent sheaf on $\underline{X'}$ for every logarithmic alteration $q\colon X'\rightarrow X$. We write $\logqcoh{X}\subset \text{Mod}(X_{\text{l\'et}})$ for the full abelian subcategory whose objects are the logarithmic quasi-coherent sheaves. The relationship between logarithmic coherent and logarithmic quasi-coherent sheaves mirrors the conventional picture: $\logqcoh{X}$ is the ind-completion of $\logcoh{X}$, see Theorem~\ref{theorem-compacted-generateness-of-quasi-coherent}. The next result follows from Proposition \ref{prop: derived push log}, Proposition \ref{prop: pullback}, Proposition \ref{prop: ext}, and Proposition \ref{prop: tor}.

\begin{maintheorem}[Proposition \ref{prop: derived push log}]\label{main thm: (quasi)coherence preserved}
    Let $f\colon X\rightarrow Y$ be a morphism of Noetherian fine and saturated logarithmic $k$-schemes. Let $\underline{f}$ be the underlying morphism of schemes.  
    The derived functors $$L_nf_{\text{l\'et}}^*\colon \text{Mod}(\mathcal{O}_{Y_{\text{l\'et}}})\longleftrightarrows \text{Mod}(\mathcal{O}_{X_{\text{l\'et}}})\colon R^nf^\text{l\'et}_*$$ preserve logarithmic quasi-coherence. Coherence is preserved by left derived pullback and derived pushforward preserves coherence whenever $f$ is proper and $\text{char}(k)=0$. Moreover, if $\cF$ is a logarithmic coherent sheaf, the functors
    $$\cT or_n^{X_{\text{l\'et}}} (\cF,-)\colon \text{Mod}(\mathcal{O}_{X_{\text{l\'et}}})\longleftrightarrows \text{Mod}(\mathcal{O}_{X_{\text{l\'et}}})\colon \cE xt^{n}_{X_{\text{l\'et}}}(\cF,-) $$
    preserve (quasi-)coherence. Finally, $\cT or_n^{X_{\text{l\'et}}}$ preserves quasi-coherence whenever $\mathcal{F}$ is quasi-coherent. 
\end{maintheorem}

In Section \ref{sec:LogCoherentSheaves}, we detail explicit algorithms for computing these functors in the coherent case. Below, we record the most subtle example, which concerns derived pushforward:
    
\begin{maintheorem}\label{main thm: main theorem push forward}
Continuing with the notation of Theorem \ref{main thm: (quasi)coherence preserved}, we further assume that $f$ is saturated and separated. Let $F$ be a static conventional quasi-coherent sheaf such that $R^{i}\underline{f}{}_*F$ is static for all $i\geq n$ and set $\mathcal{F} = \arrowup F$. 
\begin{enumerate}
    \item\label{item: compute derived push introduction} We have $$R^{i}f^{\text{l\'et}}_*\mathcal{F}=\arrowup R^{i}\underline{f}{}_*F $$ for all $i\geq n-1$. 
    \item\label{item: derived push shit bla introduction} For any logarithmically flat morphism $q\colon Y'\rightarrow Y$ we have $$R^{i}\underline{f}'{}_* (\underline{q'})^*F = \underline{q}^*R^{i}\underline{f}{}_*F$$ for all $i\geq n-1$, where 
        \begin{equation*}
            \begin{tikzcd}
                X'\dar["q'"] \rar["f'"] &Y'\dar["q"]\\
                X\rar["f"]&Y
            \end{tikzcd}
        \end{equation*}
        is the base change of $f$ along $Y'$.
    \end{enumerate}
\end{maintheorem}

If moreover $X$ and $Y$ are of finite type over $k$, Theorem \ref{main thm: main theorem push forward} can be used to compute push-forwards as follows: Given a conventional coherent sheaf $F$, we use downward induction on $n$ and Theorem \ref{main thm: main theorem push forward} \eqref{item: derived push shit bla introduction} together with Theorem \ref{main theorem: Grobner strat} to reduce to the case where the assumptions of Theorem \ref{main thm: main theorem push forward} \eqref{item: compute derived push introduction} hold for $n$. Since $R^i\underline{f}{}_*F=0$ for $i\gg 0$, there is a base case for this induction.

\subsection{Moduli of invertible logarithmic sheaves} We now assume that $k$ is algebraically closed of characteristic $0$.
An \textit{invertible logarithmic coherent sheaf} is a rank one locally free logarithmic coherent sheaf. Let $C/S$ be a logarithmically smooth curve and let $\mathfrak{M}_1(C/S)$ be the category fibred in groupoids over the category of fine and saturated logarithmic $S$-schemes whose fiber over $T$ consists of invertible logarithmic coherent sheaves on $T\times_S C$. The moduli stack $\mathfrak{M}_1(C/S)$ is the simplest example of a moduli space of logarithmic vector bundles.

In conventional algebraic geometry, the coarse space associated to the stack of line bundles on a smooth and proper curve, called the Picard scheme, is an abelian scheme, and is in particular separated and universally closed. Following a theme suggested by Illusie, the logarithmic Picard group of a curve $\operatorname{LogPic}(C/S)$ was shown to be a logarithmic abelian variety \cite{KajiwaraKatoNakayama2008} by Molcho and Wise \cite{MolchoWise}. The logarithmic Picard group admits a modular interpretation as a subfunctor of the moduli spaces of torsors under a logarithmic replacement $\mathbb{G}_m^\mathsf{log}$ for $\mathbb{G}_m$. 

The functors of points of both $\mathfrak{M}_1(C/S)$ and $\operatorname{LogPic}(C/S)$ naturally form sheaves in the big logarithmic \'etale site over $S$, see Lemma \ref{lem:MSatisfiesDescent} and Theorem \ref{thm:MolchoWise} respectively. By contrast, the functor of points of an algebraic space with logarithmic structure almost never satisfies logarithmic \'etale descent. We define \textit{descending logarithmic algebraic spaces} as a replacement for logarithmic algebraic spaces whose functors of points do satisfy this descent property, see Definition \ref{defn:logalgspace}.

\begin{maintheorem}[Proposition \ref{prop:LogVBtoLogPic}]\label{mainthm:LogPic}
    Every invertible logarithmic coherent sheaf specifies a $\mathbb{G}_m^\mathsf{log}$-torsor. This assignment descends to define a morphism $$\varpi \colon \mathfrak{M}_1(C/S)\longrightarrow \operatorname{LogPic}(C/S)$$ which is initial among morphisms from $\mathfrak{M}_1(C/S)$ to a Noetherian descending logarithmic algebraic space over $S$.
\end{maintheorem}
The morphism $\varpi$ does not induce a bijection on standard logarithmic points, and we say that any two points identified by $\varpi$ are \textit{$\ell$-equivalent}. We consider $\ell$-equivalence to be a version of $S$-equivalence \cite{MumfordFogartyKirwan1994} in the setting of logarithmic geometry. Indeed, in the conventional case, $S$-equivalence can be detected by studying maps from the quotient stack $\Theta = [\mathbb{A}^1/\mathbb{G}_m]$ into $\mathfrak{M}_1(C/S)$. In the logarithmic setting, $\Theta$ is replaced by a particular Olsson fan \cite{VBOlssonFan} which, following a suggestion of Wise, we call \textit{the interval}. The underlying geometry of the interval over a (standard log) point consists of two copies of $[\mathbb{A}^1/\mathbb{G}_m]$ identified at the closed point.

In particular, two $T$-valued points $f_1,f_2$ of the domain of $\varpi$ are sent to the same $T$-valued point of the target precisely when there exists a map from the interval $T_I$ over $T$ that interpolates between them; see Proposition \ref{prop:TheInterval} for a precise statement. 

We expect that a generalization of the theory of good moduli spaces to the logarithmic setting will have applications to a number of related moduli problems; see, in particular, the discussion of \cite[Section 0.4]{Log-generalized-Kummer}. In this respect, the recent introduction of Olsson fans by Battistella, Carocci, and Wise appears prescient \cite{VBOlssonFan}. The phenomenon on $\ell$-equivalence can be understood at the level of tropical geometry, see Section \ref{sec:TropicalVBmeaning}, where it manifests as \textit{chip firing}.

\subsection{Origin of staticity}\label{sec:origin Static}
 Our definition of static and the proof of Theorem \ref{main theorem on log coherent} owe an intellectual debt to both Gabber's notion of t-flatness \cite[Definition~6.1.35]{gabber2018foundationsringtheory}, and \cite[Lemma~6.5]{kerz2018algebraic}. After isolating these insights, we learned through private communication with Ofer Gabber, facilitated by Adam Dauser, that staticity is a special case of \cite[Definition 2.1.2]{on-toric-log-schemes}. Similarly, our argument for exactness of the up functor is closely related to arguments used in \cite{GabberToricFlattening,on-toric-log-schemes} to establish the coherence of Kato's valuative log space.
We have been unable to identify references to these insights in any articles on arxiv or in a journal. They appear to have been lost to parts of the community.

\subsection*{Conventions}
Unless otherwise stated, all logarithmic structures and monoids in this paper are fine and saturated. Moreover, for any monoid $P$ we write $\cA_P= [\text{Spec}(k[P])/\text{Spec}(k[P^{gp}])]$ for the associated quotient stack. A logarithmic scheme $X$ is a scheme $\underline{X}$ equipped with a logarithmic structure, and the notation for a non-logarithmic scheme or stack will always include an underline. Similarly, given a DM stack $\mathcal{X}$ with logarithmic structure we write $\underline{\mathcal{X}}$ for the underlying DM stack. 
We refer the reader to standard references \cite{Logarithmic-structures-of-Fontaine-Illuie-I-Kato,lecturesonlogarithmicalgebraicgeometry,temkin2022introductionlogarithmicgeometry} for background on logarithmic geometry. We call a logarithmic \'etale morphism a \emph{cover} if it is universally surjective in the sense that it remains surjective after base change along any \emph{integral} logarithmic scheme, see \cite{m-type-topology}. This is consistent with \cite{MolchoWise} who do not restrict attention to the category of fine and saturated logarithmic schemes. We abreviate the phrase Deligne-Mumford to DM. Since being logarithmic \'etale is strict \'etale local, these notions naturally generalize to DM stacks and we write $\cX_{\text{l\'et}}$ for the small logarithmic \'etale site of a DM $k$-stack with logarithmic structure $\cX$.
We typically reserve italics $\mathcal{F}$ for logarithmic \'etale $\mathcal{O}_X$-modules and use plain calligraphy $F$ for strict \'etale $\mathcal{O}_{\underline{X}}$-modules. 
In the sequel, outside of definitions and section titles, we abbreviate the word logarithmic to log. We write CFG for the phrase category fibred in groupoids, qc for quasi-compact, and qcqs for quasi-compact and quasi-separated.

\subsection*{Acknowledgements} 
We thank Ofer Gabber for sharing his handwritten note \cite{GabberToricFlattening} and Howard Thompson's thesis \cite{on-toric-log-schemes} with us. H.D. and P.K.-H. thank Sam Johnston for collaboration in the early stages of this project. This paper forms parts of the Ph.D. thesis of X.H., who would like to thank his advisors, Christian Liedtke and Helge Ruddat, for their unwavering support.
We thank Adam Dauser, Uttaran Dutta, Jeremy Feusi, Leo Herr, Dominic Joyce, Ailsa Keating, Samouil Molcho, Thibault Poiret, Dhruv Ranganathan, David Rydh, Terry Song, Pim Spelier, Richard Thomas, and Jonathan Wise 
for numerous conversations that have shaped our understanding of the topic, and especially thank Jeremy, Dhruv, Helge, Terry, and Jonathan for comments on a first draft.
The expectation that Theorem \ref{mainthm:LogPic} should be true arose from conversations with Thibault Poiret concerning moduli of higher rank logarithmic vector bundles. The result is developed here with his permission. 
Aspects of this project have drawn on lessons learned during the HIMR Focused Research Grant collaborative visit \textit{Logarithmic Geometry and Smoothing Bridgeland Stability Conditions} attended by H.D., Uttaran Dutta, Sam Johnston, and P.K.-H.. We thank HIMR for supporting this visit. H.D.  was supported by the ERC Synergy Grant HyperK (ID 854361). X.H. was supported by DFG -- Project No.516701553 as a Ph.D student in Technical University of Munich. A part of the work was carried out during X.H.'s visit to the University of Stavanger, Norway, whose hospitality is gratefully acknowledged.
P.K.-H. was a research fellow at Peterhouse, Cambridge. In the later stages of this project P.K-H., K.M.R. and M.S. were supported as postdoctoral fellows of the Max Planck Institute for Mathematics. 
\section{Logarithmic alterations and logarithmic \'etale sheaves}\label{section-on-log-etale-topoi}

This section collects generalities on logarithmic \'etale sheaves. 

Covers in the logarithmic \'etale topology are generated by strict \'etale covers inherited from conventional algebraic geometry and \textit{logarithmic alterations}, whose definition we now discuss.

\subsection{Logarithmic alterations}\label{section: log alteration}
This section generalizes the definition of log alteration of log schemes introduced in \cite{MolchoWise,molcho2023remarks} to DM stacks. 
\begin{definition}\label{defn:RootStackKummerExt}
Let $\mathcal{X}$ be a DM $k$-stack with a log structure, and let
\[
\overline{M}_{\mathcal{X}}^{\mathrm{gp}} \subset B \subset \overline{M}_{\mathcal{X}}^{\mathrm{gp}} \otimes_{\mathbb{Z}} \mathbb{Q}
\]
be a locally finitely generated extension of sheaves of abelian groups on $\underline{\mathcal{X}}{}_{\text{\'et}}$, which we call a \emph{Kummer extension} of $\overline{M}_{\mathcal{X}}^{gp}$. Regarding $\cX$ as a CFG over the category of log schemes, we define the \textit{root stack} $\!\!\sqrt[B]{\mathcal{X}}$ associated to $B$ to be the full subcategory of $\cX$ with set of objects over a log scheme $T$ given by
$$
\Bigl\{
\, f \colon T \to \mathcal{X} \ \Big| \ 
\ul{f}^{-1}\overline{M}_{\mathcal{X}}^{\mathrm{gp}} \longrightarrow \overline{M}_T^{\mathrm{gp}}
\text{ factors through } \ul{f}^{-1}B 
\Bigr\}.
$$
Unless specified otherwise, we assume that $\!\!\sqrt[B]{\mathcal{X}}$ is \textit{invertible} meaning that every stalk of $B/\overline{M}_{\mathcal{X}}^{gp}$ is finite of order coprime to $\text{char}(k)$. In accordance with \cite{talpo2018infinite}, we write $\!\!\sqrt[n]{\cX}$ for the root stack associated to $\frac{1}{n}\overline{M}^{gp}_\cX$.
\end{definition}
\begin{proposition}\label{proposition: root stacks are a thing}
    With notations as in Definition \ref{defn:RootStackKummerExt}, the root stack $\!\!\sqrt[B]{\mathcal{X}}$ is represented by a DM $k$-stack with a log structure and the projection $q\colon \!\!\sqrt[B]{\mathcal{X}}\rightarrow \mathcal{X}$ satisfies the following properties:
    \begin{enumerate}
        \item\label{item: rooto first blood} The morphism $\ul{q}$ is proper and finitely presented.
        \item\label{item: root stack pushforward exact} The higher derived pushforwards of any quasi-coherent sheaf along $\ul{q}$ vanish.
        \item\label{item: root stack push forward structure sheaf} We have $\underline{q}{}_*\mathcal{O}_{\underline{\!\!\sqrt[B]{\mathcal{X}}}} = \mathcal{O}_{\underline{\mathcal{X}}}$.
    \end{enumerate}
\end{proposition}
\begin{proof}
All claims are strict \'etale local, so we may assume that $\mathcal{X}$ is a log scheme, in which case \eqref{item: rooto first blood}, representability, and exactness of $\ul{q}{}_*$ on QCoh follow from \cite[Proposition 4.19]{Parabolic-sheaves-on-logarithmic-schemes}. \eqref{item: root stack push forward structure sheaf} follows from the fact that $\ul{\mathcal{X}}$ is the coarse moduli of $\ul{\!\!\sqrt[B]{\mathcal{X}}}$ (see the discussion after \cite[Remark 3.1]{talpo2018infinite}). Note that \eqref{item: root stack pushforward exact} is in general not equivalent to $\ul{q}{}_*$ being exact on quasi-coherent sheaves, but in our case this holds by \cite[Remark 3.5]{Alper2013GoodModuli} and the fact that $\ul{q}$ has affine diagonal.
\end{proof}
\begin{definition}\label{definition-log-modification}
    A morphism $f\colon\cX\rightarrow\cY$ of DM $k$-stacks with log structures is a \textit{logarithmic modification}\footnote{The term \emph{expansion} is used throughout the literature, and in the introduction, for  underlying morphisms of schemes associated to certain log modifications.} if $f$ is strict \'etale locally on $\cY$ pulled back from a toric modification (i.e. a proper, birational, and torus equivariant map of toric varieties). More generally, we call $f$ a \textit{logarithmic alteration} if it is a composite of logarithmic modifications and root stacks.
\end{definition}

\begin{example}
    Consider the morphism $C = \mathbb{A}^2_{u,v}\rightarrow \mathbb{A}^1_z = S$ defined by $(u,v) \mapsto uv$, and write $C_0$ for the (not proper) nodal curve which is the preimage of $0$. Note that $C_0$ consists of two components. Log \'etale covers of $C_0$ include the following.
    \begin{enumerate}
        \item The blowup $C'$ of $C$ at $0$ in $\mathbb{A}^2$ is a log modification of $C$, and thus the fibre product $C_1 = C_0 \times_C C'\rightarrow C_0$ is also a log modification. On the level of underlying schemes $C_0$ has three components $C^1,C^2,C^3$ arranged in a chain: two are reduced copies of $\mathbb{A}^1$, but the central (exceptional) component $C^2$ is non-reduced.
        \item We now demonstrate that, by working log \'etale locally, non-reduced scheme structure can be avoided. Indeed, replacing $S$ by the (Kummer) log \'etale cover $z\mapsto z^2$ and $C_1$ by its base change $\hat{C}_1$ yields a reducible and reduced rational curve with three components. 
        \item The root stack $\sqrt[2]{C_1}$ is a log alteration of $C_1$ with three (not all reduced) components. The gerbe structure at every point is $\mu_2$.
    \end{enumerate}
    We recursively define further log alterations $C_j\rightarrow C_{j-1}\rightarrow C$. Note that Kummer log \'etale locally about each node $C_j\rightarrow S$ is isomorphic to $C\rightarrow S$. To form $C_{j}$ replace each node in $C_{j-1}$ by the log alteration locally isomorphic to $C_1\rightarrow C$. Consequently, the underlying scheme of $C_j$ has $2^j-1$ components, most of which have non-reduced scheme structure. By toric semistable reduction \cite{AbramovichKaru,molcho2019universal}, after replacing $S$ by a Kummer log \'etale cover and $C$ by its base change, we may assume that $C$ is reduced.
\end{example}

We record the following well-known facts for later use.

\begin{proposition}\label{proposition: basic properties of log alterations}
    Let $q\colon \cX'\rightarrow \cX$ be a log alteration (resp. log modification, root stack).
    \begin{enumerate}
        \item\label{item: log alts stable under base change} Any base change of $q$ is a log alteration (resp. log modification, root stack).
        \item\label{item: log alts is a mono log et cover} The morphism $q$ is a monomorphism and a log \'etale cover and hence induces an isomorphism on sections of any log \'etale sheaf.
        \item\label{item: log alts are proper} The associated morphism of schemes $\ul{q}$ is proper and finitely presented.
        \item\label{item: cancellation property log alterations} If $q'\colon \cX''\rightarrow\cX'$ is a morphism and $q\circ q'$ is a log alteration (resp. log modification, root stack), then $q'$ is also a log alteration (resp. log modification, root stack).
    \end{enumerate}
\end{proposition}
\begin{proof}
    Proofs of claims \eqref{item: log alts stable under base change} - \eqref{item: log alts are proper} are omitted (see \cite{Parabolic-sheaves-on-logarithmic-schemes,MolchoWise} for further discussion). Claim \eqref{item: cancellation property log alterations} follows from \eqref{item: log alts stable under base change} and the fact that $q'$ is a base change of $q\circ q'$ along $q$, using that log alterations are monomorphisms.
\end{proof}
\subsection{Logarithmic \'etale sheaves}\label{subsection ob log etale ringed topoi}
In this section, we study an appropriate notion of log \'etale $\cO$-modules and express them in terms of conventional \'etale $\cO$-modules on log alterations. Translating computations involving log \'etale $\mathcal{O}$-modules to computations involving only conventional $\mathcal{O}$-modules is a theme of this paper.
\begin{definition}\label{definition of log structure sheaf}
Let $\mathcal{X}$ be a DM $k$-stack with a log structure. We define the \textit{logarithmic \'etale structure sheaf} $\LogO{\cX}$ as the sheaf on $\cX_{\text{l\'et}}$ associated to the presheaf $\LogO{\cX}^{\text{pre}}$ defined by $\Gamma(\cU\rightarrow \cX,\LogO{\cX}^{\text{pre}}) = \Gamma(\ul{\cU}, \cO_{\ul{\cU}})$. 
Let $\cX_{\text{l\'et}}^r\coloneqq (\text{Sh}(\mathcal{X}_{\text{l\'et}}),\mathcal{O}_{\mathcal{X}_{\text{l\'et}}})$ be the resulting ringed topos. For any morphism $f\colon \cX\rightarrow\cY$, we denote by $f_{\text{l\'et}}\colon\rt{\cX}\rightarrow\rt{\cY}$ the corresponding morphism of ringed topoi.
\end{definition}
\begin{remark}\label{remark: first instance of Max's favorite counterexample}
    \begin{enumerate}
        \item To our knowledge, the above definition of structure sheaf was first recorded by Niziol \cite[Section 2.2]{K-theory-of-log-schemes-I}.
        \item\label{item: actual counterexample first time} The sheafification step is necessary as \cite[Remark p.~671]{Log-etale-cohomology-Chikara-Nakayama} and \cite[Example 6.5]{molcho2023remarks} detail an example of a log modification $q\colon X'\rightarrow X$ for which the associated $\ul{q}^\sharp\colon\Gamma(\ul{X},\cO_{\ul{X}})\rightarrow \Gamma(\ul{X}',\cO_{\ul{X}'})$ is not injective, despite the fact that $q$ is a log \'etale cover.
    \end{enumerate}
\end{remark}

The bridge between conventional $\mathcal{O}$-modules and log \'etale $\mathcal{O}$-modules is provided by the adjoint pair of functors introduced in Construction \ref{construction: up and down}. Proposition \ref{prop: explicit description of up and down} will provide an alternative way to think about of these functors.
\begin{construction}\label{construction: up and down}
Given a log alteration $q\colon \mathcal{X'}\rightarrow \mathcal{X}$, consider the composite morphism
$$\mathsf{rest}_{\cX,\cX'}\colon \rt{\cX}\xrightarrow{q_{\text{l\'et}}^{-1}} \rt{(\cX')} \xrightarrow{\pi_{\text{l\'et}}} \rt{(\ul{\cX}_{\text{triv}}')} = (\ul{\cX}')_{\text{\'et}}^r$$
to the conventional ringed \'etale topos of $\ul{\cX}'$, where $\pi\colon \cX'\rightarrow \ul{\cX}'_{\text{triv}}$ is the canonical morphism to the underlying DM stack equipped with the trivial log structure. Note here that $q_{\text{l\'et}}$ is an equivalence of ringed topoi as $q$ is a monomorphic log \'etale cover.

Associated to any morphism of ringed topoi are an adjoint pair of pullback and pushforward functors \cite[03D5]{stacks}, which we call \textit{up} and \textit{down} functor respectively:
$$\Uparrow \coloneqq \mathsf{rest}_{\cX,\cX'}^*\colon \operatorname{Mod}(\mathcal{O}_{\underline{\cX}'_\text{\'et}})\longleftrightarrows \operatorname{Mod}(\LogO{\cX})\colon (\mathsf{rest}_{\cX,\cX'})_* \eqqcolon \Downarrow_{\cX'}.$$

\end{construction}

\begin{remark}
    Using the counterexample $q\colon X'\rightarrow X$ of Remark \ref{remark: first instance of Max's favorite counterexample} \eqref{item: actual counterexample first time}, we see that $\Uparrow$ is not always faithful as $$\Uparrow\colon \text{Hom}_{\cO_{\ul{X}_{\text{\'et}}}}(\cO_{\ul{X}},\cO_{\ul{X}}) = \Gamma(\ul{X},\cO_{\ul{X}})\longrightarrow \Gamma(X,\LogO{X}) = \text{Hom}_{\LogO{\cX}}(\LogO{X},\LogO{X})$$ factors through $\Gamma(\ul{X},\cO_{\ul{X}})\rightarrow \Gamma(\ul{X}',\cO_{\ul{X}'})$, which is not injective.
\end{remark}
\begin{proposition}\label{prop: explicit description of up and down}
Let $q\colon \mathcal{X'}\rightarrow \mathcal{X}$ be a log alteration. Then for any sheaf $\mathcal{F}$ on $\mathcal{X}_{\text{l\'et}}$ and strict \'etale neighborhood $\cU\rightarrow \mathcal{X'}$, we have 
\[\Gamma(\ul{\cU}\rightarrow \ul{\mathcal{X}}',\arrowdown_{\mathcal{X'}}\mathcal{F})=\Gamma(\cU\rightarrow \mathcal{X'}\rightarrow \mathcal{X},\mathcal{F}).\]

Moreover, for any $\mathcal{O}_{\ul{\mathcal{X}}'_{\text{\'et}}}$-module $F$, we have that $\Uparrow \!\! F$ is the sheafification of 
\[(\varphi\colon \cV\rightarrow\mathcal{X})\longmapsto \Gamma(\ul{\cV\times_\mathcal{X} \mathcal{X'}}, (\ul{\varphi}{}'_{\text{\'et}})^*{} F),\]
where $\varphi'$ denotes the base change of $\varphi$ along $q\colon\mathcal{X'}\rightarrow \mathcal{X}$.
\end{proposition}
\begin{remark}\label{remark: down arrow on presheaves}
    Note that $\mathsf{rest}_{\cX,\cX'}$ does not arise from a morphism of sites and hence $(\mathsf{rest}_{\cX,\cX'})_*$ is not a priori defined on the level of presheaves. However, there is still a functor $\Downarrow_{\cX'}\colon \text{PreSh}(\cX_{\text{l\'et}})\rightarrow\text{PreSh}(\ul{\cX}{}'_{\text{\'et}})$ defined by setting $\Gamma(\ul{\cU}\rightarrow \ul{\mathcal{X}}',\arrowdown_{\mathcal{X'}}\mathcal{F})\coloneqq \Gamma(\cU\rightarrow \mathcal{X'}\rightarrow \mathcal{X},\mathcal{F})$ for any presheaf $\cF$.
\end{remark}
\begin{proof}[Proof of Proposition \ref{prop: explicit description of up and down}]
Using the notation of Construction \ref{construction: up and down}, we have $\arrowdown_{\cX'}=\pi^{\text{l\'et}}_*\circ (q_{\text{l\'et}}^{-1})_* = \pi^{\text{l\'et}}_*\circ q_{\text{l\'et}}^*$ and $\arrowup=(q_{\text{l\'et}}^{-1})^*\circ \pi_{\text{l\'et}}^* = q^{\text{l\'et}}_*\circ \pi_{\text{l\'et}}^*$.

For the first claim, it follows from the construction of $q_{\text{l\'et}}^{*}$ that $\Gamma(\cU\rightarrow \cX',q_{\text{l\'et}}^{*}\cF) = \Gamma(\cU\rightarrow \cX'\rightarrow \cX,\mathcal{F})$, as the right hand side already defines a sheaf of $\cO_{\cX'_{\text{l\'et}}}$-modules. 

The second claim reduces to the fact that $\pi_{\text{l\'et}}^{*}F$ is the sheafification of the presheaf 
\[ (\varphi\colon \cV\rightarrow\cX')\longmapsto \Gamma(\ul{\cV},\ul{\varphi}{}_{\text{\'et}}^{*} F),\]
the proof of which we omit. See \cite[070S (2)]{stacks} for a similar statement.
\end{proof} 
\subsection{Properties of log \'etale sheaves} The purpose of this section is to establish Proposition \ref{key proposition} and Corollary \ref{Key corollary}; on first reading, little will be missed in skipping to these results. 

Our first proposition uses \cite{m-type-topology} to record a generalization of well-known results, see \cite[Lemma 3.3]{herr2025loggeometryliftingrational}, \cite[Theorem 1.1, Corollary 1.2]{kato2022integral}, and \cite[Theorem 1.01]{molcho2019universal}, to the setting where neither $\mathcal{Y}$ nor the morphism has a quasi-compactness hypothesis.
\begin{proposition}\label{proposition-alteration-to-saturated-morphisms}
Let $f\colon \cX\rightarrow \cY$ be a morphism of DM $k$-stacks with log structures so that $\cX$ is qc. Then
\begin{enumerate}
    \item\label{item: strictify log etale via log alteration} If $f$ is log \'etale, then there is a log alteration $\cY'\rightarrow \cY$ so that the base change $\cX\times_\cY \cY'\rightarrow \cY'$ is strict and conventional \'etale.
    \item If $\text{char}(k)=0$, then there is a log alteration $\cY'\rightarrow \cY$ so that the base change $\cX\times_\cY \cY'\rightarrow \cY'$ is saturated.
\end{enumerate}
\end{proposition}
\begin{proof}
We may assume that $f$ is integral, since by \cite{m-type-topology}, there is a logarithmic blow-up $\cY'\rightarrow\cY$ so that $\cX\times_\cY \cY'\rightarrow\cY'$ is integral. We will now show that there is some $n>0$ coprime to $\text{char}(k)$ so that $f$ becomes saturated under base change along $\sqrt[n]{\cY}\rightarrow\cY$. Indeed, we may replace $\cY$ with any qc open substack containing the image of $f$, which we can in turn replace with any strict \'etale cover by a log scheme with a global chart. It now follows from \cite[Theorem II.3.4]{saturatedTsuji} that we can choose $n$ so that the resulting base change is saturated. In the case that $f$ is log \'etale, the proof of \cite[Theorem II.3.4]{saturatedTsuji} combined with the chart criterion for log \'etale maps shows that $n$ can be chosen to be invertible. Now note that any saturated log \'etale morphism is strict since such a morphism must be Kummer, in which case the strictness follows from \cite[Corollary I.4.3.10, Theorem III.2.5.5]{lecturesonlogarithmicalgebraicgeometry}.
\end{proof}

\begin{remark}\label{remark: something is fucky about mixed char}
    The reason why we restrict ourselves to stacks over $k$ is that Proposition \ref{proposition-alteration-to-saturated-morphisms} fails in mixed characteristic. Indeed, note that $\text{Spec}(\mathbb{N}\xrightarrow{0}\mathbb{Z})$ does not admit any nontrivial log alterations.
\end{remark}

\begin{proposition}\label{proposition on log etale topos}
Let $\mathcal{X}$ be a DM $k$-stack with a log structure. Then the morphisms $\mathsf{rest}_{\cX,-}$ of Construction \ref{construction: up and down} give an equivalence of ringed topoi
\[\cX^{r}_{\text{l\'et}}=\underset{\cX'\rightarrow \cX\text{ log alt}}{\mathrm{lim}}(\ul{\cX}')^{r}_{\text{\'et}},\]
where the limit runs over all log alterations of $\cX$.
\end{proposition}
\begin{proof}
    Follows from Proposition \ref{proposition-alteration-to-saturated-morphisms} \eqref{item: strictify log etale via log alteration} and an analagous argument to \cite[Proposition 5.6~(1)]{Log-etale-cohomology-Chikara-Nakayama}.
\end{proof}
\begin{remark}
\begin{enumerate}
    \item Note that any two log alterations of $\cX$ can be dominated by a common refinement, hence the above limit is cofiltered.
    \item See \cite[Lemma 4.1]{molcho2023remarks} and \cite[Proposition 5.6]{Log-etale-cohomology-Chikara-Nakayama} for similar statements.
\end{enumerate}
\end{remark}
\begin{proposition}\label{key proposition}
Let $f\colon \cX\rightarrow \cY$ be a morphism of DM $k$-stacks with log structures. 
\begin{enumerate}
    \item \label{item-1-key proposition} If $f$ is qcqs, and $\mathcal{F}$ is a presheaf of $\mathcal{O}_{\cX_{\text{l\'et}}}^{\text{pre}}$-modules, then we have 
    \begin{equation}\label{equation: fickfickfick} R^{n}f^{\text{l\'et}}_{*}(\mathcal{F}^{\text{sh}})=\underset{\text{squares as in (}\ref{key-proposition-equation-1}\text{)}}{\mathrm{colim}}\arrowup R^{n}(\underline{f}{}'_{\text{\'et}})_{*}(\arrowdown_{\cX'}\mathcal{F})^{\text{sh}}\text{ for all }n\geq 0,\end{equation} 
    where $(-)^{\text{sh}}$ denotes sheafification, $\Downarrow_{\cX'}$ is defined using Remark \ref{remark: down arrow on presheaves}, and the colimit runs over all not necessarily cartesian squares
\begin{equation}\label{key-proposition-equation-1}
\begin{tikzcd}
\cX' \arrow[r, "f'"] \arrow[d, "p'"] & \cY' \arrow[d, "p"] \\
\cX \arrow[r, "f"] & \cY,
\end{tikzcd}      
\end{equation}
 where the vertical maps $p$ and $p'$ are log alterations.
\item \label{item-2-key proposition}
For any $\mathcal{O}_{\underline{\cY}{}_{\text{\'et}}}$-module $F$, we have
\[L_{n}f_{\text{l\'et}}^{*}\arrowup F=\underset{\text{squares as in (}\ref{key-proposition-equation-1}\text{)}}{\mathrm{colim}}\arrowup L_{n}(\underline{f}{}'_{\text{\'et}})^{*}\underline{p}{}_{\text{\'et}}^{*} F.\]
\end{enumerate} 
\end{proposition}
\begin{proof} 
By Proposition \ref{proposition on log etale topos}, both $\cX_{\text{l\'et}}^{r}$ and $\cY_{\text{l\'et}}^{r}$ are limits of conventional \'etale topoi and the two index diagrams admit a common refinement given by the diagram of all squares of the form \eqref{key-proposition-equation-1}. In particular, $f_{\text{l\'et}}$ arises as a limit of morphisms along this diagram. 

If $f$ is qcqs, then the $\ul{f}'$ and the log alterations are also qcqs, which implies that derived pushforward along them commutes with filtered colimits \cite[0GQW]{stacks}. Thus, we may apply \cite[Theorem VI.8.7.3]{SGA4}, which asserts that the right hand side of \eqref{equation: fickfickfick} is equal to $R^n f^{\text{l\'et}}_*\cF'$, where 
\begin{equation}\label{equation: first time sh is written differently}
\cF' = \underset{\text{squares as in (}\ref{key-proposition-equation-1}\text{)}}{\mathrm{colim}}\Uparrow\!\! (\arrowdown_{\cX'}\mathcal{F})^{\text{sh}} = \underset{\cX'\rightarrow \cX\text{ log alt}}{\mathrm{colim}}\Uparrow\!\! (\arrowdown_{\cX'}\mathcal{F})^{\text{sh}}.
\end{equation} 
To show claim \eqref{item-1-key proposition}, it thus suffices to prove that $\cF' = \cF^{\text{sh}}$. Indeed, for any $\LogO{\cX}$-module $\cG$, the datum of a morphism $\cF'\rightarrow \cG$ is equivalent to a datum of compatible morphisms $\Downarrow_{\cX'}\cF\rightarrow \Downarrow_{\cX'}\cG$ for every log alteration $\cX'\rightarrow\cX$, which in turn is equivalent to giving a morphism $\cF|_{\cC^{\text{op}}}\rightarrow\cG|_{\cC^{\text{op}}}$ on the full subcategory $\cC\subset \cX_{\text{l\'et}}$ spanned by the strict \'etale neighborhoods of all log alterations. Since $\cG$ is a log \'etale sheaf, it can be recovered from $\cG|_{\cC^{\text{op}}}$, hence the above is equivalent to a morphism of presheaves of $\LogO{\cX}^{\text{pre}}$-modules $\cF\rightarrow\cG$, which verifies that $\cF'$ has the universal property of $\cF^{\text{sh}}$.

To prove claim \eqref{item-2-key proposition}, we note that an argument analogous to the above paragraph also shows that 
\begin{equation}\label{equation: second time sh is written differently}
    \cF^{\text{sh}} = \underset{\text{squares as in (}\ref{key-proposition-equation-1}\text{)}}{\mathrm{colim}}\mathsf{rest}_{\cX,\cX'}^{-1} (\arrowdown_{\cX'}\mathcal{F})^{\text{sh}}
\end{equation}
for any presheaf on $\cX_{\text{l\'et}}$ and similarly for $\mathcal{Y}$.
In particular, for any $\cO_{\ul{\cY}{}_{\text{\'et}}}$-module $F$, we get
\begin{align*}
    L_{n}f_{\text{l\'et}}^{*}\arrowup F&=\mathcal{T}or_{n}^{f^{-1}_{\text{l\'et}}\mathcal{O}_{\cY_{\text{l\'et}}}}(\mathcal{O}_{\cX_{\text{l\'et}}},f^{-1}_{\text{l\'et}}\arrowup F)\\
    &=\mathcal{T}or_n^{\underset{\text{squares as in (}\ref{key-proposition-equation-1}\text{)}}{\mathrm{colim}}f^{-1}_{\text{l\'et}}\mathsf{rest}_{\cY,\cY'}^{-1}\cO_{\ul{\cY}{}'_{\text{\'et}}}}\left(\underset{\text{squares as in (}\ref{key-proposition-equation-1}\text{)}}{\mathrm{colim}}\mathsf{rest}_{\cX,\cX'}^{-1}\cO_{\ul{\cX}{}'_{\text{\'et}}},\underset{\text{squares as in (}\ref{key-proposition-equation-1}\text{)}}{\mathrm{colim}}f^{-1}_{\text{l\'et}}\mathsf{rest}_{\cY,\cY'}^{-1} \ul{p}{}_{\text{\'et}}^* F \right)\\
    &=\underset{\text{squares as in (}\ref{key-proposition-equation-1}\text{)}}{\mathrm{colim}}\mathcal{T}or_n^{\mathsf{rest}_{\cX,\cX'}^{-1} (\ul{f}{}'_{\text{\'et}})^{-1}\cO_{\ul{\cY}{}'_{\text{\'et}}}}\left(\mathsf{rest}_{\cX,\cX'}^{-1}\cO_{\ul{\cX}{}'_{\text{\'et}}},\mathsf{rest}_{\cX,\cX'}^{-1} (\ul{f}{}'_{\text{\'et}})^{-1}\ul{p}{}_{\text{\'et}}^* F\right)\\
    &=\underset{\text{squares as in (}\ref{key-proposition-equation-1}\text{)}}{\mathrm{colim}} \mathsf{rest}_{\cX,\cX'}^{-1} L_n (\ul{f}{}'_{\text{\'et}})^*\ul{p}{}_{\text{\'et}}^* F \\
    &=\underset{\text{squares as in (}\ref{key-proposition-equation-1}\text{)}}{\mathrm{colim}} \arrowup L_n (\ul{f}{}'_{\text{\'et}})^*\ul{p}{}_{\text{\'et}}^* F,
\end{align*}
where we used the fact that $f_{\text{l\'et}}^{-1}$ is a left adjoint and hence commutes with all colimits in the second equality. For the third equality we used that any square as in \eqref{key-proposition-equation-1} gives a commutative diagram
\[
    \begin{tikzcd}
        \cX_{\text{l\'et}} \dar["\mathsf{rest}_{\cX,\cX'}"']\rar["f_{\text{l\'et}}"]&\cY_{\text{\'et}}\dar["\mathsf{rest}_{\cY,\cY'}"]\\
        \ul{\cX}{}'_{\text{\'et}} \rar["\ul{f}{}'_{\text{\'et}}"] & \ul{\cY}{}'_{\text{\'et}}
    \end{tikzcd}
\]
and we further used that $\cT or$ commutes with filtered colimits, which follows from the fact that the derived tensor product $\cM\otimes^{\mathbb{L}}_{\cO}\cN$ is a homotopy colimit of objects of the form $\cM\otimes^{\mathbb{L}}_{\ul{\mathbb{Z}}}\cO\otimes^{\mathbb{L}}_{\ul{\mathbb{Z}}}\ldots \otimes^{\mathbb{L}}_{\ul{\mathbb{Z}}}\cO\otimes^{\mathbb{L}}_{\ul{\mathbb{Z}}}\cN$ via the bar construction (see e.g. \cite[§3.4.4]{HA}). The fourth equality uses that $\mathsf{rest}_{\cX,\cX'}^{-1}$ is exact and by \cite[07A4]{stacks} commutes with $\otimes^{\mathbb{L}}$. The last equality uses \eqref{equation: first time sh is written differently} and \eqref{equation: second time sh is written differently}.
\end{proof}
\begin{corollary}\label{Key corollary}
Let $\cX$ be a DM $k$-stack with a log structure, the following statements hold:
\begin{enumerate}
\item\label{item-1-key corollary} For any presheaf of $\mathcal{O}_{\cX_{\text{l\'et}}}^{\text{pre}}$-modules $\mathcal{F}$, we have
    \[\mathcal{F}^{\text{sh}}=\underset{\cX'\rightarrow \cX\text{ log alt}}{\mathrm{colim}}\Uparrow\!\! (\arrowdown_{\cX'}\mathcal{F})^{\text{sh}}.\]
\item\label{item-2-key-corollary} For any presheaf of $\mathcal{O}_{\cX_{\text{l\'et}}}^{\text{pre}}$-modules $\mathcal{F}$ and a log \'etale morphism $\cU\rightarrow \cX$ with $\cU$ qcqs, we have for all $n\geq 0$,
\[H^{n}(\cU,\mathcal{F}^{\text{sh}})=\underset{\cX'\rightarrow \cX\text{ log alt}}{\mathrm{colim}}H^{n}(\underline{\cU\times_\cX \cX'},\mathcal{F}|_{\underline{\cU\times_\cX \cX'}}^{\text{sh}}).\]
\item\label{item-3-key-corollary} For any presheaf of $\mathcal{O}_{\cX}^{\text{pre}}$-modules $\cF$ and log alteration $\cX'\rightarrow \cX$, we have for all $n\geq 0$, 
\[R^n\arrowdown_{\cX'}\cF^{\text{sh}}=\underset{q\colon \cX''\rightarrow \cX'\text{ log alt}}{\mathrm{colim}}R^{n}\underline{q}{}^{\text{\'et}}_{*}(\arrowdown_{\cX''}\mathcal{F})^{\text{sh}}.\]
\item\label{item-4-key-corollary} For any $\mathcal{O}_{\underline{\cX}{}_{\text{\'et}}}$-module $F$, we have for all $n\geq 0$,
\[L_n\arrowup F=\underset{q\colon \cX'\rightarrow \cX\text{ log alt}}{\mathrm{colim}}\arrowup L_{n}\underline{q
}{}_{\text{\'et}}^{*}F.\]
\end{enumerate}
\end{corollary}
\begin{proof}
Claim \eqref{item-1-key corollary} is the same as equation \eqref{equation: first time sh is written differently}, but also follows from Proposition \ref{key proposition} \eqref{item-1-key proposition} by taking $f=\text{id}_\cX$ and $n=0$. To show claim \eqref{item-2-key-corollary}, we may by Proposition \ref{proposition-alteration-to-saturated-morphisms} assume that $\cU\rightarrow \cX$ is strict \'etale, in which case the claim follows from claim \eqref{item-1-key corollary}, Proposition \ref{proposition on log etale topos} and \cite[Corollaire VI.8.7.7]{SGA4}. Claims \eqref{item-3-key-corollary} and \eqref{item-4-key-corollary} follow from applying Proposition \ref{key proposition} \eqref{item-1-key proposition} and \eqref{item-2-key proposition} respectively to the canonical morphism $\cX\rightarrow \underline{\cX}_{\text{triv}}$ to the underlying DM stack equipped with the trivial log structure.
\end{proof}
\section{Static quasi-coherent sheaves}
In this section we introduce staticity, and more generally log Tor dimension, using the framework of Olsson's stack of log structures \cite{log-geometry-and-algebraic-stacks}.
\subsection{Relative Tor dimension for algebraic stacks}\label{relative Tor dimension}
Our notion of log Tor dimension in the next section is built on the concept of relative Tor dimension for morphisms of algebraic stacks. Here we collect basic definitions and properties of the latter notion since we were not able to find an appropriate reference.
\begin{definition}\label{Definition-of-relative-Tor dimension}
Let $\underline{f}\colon \underline{X}\rightarrow \underline{Y}$ be a morphism of schemes and $F$ a quasi-coherent sheaf on $\underline{X}$. We say that $F$ has \emph{Tor dimension} $\leq d$ over $\underline{Y}$ for some integer $d\geq 0$ if for every $x\in \underline{X}$, we have
\[
\operatorname{Tor}^{\mathcal{O}_{\underline{Y},\ul{f}(x)}}_i\!\bigl(F_x,M) = 0 \qquad \text{for all } i > d,
\]
where $M$ is any $\mathcal{O}_{\underline{Y},\ul{f}(x)}$-module. We say the morphism $f$ has \textit{Tor dimension $\leq d$} if $F= \mathcal{O}_X$ has tor dimension $\leq d$ in the sense defined above.
\end{definition}
\begin{remark}
    By \cite[fd Lemma 4.1.10]{WeibelHomAlg}, it is enough to check this vanishing only for $i=d+1$.
\end{remark}
In order to generalize this definition to algebraic stacks, we need the following result.
\begin{lemma}\label{lemma: nice properties of rel Tor dim}
    Let $\underline{f}\colon \underline{X}\rightarrow \underline{Y}$ be a morphism of schemes, $F$ a quasi-coherent sheaf on $\underline{X}$ and $d\in\mathbb{Z}_{\geq 0}$.
\begin{enumerate}
    \item\label{item: rel Tor stacks cover target} Let $\underline{g}: \underline{Y}'\rightarrow \underline{Y}$ be flat. If $F$ is of Tor dimension $\leq d$ over $\underline{Y}$, then so is the pullback of $F$ along $\underline{X}\times_{\underline{Y}} \underline{Y}'\rightarrow \underline{X}$ over $\underline{Y}'$. The converse implication holds if $\underline{g}$ is also surjective.
    \item\label{item: rel Tor stacks cover source} Let $\underline{g}\colon \underline{X}'\rightarrow \underline{X}$ be flat. If $F$ is of Tor dimension $\leq d$ over $\underline{Y}$, then so is $\ul{g}^*F$ over $\underline{Y}$. If $\underline{g}$ is also surjective, then the converse holds too. 
    \item\label{item: rel Tor stacks etale} Let $\underline{g}\colon \underline{Y}\rightarrow \underline{Z}$ be of Tor dimension $\leq e$. If $F$ has Tor dimension $\leq d$ over $\underline{Y}$, then $F$ has Tor dimension $\leq d+e$ over $\underline{Z}$. If $\underline{g}$ is \'etale and $F$ has Tor dimension $\leq d$ over $\underline{Z}$, then it also has Tor dimension $\leq d$ over $\underline{Y}$.
\end{enumerate}
\end{lemma}
\begin{proof}
    Since all of these claims may be checked stalkwise, this reduces to commutative algebra. More specifically, \eqref{item: rel Tor stacks cover target} follows from \cite[066M]{stacks} and \cite[0B67]{stacks}, \eqref{item: rel Tor stacks cover source} from \cite[0DJF]{stacks} and the first part of \eqref{item: rel Tor stacks etale} follows from \cite[066K]{stacks}. For all of this also recall that a local morphism of local rings is faithfully flat if and only if it flat. 
    For the second part of \eqref{item: rel Tor stacks etale}, it follows from \eqref{item: rel Tor stacks cover target} that $G=(\underline{X}\times_{\underline{Z}} \underline{Y}\rightarrow \underline{X})^{*}F$ has Tor dimension $\leq d$ over $\underline{Y}$. Note that $(\underline{\text{Id}},\underline{f})\colon \underline{X}\rightarrow \underline{X}\times_{\underline{Z}} \underline{Y}$ is an open immersion as it is a base change of the diagonal $\Delta_{\ul{g}}\colon \underline{Y}\rightarrow \underline{Y}\times_{\underline{Z}} \underline{Y}$, which is an open immersion as $\underline{g}$ is \'etale. It thus follows from \eqref{item: rel Tor stacks cover source} that $F=(\underline{\text{Id}},\underline{f})^*G$ has Tor dimension $\leq d$ over $\underline{Y}$.
\end{proof}
\begin{definition}
    Let $f\colon \underline{\mathcal{X}}\rightarrow \underline{\mathcal{Y}}$ be a morphism of algebraic stacks and $F$ a quasi-coherent sheaf on $\underline{\mathcal{X}}$. Then we say that $F$ is of \textit{Tor dimension} $\leq d$ over $\underline{\mathcal{Y}}$ if there is a commutative diagram 
    \begin{equation*}
    \begin{tikzcd}
        \underline{U} \rar\dar[two heads, "\underline{\pi}"] & \underline{V}\dar\\ \underline{\mathcal{X}} \rar["\underline{f}"] & \underline{\mathcal{Y}}
    \end{tikzcd}
    \end{equation*}
    where $\underline{U}$ and $\underline{V}$ are schemes, the vertical morphisms are flat, $\underline{\pi}$ is surjective and $\underline{\pi}^*F$ is of Tor dimension $\leq d$ over $\underline{V}$.
\end{definition}
It is straightforward to check that Lemma \ref{lemma: nice properties of rel Tor dim} also holds more generally for algebraic stacks. The following Proposition gives us a more concrete handle on relative Tor dimension:
\begin{proposition}\label{prop: stacky Tor dimension in terms of qcoh sheaves}
    Let $\underline{f}\colon \underline{\mathcal{X}}\rightarrow \underline{\mathcal{Y}}$ be a morphism of algebraic stacks and $F$ a quasi-coherent sheaf on $\underline{\mathcal{X}}$. If $F$ has Tor dimension $\leq d$ over $\underline{\mathcal{Y}}$, then for every quasi-coherent sheaf $G$ on $\underline{\mathcal{Y}}$ we have $\mathcal{H}^{-i}(F\otimes_{\mathcal{O}_{\underline{\mathcal{X}}}}^\mathbb{L} L\underline{f}^*G) = 0$ for all $i>d$. \\
    Moreover, if $\underline{\mathcal{Y}}$ is quasi-separated and admits an open cover by substacks whose diagonal is affine\footnote{For example, this holds if $\ul{\cY}$ is a quasi-separated scheme or an algebraic stack with affine diagonal.}, then the converse implication holds and it suffices to check vanishing only for $i=d+1$.
\end{proposition}
\begin{proof}
    Let us first assume that $\ul{\mathcal{Y}}$ is a scheme. Since both the cohomology vanishing and Tor dimension are faithfully flat local in $\ul{\mathcal{X}}$, we may also assume that $\ul{\mathcal{X}}$ is a scheme. Then for any $x\in \ul{\mathcal{X}}$ and $F$ and $G$ as above we have 
    $$
    \mathcal{H}^{-i}(F\otimes_{\mathcal{O}_{\ul{\mathcal{X}}}}^\mathbb{L} L\ul{f}^*G)_x = \mathcal{H}^{-i}(F\otimes_{\ul{f}^{-1}\mathcal{O}_\mathcal{Y}}^\mathbb{L} \ul{f}^{-1} G)_x = \text{Tor}_i^{\mathcal{O}_{\ul{\mathcal{Y}},\ul{f}(x)}}(F_x,G_{\ul{f}(x)}).$$
    This implies the first claim. For the second, it suffices to show that any $\mathcal{O}_{\ul{\mathcal{Y}},\ul{f}(x)}$-module $N$ can be realized as the stalk $N=G_{\ul{f}(x)}$ of some quasi-coherent $G$. Indeed, let $i\colon \ul{U}\hookrightarrow \ul{\mathcal{Y}}$ be an affine neighborhood of $\underline{f}(x)$. Then $N$ can be regarded as an $\mathcal{O}_{\ul{\mathcal{Y}}}(\ul{U})$-module via $\mathcal{O}_{\ul{\mathcal{Y}}}(\ul{U})\rightarrow \mathcal{O}_{\ul{\mathcal{Y}},\ul{f}(x)}$. Since $\ul{\mathcal{Y}}$ is quasi-separated, it follows that $i$ is qcqs and hence $G= i_*\widetilde{N}$ is quasi-coherent with stalk at $\ul{f}(x)$ equal to $N$ as desired.\\
    For arbitrary $\ul{\mathcal{Y}}$, we choose a flat cover $\phi\colon \ul{U}\twoheadrightarrow \ul{\mathcal{Y}}$ by a scheme $\ul{U}$ and let 
    \begin{equation*}
    \begin{tikzcd}
        \ul{\mathcal{U}} = \ul{\mathcal{X}}\times_{\ul{\mathcal{Y}}}\dar[two heads, "\ul{\phi}'"] U\rar["\underline{f}'"] & U\dar[two heads, "\ul{\phi}"]\\
        \underline{\mathcal{X}}\rar["\underline{f}"]&\underline{\mathcal{Y}}
    \end{tikzcd}
    \end{equation*}
    be the resulting base change. Since $\underline{\phi}'$ is faithfully flat, it follows that $\mathcal{H}^{-i}(F\otimes_{\mathcal{O}_{\ul{\mathcal{X}}}}^\mathbb{L} L\ul{f}^*G)$ vanishes if and only if 
    \begin{align}\label{blablashit}
        L(\ul{\phi}')^*(F\otimes^{\mathbb{L}}_{\mathcal{O}_{\ul{\mathcal{X}}}} L\ul{f}^*G) = (\ul{\phi}')^*F \otimes_{\mathcal{O}_{\mathcal{U}}}^\mathbb{L} L(\ul{f}\circ \ul{\phi}')^* G = (\ul{\phi}')^*F \otimes_{\mathcal{O}_{\ul{\mathcal{U}}}}^\mathbb{L} L(\ul{f}')^*\ul{\phi}^* G
    \end{align}
    vanishes in degree $-i$. As $\ul{U}$ is a scheme, this shows the first claim.
    
    For the second claim, note that it is Zariski local in $\ul{\mathcal{Y}}$ as any quasi-coherent sheaf on a qcqs open substack of $\ul{\mathcal{Y}}$ may be extended to all of $\ul{\mathcal{Y}}$ as seen above. Thus, we may assume that $\ul{\cY}$ has affine diagonal and $\ul{U}$ is affine, which means that we only need to check that $\mathcal{H}^{-i}((\ul{\phi}')^*F\otimes_{\mathcal{O}_{\ul{\mathcal{U}}}}^\mathbb{L} L(\ul{f}')^*G)=0$ for $i=d+1$ and $G$ quasi-coherent on $\ul{U}$. Indeed, by \eqref{blablashit} we know that this is the case if $G$ is pulled back from $\ul{\mathcal{Y}}$. In general, denote by 
    \[
    \begin{tikzcd}
        \ul{\mathcal{U}}\times_{\ul{\mathcal{X}}} \ul{\mathcal{U}}\dar[xshift=0.7ex,"\pi_2'"] \dar[xshift=-0.7ex,"\pi_1'"'] \rar["\ul{f}''"] &\ul{U}\times_{\ul{\mathcal{Y}}} \ul{U}\dar[xshift=0.7ex,"\pi_2"] \dar[xshift=-0.7ex,"\pi_1"']\\
        \ul{\mathcal{U}} \rar["\ul{f}'"] &\ul{U}
    \end{tikzcd}
    \]
    the cartesian diagram involving the canonical projections. Since $\ul{\pi}'_1$ is faithfully flat, it suffices to show that
    $$
        L(\ul{\pi}_1')^*((\ul{\phi}')^*F\otimes_{\mathcal{O}_{\ul{\mathcal{U}}}}^\mathbb{L} L(\ul{f}')^*G) = (\ul{\phi}'\circ \ul{\pi}_1')^*F \otimes^\mathbb{L}_{\mathcal{O}_{\ul{\mathcal{U}}\times_{\ul{\mathcal{X}}} \ul{\mathcal{U}}}}L(\ul{f}'\circ\ul{\pi}_1')^*G = (\ul{\pi}'_2)^*(\ul{\phi}')^*F \otimes^\mathbb{L}_{\mathcal{O}_{\ul{\mathcal{U}}\times_{\ul{\mathcal{X}}} \ul{\mathcal{U}}}}L(\ul{f}'')^*(\ul{\pi}_1^*G)
    $$
    has the desired vanishing. Note also that $\ul{\pi}_2'$ is affine since it is the base change of $\ul{\pi}_2\colon \ul{U}\times_{\ul{\mathcal{Y}}}\ul{U}\rightarrow \ul{U}\times \ul{U}\xrightarrow{\ul{\pi}_2} \ul{U}$, where the first morphism is a base change of $\Delta_{\ul{\mathcal{Y}}}$ and the second is affine since $\ul{U}$ is. Thus, it is enough to show vanishing for the push forward
    \begin{align*}
         R(\ul{\pi}_2')_* \left((\ul{\pi}'_2)^*(\ul{\phi}')^*F \otimes^\mathbb{L}_{\mathcal{O}_{\ul{\mathcal{U}}\times_{\ul{\mathcal{X}}} \ul{\mathcal{U}}}}L(\ul{f}'')^*(\ul{\pi}_1^*G)\right)
        &= (\ul{\phi}')^*F\otimes^\mathbb{L}_{\mathcal{O}_{\ul{\mathcal{U}}}}R(\ul{\pi}_2')_*L(\ul{f}'')^*(\ul{\pi}_1^*G)\\
        &= (\ul{\phi}')^*F\otimes^\mathbb{L}_{\mathcal{O}_{\ul{\mathcal{U}}}}L(\ul{f}')^*((\ul{\pi}_2)_*\ul{\pi}_1^*G)\\
        &= (\ul{\phi}')^*F\otimes^\mathbb{L}_{\mathcal{O}_{\ul{\mathcal{U}}}}L(\ul{f}')^*(\ul{\phi}^*\ul{\phi}_*G),
    \end{align*}
    where we used the projection formula \cite[Corollary 4.12]{HallRydh2017} (note that $\ul{\pi}_2'$ is concentrated as it is affine) in the first equality and flat base change for stacks (see \cite[Lemma 1.2(4)]{HallRydh2017}) in the last two equalities. By \eqref{blablashit}, this finishes the proof.
\end{proof}
In particular, this gives us the following characterization that will be useful later on:
\begin{corollary}\label{corollary: relative Tor dimension quotient stack}
    Consider a sequence of morphisms
    \[
        \ul{X}\longrightarrow \ul{Y}\longrightarrow [\ul{Y}/\ul{G}],
    \]
where $\ul{X}\rightarrow \ul{Y}$ is a morphism of affine schemes corresponding to a ring homomorphism $A\rightarrow B$ and $\ul{G}=\text{Spec}(C)$ is a smooth affine group scheme acting on $\ul{Y}$. Then for any $B$-module $M$ the following are equivalent:
    \begin{enumerate}
        \item\label{item: tor dim quotient stack} $\widetilde{M}$ is of Tor dimension $\leq d$ over $[\ul{Y}/\ul{G}]$.
        \item\label{item: tor dim equiv mods} We have $\text{Tor}_{i}^A(M,N)=0$ for all $i>d$ (or equivalently, $i=d+1$) and any $A$-module $N$ that admits a $C$-comodule structure compatible with the one on $A$. 
    \end{enumerate}
\end{corollary}
\begin{proof}
    Note that the modules occuring in \eqref{item: tor dim equiv mods} are exactly the quasi-coherent sheaves on $Y$ that are pulled back from $[Y/G]$. By Proposition \ref{prop: stacky Tor dimension in terms of qcoh sheaves}, it thus follows that \eqref{item: tor dim quotient stack} is equivalent to the vanishing of
    \[
        H^{-i}(M\otimes^{\mathbb{L}}_B (B\otimes^\mathbb{L}_A N) ) = \text{Tor}_i^A(M,N)
    \]
    for $i>d$ and such $N$, which is \eqref{item: tor dim equiv mods}.
\end{proof}
\subsection{Definition of static quasi-coherent sheaves}
In order to define log Tor dimension, we first recall Olsson's stack of log structures.
\begin{definition}
    Let $\cX$ be a DM $k$-stack with a log structure. We define \textit{Olsson's stack of log structures} $\cL og_{\cX}$ \textit{over} $\cX$ as the stack whose objects over a scheme $\ul{T}$ are pairs $(T, f)$, where $T$ is a (fine and saturated) log scheme with underlying scheme $\ul{T}$, and $f \colon T \to \cX$ is a morphism of stacks with log structures. When $\cX=\text{Spec}(k)$ with the trivial log structure, we write $\cL og_{k}\coloneqq \cL og_{\cX}$.
\end{definition}
\begin{remark}
    \begin{enumerate}
        \item This notion goes back to \cite{log-geometry-and-algebraic-stacks}, where $\cL og_{\cX}$ was shown to be an algebraic $\ul{\cX}$-stack if $\ul{\cX}$ is a scheme. By \cite[Exercise 10.A]{Algebraic-spaces-and-stacks}, it follows that $\cL og_{\cX}$ is always algebraic.
        \item Our notation differs from that of \cite{log-geometry-and-algebraic-stacks}, where the above notion is denoted $\cT or$ and $\cL og$ instead refers to the stack parametrizing \textit{fine} log structures.
    \end{enumerate}
\end{remark}
We are now ready to define log Tor dimension.
\begin{definition}\label{definition: log Tor dimension}
Let $\mathcal{X}$ be a DM $k$-stack with a log structure and $F$ a quasi-coherent sheaf on the underlying stack $\underline{\mathcal{X}}$. We say that $F$ has \emph{log Tor dimension $\leq d$} (over k) if $F$ has Tor dimension $\leq d$ along $\underline{\mathcal{X}}\rightarrow \mathcal{L}og_{k}$.
Moreover, in the case of $d=1$ (resp. $d=0$), we call $F$ \emph{static} (resp. \emph{log flat}). Furthermore, we call the DM stack $\mathcal{X}$ \textit{static} if its structure sheaf $\mathcal{O}_{\mathcal{X}}$ is static.
\end{definition}
This definition fits the pattern of \cite[Definition 4.1]{log-geometry-and-algebraic-stacks}. Also see \cite{gillam2016logarithmicflatness} for the first account of log flatness.
\begin{remark}\label{remark-on-reduction-on-artin-fan-log-Tor dimension}
If $\mathcal{X}$ admits an Artin fan $\mathcal{A}_\mathcal{X}$ in the sense of \cite{AWbirational} or a global chart $P\rightarrow \Gamma(\mathcal{X},\mathcal{M}_\mathcal{X})$, then it follows from Lemma \ref{lemma: nice properties of rel Tor dim} \eqref{item: rel Tor stacks etale} that a quasi-coherent sheaf $F$ on $\ul{\mathcal{X}}$ has log Tor dimension $\leq d$ if and only if it has Tor dimension $\leq d$ over $\mathcal{A}_\mathcal{X}$ or $\mathcal{A}_P$.
\end{remark}

Before giving examples, we provide a chart criterion for logarithmic Tor dimension in the spirit of \cite{gillam2016logarithmicflatness,Logarithmic-structures-of-Fontaine-Illuie-I-Kato}.

\begin{proposition}[Chart criterion]\label{statement-monoid-static}
Let $P\rightarrow A$ be a pre-log $k$-algebra and $M$ an $A$-module. Then for each integer $d\geq 0$, the following are equivalent:
\begin{enumerate}
    \item\label{item-1-monoid-claim} $\widetilde{ M}$ has log Tor dimension $\leq d$ on $X=\text{Spec}(P\rightarrow A)$. 
    \item\label{item-2-monoid-claim} For every prime ideal $\mathfrak{p}\subset P$, we have \[\text{Tor}_{d+1}^{k[P]}(M,k[P]/k[\mathfrak{p}])=0.\]
    \item\label{item-3-monoid-claim} For every abelian group $G$ and inclusion of monoids $ P\subset G$, we have
   \[\text{Tor}_{i}^{k[P]}(M,N_{\bullet})=0\]
    for all $i> d$ and $N_{\bullet}$ a $G$-graded $k[P]$-module such that $pN_{g}\subset N_{pg}$ for all $p\in P$ and $g\in G$.
\end{enumerate}
\end{proposition}
\begin{proof}
By Remark \ref{remark-on-reduction-on-artin-fan-log-Tor dimension}, log Tor dimension is the same as Tor dimension over $\mathcal{A}_{P}$. Thus, it follows from Corollary \ref{corollary: relative Tor dimension quotient stack} that $\eqref{item-1-monoid-claim}$ is equivalent to $\eqref{item-3-monoid-claim}$ for $G=P^{\text{gp}}$, which establishes the implications \eqref{item-3-monoid-claim}$\Rightarrow$\eqref{item-1-monoid-claim}$\Rightarrow$\eqref{item-2-monoid-claim}. To show that \eqref{item-2-monoid-claim} implies \eqref{item-3-monoid-claim} it suffices to show that $\text{Tor}_{n+1}^{k[P]}(M,N_{\bullet})=0$ for all $N_{\bullet}$ as in \eqref{item-3-monoid-claim}, since for any $N_{\bullet}$, we can always find a $G$-graded surjection $\pi\colon P_{\bullet}\rightarrow N_{\bullet}$ by a free $k[P]$-module and so 
\[\text{Tor}^{k[P]}_{i+1}(M,N_{\bullet})=\text{Tor}_{i}^{k[P]}(M,\text{Ker}(\pi))\text{ for all }i\geq 1.\]
Since $N_{\bullet}$ is the filtered colimit of its finitely generated graded submodules and Tor commutes with filtered colimits, it suffices to check the claim for $N_{\bullet}$ finitely generated. By \cite[Exercise 3.5 (b)]{eisenbud2013commutative}, every such module admits a filtration by graded submodules with subquotients of the form $k[P]/k[\mathfrak{p}]$ with $\mathfrak{p}\subset P$ a prime ideal, which concludes the proof.
\end{proof}
\begin{remark}\label{ex:YesLogFlat}
    Let $X$ be a toric variety and let $Z\hookrightarrow X$ be any closed subscheme with pulled back log structure. Log flatness is a condition which enforces that $Z$ meets the toric boundary transversely. This condition ensures that, whenever $X_\sigma$ is a toric stratum of codimension $k$, the codimension of $Z\cap X_\sigma$ in $Z$ is $k$. Log flat subschemes of toric varieties play a central role in \cite{kennedy2023logarithmic,kennedy2025logarithmic,maulik2024logarithmic,tevelev2005compactifications}.
\end{remark}

\begin{example}
    Let $X$ be a smooth toric variety with toric log structure. The closure of a torus orbit is log flat if and only if this torus orbit is the dense torus. The closure of a torus orbit is static if and only if the torus orbit has codimension at most one.
\end{example}
\begin{remark}
In \cite{GabberToricFlattening, gabber2018foundationsringtheory,on-toric-log-schemes}, a module $M$ on a pre-log ring $P\rightarrow A$ is defined to be of 
log Tor dimension $\leq d$ if 
\[
    \text{Tor}_{d+1}^{\mathbb{Z}[P]}(M,\mathbb{Z}[P]/\mathbb{Z}[\mathfrak{p}])=0
\]
for all $i$ and prime ideals $\mathfrak{p}\subset P$. In the case that $A$ is a $k$-algebra, one can show that 
\[
    \text{Tor}_{d+1}^{\mathbb{Z}[P]}(M,\mathbb{Z}[P]/\mathbb{Z}[\mathfrak{p}]) = \text{Tor}_{d+1}^{k[P]}(M,k[P]/k[\mathfrak{p}]).
\]
Together with Proposition \ref{statement-monoid-static}, this implies that the notion in op. cit. agrees with Definition \ref{definition: log Tor dimension}, which is therefore independent of the ground field $k$.
\end{remark}

\subsection{Functoriality properties of static sheaves} In this section, we describe how staticity behaves under pushforward along integral maps and pullback along log flat morphisms. An important application is Proposition \ref{proposition: log flat over static is static} which plays a central role in understanding behaviour of the functor of points of moduli of log coherent sheaves \cite{Kennedy-Hunt-Poiret-Song:Log-Vector-bundle}.
\begin{lemma}\label{lemma-on-pushforward-along-integral-affine-preserves-log-Tor dimensions}
Let $f\colon \mathcal{X}\rightarrow \mathcal{Y}$ be an affine and integral morphism of DM $k$-stacks with log structures and $F$ a quasi-coherent sheaf on $\mathcal{X}$ that has log Tor dimension $\leq d$. Then $\ul{f}{}_{*}F$ also has log Tor dimension $\leq d$.
\end{lemma}
\begin{proof}
By flat base change \cite[Lemma 1.2(4)]{HallRydh2017}, we may replace $\mathcal{Y}$ by a strict \'etale cover and hence assume that $\mathcal{Y}$ is affine and has a global chart by a sharp monoid $P$. Let $\{U_i\rightarrow \mathcal{X}\}_{i\in I}$ be a finite \'etale cover by affines so that each $U_i\rightarrow \mathcal{Y}$ admits a chart $P\rightarrow Q_i$, which is local and integral. By \cite[Proposition I.4.6.7]{lecturesonlogarithmicalgebraicgeometry}, it follows that $k[P]\rightarrow k[Q_i]$ and hence $\mathcal{A}_{Q_i}\rightarrow \mathcal{A}_{P}$ is flat, which implies that $F|_{U_i}$ has Tor dimension $\leq d$ over $\mathcal{A}_{P}$ for all $i$. It follows that $F$ and hence (e.g. by Corollary \ref{corollary: relative Tor dimension quotient stack}) $\ul{f}{}_{*}F$ has Tor dimension $\leq d$ over $\mathcal{A}_P$.  
\end{proof}

\begin{lemma}\label{characterzation of absolute flatness and staticity}
Let $\mathcal{X}$ be a DM $k$-stack with a log structure and $F$ be a quasi-coherent sheaf on $\mathcal{X}$. Then the following statements are equivalent:
\begin{enumerate}
    \item\label{char-1} $F$ is static (resp. log flat).
    \item\label{char-2} Denoting by $p\colon \mathcal{L}og_\mathcal{X}\rightarrow \underline{\mathcal{X}}$ the canonical projection, $\ul{p}^{*}F$ is of Tor dimension $\leq 1$ (resp. flat) over $\mathcal{L}og_k$.
\end{enumerate}
\end{lemma}
\begin{proof}
First, we show the implication \eqref{char-2}$\Rightarrow$\eqref{char-1}: The canonical morphism $\ul{\pi}\colon \underline{\mathcal{X}}\rightarrow \mathcal{L}og_\mathcal{X}$ is representable by open immersions and compatible with morphisms to $\mathcal{L}og_k$. Hence, if $\ul{p}^{*} F$ is of Tor dimension $\leq 1$ (resp. flat), then $F=\ul{\pi}^{*}\ul{p}^{*}F$ is log flat (resp. static).

For the inverse implication \eqref{char-1}$\Rightarrow$\eqref{char-2}: since the claim is strict \'etale local in $\mathcal{X}$, we may assume that $\mathcal{X}=\text{Spec}(Q\rightarrow A)$. By \cite[Remark 5.26]{log-geometry-and-algebraic-stacks}, there is a fppf cover of $\mathcal{L}og_X$ by morphisms of the form $\text{Spec}(P\rightarrow A\otimes_{k[Q]}k[P])\rightarrow \mathcal{L}og_X$, where $Q\hookrightarrow P$ is an inclusion of fs monoids. Hence, it suffices to show that $M\otimes_{k[Q]}k[P]$ is static (resp. log flat).

For log flatness, Lemma \ref{Tor-Lemma} gives a surjection
\[
    \text{Tor}_{1}^{k[Q]}(M,k[P]/k[\mathfrak{p}]) \longtwoheadrightarrow\text{Tor}_{1}^{k[P]}(M\otimes_{k[Q]}k[P],k[P]/k[\mathfrak{p}])
\]
for every prime ideal $\mathfrak{p}\subset P$ and the left side vanishes by Proposition \ref{statement-monoid-static} and the fact that $k[P]/k[\mathfrak{p}]$ is $P^{\text{gp}}$-graded.

For staticity, since we can replace $A$ with $k[Q]$, we may assume that $A$ is log flat with respect to $Q$. Indeed, in this case, we may find a short exact sequence
\[0\longrightarrow K\longrightarrow F\longrightarrow M\longrightarrow 0,\]
where $F$ is free and hence log flat. Thus, $K$ is also log flat. By Proposition \ref{proposition-on-derived-pullback-vanishing-static}, the sequence 
\[0\longrightarrow K\otimes_{k[Q]}k[P]\longrightarrow F\otimes_{k[Q]}k[P]\longrightarrow M\otimes_{k[Q]}k[P]\longrightarrow 0\]
is also exact and hence the claim follows from how $K\otimes_{k[Q]}k[P]$ and $F\otimes_{k[Q]}k[P]$ are log flat.
\end{proof}
In the above proof we needed the following fact about Tor groups:
\begin{lemma}[Tor Lemma]\label{Tor-Lemma}
Let $R\rightarrow S$ be a ring homomorphism, $M$ be an $R$-module and $N$ be an $S$-module. Then the following statements hold.
\begin{enumerate}
    \item\label{item-1-Tor-Lemma} There is a natural surjection:
\[\text{Tor}^{R}_{1}(M,N)\longtwoheadrightarrow \text{Tor}_{1}^{S}(M\otimes_{R} S,N).\]
    \item\label{item-2-Tor-lemma} If $\text{Tor}^{R}_{i}(M,S)=0$ for $\forall \,i\geq 0$, then $\text{Tor}^{R}_{n}(M,N)=\text{Tor}^{S}_{n}(M\otimes_R S,N)$ for all $n$.
\end{enumerate}
\end{lemma}
\begin{proof}
Both statements follow from the Tor spectral sequence (see \cite[068F]{stacks}).   
\end{proof}

\begin{corollary}\label{cor: static and log flat stable under log flat pullback}
Let $f\colon \mathcal{X}\rightarrow \mathcal{Y}$ be a log flat morphism between DM $k$-stacks with log structures and $F$ be a conventional quasi-coherent sheaf on $\ul{\mathcal{Y}}$. If $F$ is static (resp. log flat), then $\ul{f}^{*}F$ is also static (resp. log flat).
\end{corollary}
\begin{proof}
The morphism $f$ induces the following diagram:
\begin{equation*}
\begin{tikzcd}
\underline{\mathcal{X}} \arrow[r, "\pi_f"] \arrow[rd, "\underline{f}"] & \mathcal{L}og_\mathcal{Y} \arrow[d, "p_\mathcal{Y}"] \\
                                                           & \underline{\mathcal{Y}}                     
\end{tikzcd}.
\end{equation*}
If $F$ is static (resp. log flat), by Lemma \ref{characterzation of absolute flatness and staticity}, $p_\mathcal{Y}^{*}F$ is static (resp. log flat). Since $\pi_f$ is flat and compatible with morphisms to $\mathcal{L}og_k$, it follows that $\underline{f}^{*}F=\pi_f^{*}(p_\mathcal{Y}^{*}F)$ is static (resp. log flat).  
\end{proof}

\begin{proposition}\label{proposition: log flat over static is static}
Let $f\colon \mathcal{X}\rightarrow \mathcal{Y}$ be a morphism of DM $k$-stacks with log structures and $F$ be a conventional quasi-coherent sheaf on $\ul{\mathcal{X}}$ that is log flat over $\mathcal{Y}$. If $\mathcal{Y}$ is static (resp. log flat), then $F$ is also static (resp. log flat). 
\end{proposition}
\begin{proof}
By Lemma \ref{characterzation of absolute flatness and staticity}, $\mathcal{L}og_\mathcal{Y}\rightarrow \mathcal{L}og_{k}$ is of Tor dimension $\leq 1$ (resp. flat). Since $F$ is flat over $\mathcal{L}og_\mathcal{Y}$, it follows from Lemma \ref{lemma: nice properties of rel Tor dim} \eqref{item: rel Tor stacks etale} that $F$ is static (resp. log flat).
\end{proof}
\subsection{Statification by logarithmic modification} 
Staticity is a very restrictive condition. Indeed, a log scheme whose log structure is everywhere of rank $\geq 2$ does not admit any non-zero static quasi-coherent sheaves. However, the following result, due to Gabber and Thompson, implies that every coherent sheaf is log \'etale locally static. 
\begin{theorem}[{\cite{GabberToricFlattening}, \cite[Theorem 3.3.4]{on-toric-log-schemes}}]\label{really technical theorem}
Let $A$ be a Noetherian ring, $P$ a fine and saturated monoid, and $X\coloneqq \text{Spec}(P\rightarrow A[P])$. Given a conventional coherent sheaf $F$ on $\underline{X}$,  there is an ideal $I\subset P$ such that the strict transform $\ul{q}^{!}F$ in the sense of \cite[080D]{stacks} is log flat, where $q\colon \text{Bl}_{A[I]}X\rightarrow X$ is the corresponding log modification.
\end{theorem}

\begin{theorem}\label{thm-pullback-static}
Let $\mathcal{X}$ be a Noetherian DM $k$-stack with a log structure and $F$ be a coherent sheaf on $\ul{\mathcal{X}}$. Then there is a log modification $q\colon \mathcal{X}'\rightarrow \mathcal{X}$ such that $\ul{q}^{*}F$ is static.
\end{theorem}
\begin{definition}
    We will call any log modification as above a \textit{statification} of $F$.
\end{definition}
\begin{proof}[Proof of Theorem \ref{thm-pullback-static}]
We first assume that $\mathcal{X}=\text{Spec}(P\rightarrow A)$. Using the surjection $A[P]\twoheadrightarrow A$ induced by the log structure, we get a strict closed immersion $i\colon \mathcal{X}\hookrightarrow Y\coloneqq \text{Spec}(P\rightarrow A[P])$.
Since $\ul{i}{}_{*}F$ is again coherent, there exists a short exact sequence 
\[0\longrightarrow K\longrightarrow \mathcal{O}_{\ul{Y}}^{\oplus n}\longrightarrow \ul{i}{}_{*}F\longrightarrow 0\]
with $K$ coherent.
By Theorem \ref{really technical theorem}, there is a log modification $q\colon Y'\rightarrow Y$ such that $\ul{q}^{!}K$ is log flat. As a result, the short exact sequence 
\[0\longrightarrow \ul{q}^{!}K\longrightarrow \mathcal{O}_{\ul{Y}'}^{\oplus n}\longrightarrow \ul{q}^{*}\ul{i}{}_{*}F\longrightarrow 0\]
witnesses that $\ul{q}^{*}\ul{i}{}_{*}F$ is static. 
Taking the pullback 
\begin{equation*}
\begin{tikzcd}
\mathcal{X}' \arrow[r, "i'",hook] \arrow[d, "q'"] & Y' \arrow[d, "q"] \\
\mathcal{X} \arrow[r, "i",hook]                                   & Y                         
\end{tikzcd}
\end{equation*}
We see that $(\ul{q}')^{*}F$ is static as $(\ul{i}')_{*}(\ul{q}')^{*}F=\ul{q}^{*}\ul{i}{}_{*}F$ by affine base change \cite[02KG]{stacks}. Hence, $q'$ is the desired log modification.

For arbitrary $\mathcal{X}$ we choose a finite \'etale cover $\{U_i\rightarrow \mathcal{X}\}_{i=1}^n$ by globally charted affine log schemes. By the above, there are log modifications $q_i\colon U'_i\rightarrow U_i$ such that the restriction of $F$ to each $U_i'$ is static. Then, as shown in \cite{m-type-topology}, there is a log modification $q\colon \mathcal{X}'\rightarrow \mathcal{X}$ such that the base change of $\coprod_{i=1}^n U_i'\rightarrow \coprod_{i=1}^n U_i\rightarrow \mathcal{X}$ along $q$ is a strict \'etale cover. Since being static is strict \'etale local, this implies that $\ul{q}^* F$ is static as desired.
\end{proof}
\subsection{Computing statification.}
Although Theorem~\ref{thm-pullback-static} guarantees the existence of statifications in great generality, it is not immediately clear how to find them in practice. In this section, we present an explicit algorithm to do this using tropical geometry, and in particular Gröbner theory, for log schemes that are of finite type over $k$. Our algorithm also provides a method for checking whether a sheaf is static under the same hypotheses.

First, we recall the definition of initial forms and modules.\footnote{We warn the reader that in the ideal sheaf case (introduced in \cite{Cartwright}), the Gr\"obner stratification is often strictly coarser than the well-studied Gr\"obner fan \cite{MaclaganSturmfels}.}
\begin{definition}
Let $f\in k[x_1^{\pm},...,x_{n}^{\pm}]^{\oplus m}$ be a vector of Laurent polynomials and $\mathbf{w}\in \mathbb{R}^m$. Writing 
$$f=\sum_{\mathbf{n}\in \mathbb{Z}^n}\sum_{1\leq i\leq m} a_{\mathbf{n},i}\cdot \mathbf{x}^\mathbf{n}e_i$$ with $a_{\mathbf{n},i}\in k$ and $e_i$ the standard basis, we define the \textit{initial form} of $f$ with respect to $\mathbf{w}$ by
$$\text{in}_\mathbf{w}(f)\coloneqq \sum_{\substack{\mathbf{n}, i \\ \mathbf{n}\cdot \mathbf{w} = d_0}} a_{\mathbf{n},i}\cdot\mathbf{x}^\mathbf{n}e_i,$$
where $d_0$ is the minimum value of $\mathbf{n}\cdot \mathbf{w}$ occurring such that $a_{\mathbf{n},i}\neq 0$.

More generally, if $G\subset k[x_1^{\pm},...,x_{n}^{\pm}]^{\oplus m}$ is a submodule, then the \textit{initial module} of $G$ with respect to $\mathbf{w}$ $$\text{in}_\mathbf{w}(G) = \langle \text{in}_\mathbf{w}(g) \mid g\in G\rangle \subset k[x_1^{\pm},x_2^{\pm},...,x_{n}^{\pm}]^{\oplus m}$$ is the submodule generated by the initial forms of elements in $G$. The \textit{Gröbner stratification} of $G$ is the stratification on $\mathbb{R}^n$ given by 
$$\mathbb{R}^n = \bigsqcup_{G_0\in k[x_1^{\pm},...,x_{n}^{\pm}]^{\oplus m}} \{\mathbf{w}\in \mathbb{R}^n\mid \text{in}_{\mathbf{w}}(G) = G_0\}.$$
\end{definition}

The Gröbner stratification of $G$ is a finite piecewise linear stratification \cite{Cartwright,kennedy2023logarithmic}. Our next result provides a tropical characterization of staticity, and the promised algorithm for computing statification. Parallel results for log flatness are available \cite[Theorem~B]{kennedy2025logarithmic}.

\begin{proposition}\label{prop:computing statification}
Let $i\colon X\hookrightarrow Y$ be a strict closed embedding into a smooth toric $k$-variety and $\mathbb{G}_m^{n} \subset Y$ be the dense open torus. Let $F$ be a coherent sheaf on $\underline{X}$ and
$$0\longrightarrow K\longrightarrow\mathcal{O}_{\underline{Y}}^{\oplus n_2}\longrightarrow\mathcal{O}_{\underline{Y}}^{\oplus n_1}\longrightarrow \underline{i}{}_{*}F\longrightarrow 0$$ a resolution. For any smooth toric modification $Y'\rightarrow Y$, the following are equivalent:
\begin{enumerate}
\item\label{item-1-prop-computing-statification} The fan of $Y'$ subdivides the Gröbner stratification of $$H^0(\mathbb{G}_m^n,K)\subset H^0(\mathbb{G}_m^n,\cO_{\underline{Y}}^{\oplus n_2}) = k[x_1^{\pm},\ldots,x_n^{\pm}]^{\oplus n_2}$$ on the support of the fan of $Y$.
\item \label{item-2-prop-computing-statification}
The pullback $(\underline{q}')^*F$ along the corresponding log modification $q'\colon X' = X\times_Y Y'\rightarrow X$ is static.
\end{enumerate}
In particular, the Gr\"obner stratification of $K$ depends only upon $F$, and not on the chosen presentation.
\end{proposition}   
\begin{proof}
Applying affine base change \cite[02KG]{stacks} to the cartesian square
\begin{equation*}
\begin{tikzcd}
    X'\dar["q'"]\rar["i'",hook] &Y'\dar["q"]\\
    X\rar["i",hook] &Y 
\end{tikzcd}
\end{equation*}
yields $\underline{q}^*\underline{i}{}_*F = \underline{i}{}'_*(\underline{q}')^*F$. Hence statement \eqref{item-2-prop-computing-statification} above is equivalent to $\underline{q}^*\underline{i}{}_*F$ being static, which is in turn equivalent to $\underline{q}^{!}(\cO_{\underline{Y}}^{\oplus n_2}/K)$ being log flat. We conclude by \cite[Theorem B]{kennedy2025logarithmic}.
\end{proof}
\begin{remark}\label{remark on computing statification}
\begin{enumerate}
    \item\label{item-1-remark-on-computing-stratification} The Gröbner stratification can be explicitly computed (often by hand) using Gröbner theory and any finite piecewise linear stratification can always be refined by a fan. Thus, Proposition \ref{prop:computing statification} gives an algorithm to compute statifications of $F$.
    \item More generally, this algorithm applies whenever $X$ is of finite type over $k$, since locally, one can always embed $X$ into a toric variety, whose singularities can always be resolved by toric blowup.
\end{enumerate}
\end{remark}
We present the following examples to illustrate Remark \ref{remark on computing statification} \eqref{item-1-remark-on-computing-stratification}.
\begin{example}\label{example: statification of structure sheaf 1}
Consider the counterexample of Remark \ref{remark: first instance of Max's favorite counterexample} \eqref{item: actual counterexample first time}, which is given by $$
    i\colon X= V(x^{2},y^{2})\hookrightarrow  \mathbb{A}^2_{x,y}
$$ 
with log structure induced by the toric log structure on $\mathbb{A}^2_k$.
To compute a statification of the structure sheaf $\cO_{\ul{X}}$, we use the following resolution:
$$ 0 \longrightarrow 
\mathcal {O}_{\mathbb{A}^2}
\xrightarrow{(y^{2}, -x^{2})}
\mathcal{O}_{\mathbb{A}^2}^{\oplus 2}
\xrightarrow{(x^2,y^2)}
   \mathcal{O}_{\mathbb{A}^2}
   \longrightarrow
   \underline{i}{}_*\cO_{\ul{X}}
   \longrightarrow 0.
$$
The Gröbner stratification of $k[x^{\pm},y^{\pm}]\cdot (y^2,-x^2)\subset k[x^{\pm},y^{\pm}]^{\oplus 2}$ coincides with the fan of $q\colon\text{Bl}_0\mathbb{A}^2\rightarrow \mathbb{A}^2$, hence it follows from Proposition \ref{prop:computing statification} that the structure sheaf of $X' = X\times_{\mathbb{A}^2}\text{Bl}_0\mathbb{A}^2$ is static.
\end{example}
\begin{example}\label{example: statification of structure sheaf 2}
    We repeat the above analysis for \cite[Example 6.8]{molcho2023remarks}, where $X$ is the strict closed subscheme $$i\colon X = V(x^3 z - x y^2 z,\; x y^3 - x y z^2,\; y z^3 - x^2 y z)\hookrightarrow \mathbb{A}^3_{x,y,z}.$$ The structure sheaf $\cO_{\ul{X}}$ admits a resolution
$$
0 \longrightarrow \cO_{\mathbb{A}^3}
\xrightarrow{(y,z,x)}
\cO_{\mathbb{A}^3}^{\oplus 3}
\xrightarrow{(x^{3}z - x y^{2}z ,\; x y^{3} - x y z^{2} ,\; y z^{3} - x^{2} y z)}
\cO_{\mathbb{A}^3} \longrightarrow \ul{i}{}_*\cO_{\ul{X}} \longrightarrow 0,
$$
which leads us to consider the Gröbner stratification of $k[x^{\pm},y^{\pm},z^{\pm}]\cdot (y,z,x)\subset k[x^{\pm},y^{\pm},z^{\pm}]^{\oplus 3}$. One can verify that this coincides with the fan of $q\colon\text{Bl}_0\mathbb{A}^3\rightarrow \mathbb{A}^3$, which implies that the structure sheaf of $X' = X\times_{\mathbb{A}^3}\text{Bl}_0\mathbb{A}^3$ is static.
\end{example}
%
%
%
%
%
\subsection{What is staticity good for?}
\subsubsection{Logarithmic \'etale descent and exactness} In conventional scheme theory, there is a close connection between modules and quasi-coherent sheaves on an affine scheme. Corollary \ref{corollary-DM-stacks-static-pullback} demonstrates that in log geometry there is a parallel close connection between static sheaves and global sections on an affine patch.
\begin{proposition}\label{proposition-on-derived-pullback-vanishing-static}
Let $f\colon \mathcal{X}'\rightarrow \mathcal{X}$ be a morphism between DM $k$-stacks with log structures, and $F$ be a static quasi-coherent sheaf on $\mathcal{X}_{\text{\'et}}$.
\begin{enumerate}
    \item\label{item: static first derived pull vanishes} If $f$ is log flat, then $L_{1}\ul{f}^{*}F=0$.
    \item\label{item: static all derived pulls vanish} If $f$ is log flat and $\mathcal{X}$ is static, then $L_{i}\ul{f}^{*} F=0$ for all $i\geq 1$.
    \item\label{item: static push pull is identity} If $f$ is a log alteration, then we have
\[F=R\ul{f}{}_{*}\ul{f}^{*}F.\]
\end{enumerate}
\end{proposition}

\begin{proof}
The claims about derived pullback are local in $\mathcal{X}$ and $\mathcal{X}'$, hence we may assume that $\mathcal{X}=\text{Spec}(P\rightarrow A)$, $\mathcal{X}'=\text{Spec}(Q\rightarrow B)$, $F=\widetilde{M}$ for an $A$-module $M$, and $f$ admits a chart $P\hookrightarrow Q$ given by an inclusion of monoids. Since $k[Q^{\text{gp}}]$ is a free $k[P^{\text{gp}}]$-module, it is flat over $k[P]$. Hence
\[\text{Tor}^{k[P]}_{i}(k[Q],M)=\text{Tor}^{k[P]}_{i+1}(k[Q^{\text{gp}}]/k[Q],M)=0\]
for all $i\geq 1$ as $M$ is static. By Lemma \ref{Tor-Lemma} \eqref{item-1-Tor-Lemma}, we thus have $\text{Tor}_{1}^{A}(A\otimes_{k[P]}k[Q],M)=0$, which implies the $L_{1}\ul{f}^{*}F=0$. In the case that $A$ is static, the vanishing of all higher Tors follows from Lemma \ref{Tor-Lemma} \eqref{item-2-Tor-lemma}.

It suffices to verify the last claim for log modifications and root stacks separately. Indeed, for root stacks, this holds as
\[
    R\ul{f}{}_*\ul{f}^*F = \ul{f}{}_*\ul{f}^*F = \mathcal{H}^0(R\ul{f}{}_*L\ul{f}^*F) = \mathcal{H}^0(F\otimes_{\mathcal{O}_\mathcal{\mathcal{X}}}^\mathbb{L} R\ul{f}{}_*\mathcal{O}_\mathcal{X'}) = F,
\]
where we used the projection formula \cite[Corollary 4.12]{HallRydh2017} and Proposition \ref{proposition: root stacks are a thing} (note that this did not use that $F$ is static). Since the claim for log modifications is local in $\mathcal{X}$, we may therefore assume that $\mathcal{X}=\text{Spec}(P\rightarrow A)$, $F=\widetilde{M}$ and $f$ arises as a pullback
$$
\begin{tikzcd}
\ul{\mathcal{X}}' \arrow[d, "i"] \arrow[r, "\ul{f}"]                & \ul{\mathcal{X}} \arrow[d, "g"]    \\
{\text{Bl}_{k[I]}\text{Spec}(k[P])} \arrow[r, "h"] & {\text{Spec}(k[P])}
\end{tikzcd}
$$
for some ideal $I\subset P$. Since $g$ is affine, it follows that the natural homomorphism $F\rightarrow R\ul{f}{}_{*}L\ul{f}^{*}F\rightarrow R\ul{f}{}_{*}\ul{f}^{*}F$ is an isomorphism if and only if $R\ul{g}{}_{*}F=\ul{g}{}_{*}F\rightarrow R\ul{g}{}_{*}R\ul{f}{}_{*}(f^{*}F)$ is an isomorphism. However, by affine base change \cite[02KG]{stacks}, we have 
\[R\underline{g}{}_{*}R\ul{f}{}_{*}\ul{f}^{*}F=R\ul{h}{}_{*}R\ul{i}{}_{*}\ul{f}^{*}F=R\ul{h}{}_{*}\ul{i}{}_{*}\ul{f}^{*}F=R\ul{h}{}_{*}\ul{h}^{*}\ul{g}{}_{*}F.\]
Hence, it suffices to prove the claim for $\ul{g}{}_{*}F$ instead, which is also static by Lemma \ref{lemma-on-pushforward-along-integral-affine-preserves-log-Tor dimensions}. Therefore, we may assume that $\mathcal{X}=\text{Spec}(k[P])$ is static and hence $L_{i}\ul{f}^{*}F=0$ for all $i\geq 1$ as shown above. The claim now follows from the projection formula:
\[R\ul{f}{}_{*}\ul{f}^{*}F=R\ul{f}{}_{*}(L\ul{f}^{*}F)=F\otimes_{\mathcal{O_\mathcal{X}}}^{\mathbb{L}}R\ul{f}{}_{*}\mathcal{O}_{\ul{\mathcal{X}}'}=F,\]
where 
$R\ul{f}{}_*\mathcal{O}_{\ul{\mathcal{X}}'}$$=\mathcal{O}_{\ul{\mathcal{X}}}$ 
follows from \cite[Theorem 11.3]{Toric-singularities-Kato}.
\end{proof}

\begin{corollary}\label{corollary-DM-stacks-static-pullback}
Let $\mathcal{X}$ be a DM $k$-stack with a log structure and $F$ a static quasi-coherent sheaf on $\ul{\mathcal{X}}$. Then we have $L_1\arrowup F=0 $ and for any log alteration $q\colon \cX'\rightarrow \cX$ we have $R\arrowdown_{\mathcal{X}'} \arrowup F=\ul{q}^*F$.

Moreover, for any log \'etale morphism $\varphi\colon U\rightarrow \mathcal{X}$, we have 
\[\Gamma(U,\arrowup F)=\Gamma(U,\varphi^{*} F).\]
If in addition $\mathcal{X}$ is static, then $L_{i}\arrowup {F}=0$ for all $i\geq 1$.
\end{corollary}
\begin{proof}
The claims follow from Corollary \ref{Key corollary} \eqref{item-3-key-corollary} and \eqref{item-4-key-corollary} as well as Proposition \ref{proposition-on-derived-pullback-vanishing-static}.
\end{proof}

\begin{remark}\label{remark: up is exact on statics}
    \begin{enumerate}
        \item In particular, for any exact sequence $$0\longrightarrow F\longrightarrow G\longrightarrow H\longrightarrow 0$$ of quasi-coherent sheaves on $\mathcal{X}$ with $H$ static, we see that $$0\longrightarrow \arrowup F\longrightarrow \arrowup G\longrightarrow\arrowup H\longrightarrow 0$$ is still exact. Thus, $\Uparrow$ is exact on static sheaves.
        \item If $\cX$ is a static DM $k$-stack with log structure, then it follows that $\cO_{\ul{\cX}}$ satisfies log \'etale descent in the sense of \cite{molcho2023remarks}. This partially answers \cite[Question 1.7.2]{molcho2023remarks}.
        \item If $\cX$ is Noetherian, but not necessarily static, it follows that $\Downarrow \!\!\LogO{\cX} = \ul{q}{}_*\cO_{\ul{\cX}'}$ for any statification $q\colon\cX'\rightarrow\cX$ of $\cO_{\ul{\cX}}$. As a result, $\Downarrow\!\!\LogO{\cX}$ is coherent, which answers and strengthens \cite[Question 1.7.1]{molcho2023remarks} in our setting.
        \item\label{item: answer sheafification question of molcho} 
        In particular, if $X'\rightarrow X$ is one of the log modifications in Example \ref{example: statification of structure sheaf 1} or \ref{example: statification of structure sheaf 2}, then for every log \'etale neighborhood $U\rightarrow X$ we have $\Gamma(U,\LogO{X}) = \Gamma(\ul{U\times_X X'},\cO_{\ul{U\times_X X'}})$. This answers \cite[Question 1.7.3]{molcho2023remarks}.
        \item Using \eqref{item: answer sheafification question of molcho}, we can show that $\Uparrow$ is not full in general. Indeed, it was shown in \cite[Example 6.8]{molcho2023remarks} that the map $\Gamma(\ul{X},\cO_{\ul{X}})\rightarrow \Gamma(\ul{X}',\cO_{\ul{X}'})$ is not surjective, where $X'\rightarrow X$ is the log modification of Example \ref{example: statification of structure sheaf 2}. Hence
        \[
            \text{Hom}_{\cO_{\underline{X}}}(\cO_{\ul{X}},\cO_{\ul{X}}) = \Gamma(\ul{X},\cO_{\ul{X}}) \longrightarrow \Gamma(\ul{X}',\cO_{\ul{X}'}) =\Gamma(X,\LogO{X}) = \text{Hom}_{\cO_{X_{\text{l\'et}}}}(\LogO{X},\LogO{X})
        \]
        is not surjective.
    \end{enumerate} 
\end{remark}
More generally, $\Uparrow$ is not exact in the sense that it does not commute with taking cohomology sheaves - even for complexes consisting of static quasi-coherent sheaves. However, this holds under stronger hypotheses:
\begin{proposition}\label{prop: static complex prop}
    Let $\mathcal{X}$ be a DM $k$-stack with a log structure and $E^\bullet$ be a complex of static quasi-coherent sheaves on $\mathcal{X}$ so that $E^i=0$ for $i\gg 0$. If for some integer $n$ we have \begin{equation}\label{sheaf coh is static}
        \mathcal{H}^i(E^\bullet) \text{ is static for all } i\geq n,
    \end{equation}
    then $\ul{f}^*\mathcal{H}^i(E^\bullet) = \mathcal{H}^i(\ul{f}^*E^\bullet)$ and $\arrowup \mathcal{H}^i(E^\bullet) = \mathcal{H}^i(\arrowup E^\bullet)$ for every log flat morphism $f\colon\mathcal{X}'\rightarrow\mathcal{X}$ and $i\geq n-1$. Moreover, if $\mathcal{X}$ is Noetherian and the $\mathcal{H}^{i}(E^\bullet)$ are coherent, then for any $n$, there is a log modification $\pi\colon \mathcal{X}'\rightarrow\mathcal{X}$ so that $\pi^*E^\bullet$ satisfies \eqref{sheaf coh is static}.
\end{proposition}
\begin{remark}
    In particular, if $E^\bullet$ is an acyclic complex of static quasi-coherent sheaves on $\mathcal{X}$ that is bounded from above, then $\Uparrow \!\! E^\bullet$ is also acyclic.
\end{remark}
\begin{proof}
    Using the exact sequences 
    \begin{equation}\label{eqn: some random shit wtf}
        0\longrightarrow \text{Ker}(d^{i}_{E^\bullet}) \longrightarrow E^i \longrightarrow \text{Im}(d^i_{E^\bullet})\longrightarrow 0,
    \end{equation}
    and
    \[
        0\longrightarrow \text{Im}(d^{i-1}_{E^\bullet})\longrightarrow \text{Ker}(d^{i}_{E^\bullet})\longrightarrow \mathcal{H}^{i}(E^\bullet)\longrightarrow 0
    \]
    and the fact that staticity is stable under taking kernels of surjections, we see by downward induction in $i$ that $\text{Ker}(d^{i}_{E^\bullet})$ and $\text{Im}(d^{i-1}_{E^\bullet})$ are static for all $i\geq n$. Since $\ul{f}^*$ and $\Uparrow$ commute with $\text{Im}({-})$, it thus follows from Proposition \ref{proposition-on-derived-pullback-vanishing-static}, Corollary \ref{corollary-DM-stacks-static-pullback} and \eqref{eqn: some random shit wtf} that $\ul{f}^*\text{Ker}(d^{i}_{E^\bullet}) = \text{Ker}(\ul{f}^*d^{i}_{E^\bullet})$ for all $i\geq n-1$ and hence 
    \begin{equation}\label{eqn: coh sheaf pullback stuff}
        \ul{f}^*\mathcal{H}^i(E^\bullet) = \ul{f}^*\text{Ker}(d^{i}_{E^\bullet})/\ul{f}^*\text{Im}(d^{i-1}_{E^\bullet}) = \text{Ker}(\ul{f}^*d^{i}_{E^\bullet})/\text{Im}(\ul{f}^*d^{i-1}_{E^\bullet}) = \mathcal{H}^i(\ul{f}^*E^\bullet)
    \end{equation}
    as desired (and the same for $\Uparrow$).
    
    For the last claim, note that the pullback $\ul{f}^* E^\bullet$ along any log flat $f\colon \mathcal{X}'\rightarrow \mathcal{X}$ preserves our hypotheses. Indeed, Corollary \ref{cor: static and log flat stable under log flat pullback} implies that the $\ul{f}^* E^{i}$ are still static and by right exactness of $\ul{f}^*$ there is always a surjection $\ul{f}^*\mathcal{H}^{i}(E^\bullet)\twoheadrightarrow \mathcal{H}^{i}(\ul{f}^* E^\bullet)$ similar to \eqref{eqn: coh sheaf pullback stuff} and so the $\mathcal{H}^{i}(\ul{f}^* E^\bullet)$ are again coherent. Since $E^\bullet$ is bounded from above, the second claim thus follows from Theorem \ref{thm-pullback-static}, as the first claim allows one to inductively make the cohomology sheaves static.
\end{proof}
Using the machinery described above, we can finally show that $\Uparrow$ is in general not exact.
\begin{example}\label{example: up not exact}
    Consider $q\colon X' = \text{Bl}_0 \mathbb{A}^2 \rightarrow X = \mathbb{A}^2$ as a toric log modification. We claim that  
    \[
        0\longrightarrow I = (x,y) \longrightarrow \cO_{\mathbb{A}^2} \longrightarrow \cO_0\longrightarrow 0,
    \]
    does not remain left exact after applying $\Uparrow$. Indeed, pulling back along $\ul{q}$ yields an exact sequence
    \begin{equation}\label{equation: the circus of shitty fucks is back in town, man!}
        0\longrightarrow L_1\ul{q}^* \cO_0 \longrightarrow \ul{q}^*I\longrightarrow \cO_{\ul{X}'} \longrightarrow \cO_{\ul{E}}\longrightarrow 0,
    \end{equation}
    where $\ul{E}\subset\ul{X}'$ is the exceptional divisor. One can compute that $L_1\ul{q}^* \cO_0 = \cO_{\ul{E}}(-1)$ and so \eqref{equation: the circus of shitty fucks is back in town, man!} remains exact after applying $\Uparrow$ as all terms are static. Moreover, staticity of $\cO_{\ul{E}}(-1)$ also implies that $\Downarrow_{X'}\arrowup \cO_{\ul{E}}(-1) = \cO_{\ul{E}}(-1)$ and in particular $\Uparrow\!\! \cO_{\ul{E}}(-1)\neq 0$, which shows the claim.
\end{example}

\subsubsection{Logarithmic \'etale sheaf cohomology} We complete our discussion of static sheaves by showing that cohomology behaves in the expected way, see Corollary \ref{corollary-sheaf-cohomology-of-static-sheaves}.
\begin{proposition}\label{prop: Ext groups static}
Let $\mathcal{X}$ be a DM $k$-stack with a log structure and $F$ and $G$ be quasi-coherent sheaves on $\mathcal{X}$ with $G$ being static. 
\begin{enumerate}
    \item\label{Static-Hom} The canonical morphism 
    \[\text{Hom}_{\cO_{\underline{\mathcal{X}}{}_{\text{\'et}}}}(F,G)\longrightarrow \text{Hom}_{\LogO{\cX}}(\arrowup F,\arrowup G)\]
    is an isomorphism;
    \item\label{static-ext} If for every log alteration $q\colon\mathcal{X}'\rightarrow \mathcal{X}$, we have $L^{n}\ul{q}^{*} F=0$ for all $n>0$, then there is a canonical isomorphism
    \[\text{Ext}^{i}_{\cO_{\underline{\mathcal{X}}{}_{\text{\'et}}}}(F,G)\cong \text{Ext}_{\LogO{\mathcal{X}}}^{i}(\arrowup F,\arrowup G)\]
    for all $i$.

    Moreover, the assumption on $F$ holds in any of the following cases
    \begin{enumerate}
        \item $F$ is conventionally flat on $\mathcal{X}$;
         \item Both $F$ and $\mathcal{X}$ are static. 
    \end{enumerate}
\end{enumerate}
\end{proposition}
\begin{proof}
Statement \eqref{Static-Hom} follows from the adjunction 
\[\text{Hom}_{\cO_{\underline{\mathcal{X}}{}_{\text{\'et}}}}(F,\arrowdown_\mathcal{X}\arrowup G) = \text{Hom}_{\LogO{\mathcal{X}}}(\arrowup F,\arrowup G)\]
and $G= \arrowdown_{\mathcal{X}}\arrowup G$ as $G$ is static.

To show statement \eqref{static-ext}, we first note that $L\arrowup F=\arrowup F$ by Corollary \ref{Key corollary} \eqref{item-4-key-corollary}. Thus, we have
\[\text{Ext}^{i}_{\cO_{\mathcal{X}_{\text{l\'et}}}}(\arrowup F,\arrowup G)=\text{Ext}^{i}_{\cO_{\mathcal{X}_{\text{l\'et}}}}(L\arrowup F,\arrowup G)=\text{Ext}^{i}_{\mathcal{O}_{\ul{\mathcal{X}}{}_{\text{\'et}}}}(F,R\arrowdown_\mathcal{X}\arrowup G),\]
where the last equality uses the adjunction between $L\arrowup$ and $R\arrowdown_\mathcal{X}$ (see \cite[0FTR]{stacks}). Finally, since $G$ is static, we have $R\arrowdown_\mathcal{X}\arrowup G = G$ by Corollary \ref{corollary-DM-stacks-static-pullback}, which gives the desired isomorphism of Ext groups.
The final claim also follows from Corollary \ref{corollary-DM-stacks-static-pullback}.
\end{proof}
As an immediate consequence, this gives us control over the log \'etale sheaf cohomology of a static quasi-coherent sheaf.
\begin{corollary}\label{corollary-sheaf-cohomology-of-static-sheaves}
Let $\mathcal{X}$ be a DM $k$-stack with log structure and $F$ a static quasi-coherent sheaf on $\mathcal{X}$, then we have
\[H^{i}(\mathcal{X}_{\text{l\'et}},\arrowup F)=H^{i}(\ul{\mathcal{X}}{}_{\text{\'et}},F), \text{for }i\geq 0.\]
\end{corollary}
\section{Logarithmic (quasi-)coherent sheaves}\label{sec:LogCoherentSheaves}

\subsection{Logarithmic coherent sheaves}\label{sec:coherent sheaves}
In this section, we introduce logarithmic coherent sheaves and $\cO_{\cX_{\text{l\'et}}}$-modules of finite presentation. Parallel to conventional scheme theory, these definitions coincide when $\mathcal{X}$ is locally Noetherian, see Theorem~\ref{thm: FPisCoherent}.
\subsubsection{$\mathcal{O}_{\mathcal{X}_{\text{l\'et}}}$-modules of finite presentation} 
Let $\mathcal{X}$ be a DM $k$-stack with a log structure. Recall from \cite[03DL]{stacks} that a $\mathcal{O}_{\mathcal{X}_{\text{l\'et}}}$-module $\mathcal{F}$ is said to be \emph{of finite presentation} if there exists a log \'etale cover 
$\{\cU_i\rightarrow \mathcal{X}\}_{i\in I}$
such that for each $i$ there exists an exact sequence of $\mathcal{O}_{(\mathcal{U}_i){}_{\text{l\'et}}}$-modules
of the form:  $$\mathcal{O}_{(\cU_i)_{\text{l\'et}}}^{\oplus m_i}\longrightarrow\mathcal{O}_{(\cU_i)_{\text{l\'et}}}^{\oplus n_i}\longrightarrow \mathcal{F}|_{\cU_i}\longrightarrow 0.$$
We denote by $\text{FP}(\LogO{\cX})\subset \text{Mod}(\mathcal{O}_{\mathcal{X}_{\text{l\'et}}})$ the full subcategory of finitely presented $\mathcal{O}_{\mathcal{X}_{\text{l\'et}}}$-modules.

By the right exactness of the up functor, every finitely presented $\mathcal{O}_{\mathcal{X}'_{\text{\'et}}}$-module on some log alteration $\mathcal{X}'\rightarrow \mathcal{X}$ gives rise to a finitely presented $\mathcal{O}_{\mathcal{X}_{\text{l\'et}}}$-module $\Uparrow \!\! F$. For the converse, we have:
\begin{proposition}\label{proposition on finite presentation}
Let $\mathcal{X}$ be a qcqs DM $k$-stack with a log structure. The functor
\[
    \mathop{\text{colim}}_{\mathcal{X}'\rightarrow \mathcal{X}}  \ \text{FP}(\cO_{\ul{\mathcal{X}}{}'_{\text{\'et}}})\longrightarrow \text{FP}(\LogO{\cX})
\]
induced by $\Uparrow$ is an equivalence,
where the (filtered) colimit of categories runs over all log alterations and pullbacks between them. 
\end{proposition}

\begin{corollary}\label{corr: FP iff up FP}
    An $\mathcal{O}_{\mathcal{X}_{\text{l\'et}}}$-module $\mathcal{F}$ is of finite presentation if and only if it arises as $\cF= \arrowup F$ for some conventional finitely presented sheaf $F$ on a log alteration $\mathcal{X}'\rightarrow\mathcal{X}$. Moreover, for any morphism $\phi\colon \arrowup F\rightarrow \arrowup G$ between finitely presented $\mathcal{O}_{\ul{\mathcal{X}}{}_{\text{\'et}}}$-modules $F$ and $G$, there exists a log alteration $q\colon\mathcal{X}'\rightarrow \mathcal{X}$ and a morphism $\phi'\colon \ul{q}{}_{\text{\'et}}^*F\rightarrow \ul{q}{}_{\text{\'et}}^*G$ such that $\phi = \arrowup \phi'$.
\end{corollary}

\begin{proof}[Proof of Proposition \ref{proposition on finite presentation}]
We will first show that for any log alteration $\mathcal{X}'\rightarrow \mathcal{X}$ and any two finitely presented modules $F$ and $G$ on $\ul{\mathcal{X}}{}'_{\text{\'et}}$, the morphism of sheaves
\begin{equation}\label{eqn: some equation somewhere somehow}
    \mathop{\text{colim}}_{q\colon\mathcal{X}''\rightarrow \mathcal{X}'} \ul{q}{}^{\text{\'et}}_*\mathcal{H}om_{\mathcal{O}_{\ul{\mathcal{X}}{}''_{\text{\'et}}}}(\ul{q}{}_{\text{\'et}}^*F,\ul{q}{}_{\text{\'et}}^*G)\longrightarrow \arrowdown_{\mathcal{X}'}\mathcal{H}om_{\mathcal{O}_{\mathcal{X}_{\text{l\'et}}}}(\arrowup F,\arrowup G)
\end{equation}
is an isomorphism. From this fact, full faithfulness of the functor in Proposition \ref{proposition on finite presentation} follows since taking $\Gamma(\mathcal{X}',{-})$ on both sides yields the desired equality of Hom spaces, where we use that $\Gamma(\mathcal{X}',{-})$ commutes with filtered colimits as $\mathcal{X}'$ is qcqs. Since \eqref{eqn: some equation somewhere somehow} can be checked strict \'etale locally we may assume that $F$ is globally of finite presentation. From Corollary \ref{Key corollary} \eqref{item-3-key-corollary} it follows that \eqref{eqn: some equation somewhere somehow} is an isomorphism in the case that $F=\mathcal{O}_{\ul{\mathcal{X}}{}'_{\text{\'et}}}$. By left exactness of $\mathcal{H}om$, pushforward, filtered colimits and right exactness of $\Uparrow$ and pullback, \eqref{eqn: some equation somewhere somehow} holds whenever $F$ admits a global finite presentation.

For essential surjectivity, note that any $\mathcal{O}_{\mathcal{X}_{\text{l\'et}}}$-module $\mathcal{F}$ that admits a global finite presentation 
\[
    \mathcal{O}_{\mathcal{X}_{\text{l\'et}}}^{\oplus m}\xrightarrow{\;\;\phi\;\;}\mathcal{O}_{\mathcal{X}_{\text{l\'et}}}^{\oplus n} \longrightarrow\mathcal{F}\longrightarrow 0
\]
is in the essential image since by full faithfulness we have $\phi=\arrowup \phi'$ for some morphism $\phi'\colon \mathcal{O}_{\ul{\mathcal{X}}{}'_{\text{\'et}}}^{\oplus m}\rightarrow\mathcal{O}_{\ul{\mathcal{X}}{}'_{\text{\'et}}}^{\oplus n}$ on some log alteration $\mathcal{X}'\rightarrow \mathcal{X}$ and hence $\mathcal{F} = \text{Coker}(\phi) = \arrowup \text{Coker}(\phi')$. In general, we may choose a finite (as $\mathcal{X}$ is qc) cover $\{\cU_i\rightarrow \mathcal{X}\}_{i=1}^n$ by qcqs $\cU_i$ so that each $\mathcal{F}|_{\cU_i}$ is globally of finite presentation. By Proposition \ref{proposition-alteration-to-saturated-morphisms} it is possible to replace our cover by a refinement which is a strict \'etale cover of some log alteration $q\colon \mathcal{X}'\rightarrow \mathcal{X}$. By the previous discussion and the fact that log alterations of $\mathcal{X}$ form a filtered poset, we may choose a log alteration $\mathcal{X}'\rightarrow \mathcal{X}$ on which we can find lifts of all $\mathcal{F}|_{\cU_i}$ and of the gluing isomorphisms $\mathcal{F}|_{\cU_i}|_{\cU_i\times_\mathcal{X} \cU_j}\cong \mathcal{F}|_{\cU_j}|_{\cU_i\times_\mathcal{X} \cU_j}$ so that the cocycle condition holds for these lifts. This glues to the desired preimage of $\mathcal{F}$ along $\Uparrow$.
\end{proof}
\subsubsection{Logarithmic coherent sheaves} 
We define log coherent sheaves as coherent sheaves in the sense of \cite[03DL]{stacks} on the ringed log \'etale site. More precisely:
\begin{definition}
Let $\mathcal{X}$ be a DM $k$-stack with a log structure. A \emph{log coherent sheaf} is an $\mathcal{O}_{\mathcal{X}_{\text{l\'et}}}$-module $\mathcal{F}$ satisfying the following hypotheses:
\begin{enumerate}
\item The sheaf $\mathcal{F}$ is of finite presentation.
\item The kernel of any morphism $$\mathcal{O}_{\cU_{\text{l\'et}}}^{\oplus n}\longrightarrow \mathcal{F}|_{\cU_{\text{l\'et}}}$$ on some log \'etale open $\cU\rightarrow \mathcal{X}$ is generated by finitely many global sections. 
\end{enumerate}
The category of log coherent sheaves is denoted $\mathcal{L}\operatorname{Coh}(\mathcal{X})$.
\end{definition}
Recall that the category of coherent sheaves on a ringed site is always abelian, although it is not a priori clear whether $\logcoh{\mathcal{X}}$ contains any non-zero objects. 
Theorem \ref{thm: FPisCoherent} generalizes the following conventional fact to the logarithmic setting: An $\cO_{\ul{X}}$-module on a locally Noetherian scheme $\ul{X}$ is of finite presentation if and only if it is coherent.
\begin{theorem}\label{thm: FPisCoherent}
Let $\mathcal{X}$ be a locally Noetherian DM $k$-stack with a log structure.
\begin{enumerate}
    \item\label{item: fp is coherent} An $\mathcal{O}_{\mathcal{X}_{\text{l\'et}}}$-module $\mathcal{F}$ on $\mathcal{X}$ is coherent if and only if it is of finite presentation. Hence $\text{FP}(\LogO{\cX}) = \logcoh{\mathcal{X}}$ is abelian.
    \item\label{item: how to compute kernels} Let $\mathcal{X}$ be Noetherian and $\phi\colon F\rightarrow G$ a morphism of strict \'etale $\mathcal{O}_{\ul{\mathcal{X}}{}_{\text{\'et}}}$-modules of finite presentation. Then there is some log modification $q\colon \mathcal{X}'\rightarrow \mathcal{X}$ so that $\text{Ker}(\arrowup \phi) = \arrowup \text{Ker}(\ul{q}{}_{\text{\'et}}^*\phi)$. Indeed, this holds for any log alteration $q$ upon which $\ul{q}{}_{\text{\'et}}^*F$, $\ul{q}{}_{\text{\'et}}^*G$ and $\ul{q}{}_{\text{\'et}}^*\text{Coker}(\phi)$ are static.
\end{enumerate}
\end{theorem}
\begin{proof}
    The second claim follows from Proposition \ref{prop: static complex prop} applied to the complex $F\xrightarrow{\phi} G$. The first claim is log \'etale local in $\mathcal{X}$ and thus reduces to the second claim and the fact that every log alteration of $\mathcal{X}$ is Noetherian if $\mathcal{X}$ was Noetherian.
\end{proof}
In the setting of Theorem \ref{thm: FPisCoherent}, one might further hope that the category of log coherent sheaves is also Noetherian in the sense that any increasing sequence of subobjects must stabilize. However, this fails due to the following counterexample.

\begin{example}[The category of log coherent sheaves is not Noetherian]\label{example: not Noetherian}
Let $q_n\colon C_{n+1}\rightarrow C_n$ and $\ul{Z}{}_n\subset \ul{C}{}_n$ be the sequence of log modifications and closed subschemes of chains of projective lines depicted in Figure \ref{figurino}. Let $I_n$ be the conventional ideal sheaf associated to $\ul{Z}{}_n$.
Since there are proper inclusions $Z_{n+1}\subsetneq \ul{q}{}_n^{-1}Z_n$, we also get proper inclusions $I_{\ul{q}{}_n^{-1}\ul{Z}{}_n}\subsetneq I_{n+1}$ and since $\cO_{\ul{Z}{}_n}$ is static, we further have $I_{\ul{q}{}_n^{-1}\ul{Z}{}_n}=\ul{q}{}_n^*I_n$. One can check that the cokernels of these inclusions are static, which implies that there is an infinite ascending chain $$\arrowup I_1\subsetneq \arrowup I_2\subsetneq \ldots \subset \LogO{X}.$$
where $X=C_1$. Hence, $\text{FP}(\LogO{X})$ is not Noetherian.
\end{example}
\begin{figure}[h]
    \centering
    \begin{tikzpicture}[scale=0.75, transform shape]
  \tikzset{
    sphere/.style={
      fill=gray!10,
      draw=black,
      line width=1pt
    },
    redsphere/.style={
      fill=blue!40,
      draw=black,
      line width=1pt
    }
  }

  \node at (-4,-0.6) {\scalebox{1.5}{$\cdot\;\cdot\;\cdot$}};

  \draw[->, thick] (-2.6,-0.6) -- (-1.0,-0.6);

  \filldraw[redsphere] (0,1.6) ellipse (0.6cm and 0.8cm);
  \node at (1,1.6) {$\mathbb{P}^1$};

  \foreach \y in {0,-1.6} {
    \filldraw[sphere] (0,\y) ellipse (0.6cm and 0.8cm);
    \node at (1,\y) {$\mathbb{P}^1$};
    \fill[blue!40] (0,\y) circle (3pt);
  }

  \fill (0,2.4) circle (2pt);
  \fill (0,0.8) circle (2pt);
  \fill (0,-0.8) circle (2pt);

  \node[above] at (0,2.4) {$\infty$};

  \filldraw[redsphere] (4,0) ellipse (0.6cm and 0.8cm);
  \node at (5,0) {$\mathbb{P}^1$};

  \filldraw[sphere] (4,-1.6) ellipse (0.6cm and 0.8cm);
  \node at (5,-1.6) {$\mathbb{P}^1$};
  \fill[blue!40] (4,-1.6) circle (3pt);

  \fill (4,0.8) circle (2pt);
  \fill (4,-0.8) circle (2pt);

  \node[above] at (4,0.8) {$\infty$};

  \filldraw[redsphere] (8,-1.6) ellipse (0.6cm and 0.8cm);
  \node at (9,-1.6) {$\mathbb{P}^1$};

  \fill (8,-0.8) circle (2pt);

  \node[above] at (8,-0.8) {$\infty$};

  \draw[->, thick] (1.2,-0.6) -- (2.8,-0.6);
  \draw[->, thick] (5.2,-0.6) -- (6.8,-0.6);

\end{tikzpicture}
\caption{The above picture displays $\ldots\rightarrow C_3\rightarrow C_2\rightarrow C_1$ where $q_n\colon C_{n+1}\rightarrow C_n$ is a log modification contracting the top component to $\infty$. The closed subschemes $\ul{Z}{}_n\subset \ul{C}{}_n$ are depicted in blue.}
\label{figurino}
\end{figure}
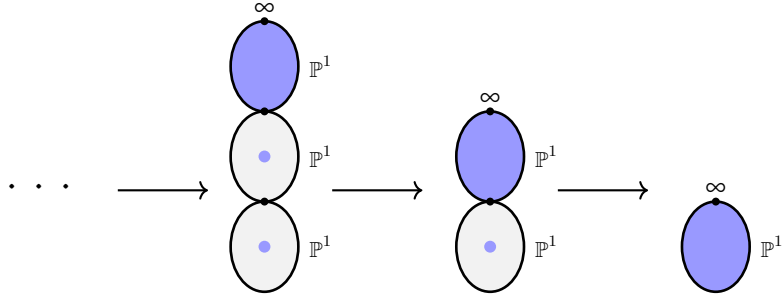

\subsubsection{Logarithmic coherent sheaves as systems of coherent sheaves}
\begin{proposition}\label{prop:DownCohDoesnt Imply Coh}
    Let $\mathcal{X}$ be a Noetherian DM k-stack with a log structure, and let $\mathcal{F}$ be a logarithmic coherent sheaf on $\mathcal{X}$. Then for every logarithmic alteration $\mathcal{X'}\rightarrow \mathcal{X}$, the sheaf $\arrowdown_{\cX'} \mathcal{F}$ is a conventional coherent sheaf.
\end{proposition}
\begin{proof}
    This is immediate from Proposition \ref{prop: derived push log}, noting that $\Downarrow_{\cX'}$ arises as the log \'etale pushforward along the canonical (saturated) morphism $\cX\rightarrow \ul{\cX}_{\text{triv}}$ to the underlying stack $\ul{\mathcal{X}}$ equipped with the trivial log structure.
\end{proof}
\begin{example}[No converse to Proposition \ref{prop:DownCohDoesnt Imply Coh}]
    Using the notation of Example \ref{example: not Noetherian}, let $$\cI= \bigcup_{n\geq 0}\; \arrowup I_n\subset \cO_{X_{\text{l\'et}}}.$$ Since any log alteration $\cX'\rightarrow X$ is static and Noetherian, it follows that $\Downarrow_{\cX'} \!\! \cI \subset \;\Downarrow_{\cX'}\!\LogO{X} = \cO_{\ul{\cX}'}$ 
    is coherent as it is a quasi-coherent ideal sheaf. However, we now argue that $\cI$ itself cannot be coherent: indeed if it were coherent then the cokernel $Q$ to $\mathcal{I}\hookrightarrow \LogO{X}$ would be coherent and thus enjoy the property described in Proposition~\ref{prop:DownCohDoesnt Imply Coh}. However, $Q$ is the colimit of $\Uparrow \!\! \!\mathcal{O}_{\underline{Z}{}_n}$ over all integers $n$. Since $X$ is qcqs and by \cite[0738]{stacks}, taking global sections commutes with passing to this colimit. But $\Gamma(C_n,\mathcal{O}_{\underline{Z}{}_n})$ has dimension $n$ for each $n$, and thus the colimit is not a finite dimensional $k$ vector space. In particular, $\LogO{X}$-modules with coherent restrictions to all log alterations do not form an abelian category. 
\end{example}

\subsection{Logarithmic quasi-coherent sheaves}\label{subsection on log quasi-coherent sheaves}
We will now introduce a notion of log quasi-coherence that generalizes log coherence and behaves similarly to quasi-coherence in conventional algebraic geometry. In particular, we show in Theorem \ref{theorem-compacted-generateness-of-quasi-coherent} that the category of log quasi-coherent sheaves is the smallest subcategory of $\text{Mod}(\LogO{X})$ containing all log coherent sheaves that is closed under small colimits.

\begin{definition}\label{Def:quasi-coherent}
    An $\mathcal{O}_{\mathcal{X}_{\text{l\'et}}}$-module $\mathcal{F}$ is said to be a \textit{logarithmic quasi-coherent sheaf} if for every log alteration $q\colon \mathcal{X}'\rightarrow \mathcal{X}$, the restriction $\Downarrow_{\cX'}\!\! \mathcal{F}$ is a conventional quasi-coherent sheaf. A morphism of logarithmic quasi-coherent sheaves is a morphism in the category of $\mathcal{O}_{\mathcal{X}_{\text{l\'et}}}$-modules. The category of log quasi-coherent sheaves is denoted $\logqcoh{\mathcal{X}}\subset \text{Mod}(\mathcal{O}_{\mathcal{X}_{\text{l\'et}}}).$ 
\end{definition}
\begin{remark}
    It follows from Corollary \ref{Key corollary} \eqref{item-1-key corollary} that every log quasi-coherent sheaf is of the shape $\underset{i}{\mathrm{colim}}\arrowup F_i$ where the $F_i$ are conventional quasi-coherent sheaves on various log alterations of $\cX$. By Theorem \ref{thm-pullback-static}, we may choose the $F_i$ to be static if $\cX$ is Noetherian.
\end{remark}
We provide alternative characterizations of log quasi-coherent sheaves.

\begin{proposition}\label{Prop-on-quasi-coherent-sheaves}
Let $\mathcal{X}$ be a DM $k$-stack with a log structure and $\mathcal{F}$ be an $\mathcal{O}_{\mathcal{X}_{\text{l\'et}}}$-module. The following statements are equivalent.
\begin{enumerate}
    \item \label{item-1-prop-quasi-coherent}
    The sheaf $\mathcal{F}$ is log quasi-coherent.
    \item\label{item-2-prop-quasi-coherent} For any log \'etale morphism $\mathcal{U}\rightarrow \mathcal{X}$, the restriction $\mathcal{F}|_{\ul{\mathcal{U}}{}_{\text{\'et}}}$ is a conventional quasi-coherent sheaf on $\underline{\mathcal{U}}$.
    \item\label{item-3-prop-quasi-coherent} For any strict morphism $V\rightarrow U$ in $\mathcal{X}_{\text{l\'et}}$ with both $U$ and $V$ affine, the canonical morphism
    $$\mathcal{F}(U)\otimes_{\mathcal{O}_{\ul{U}}(\ul{U})}\mathcal{O}_{\ul{V}}(\ul{V})\longrightarrow \mathcal{F}(V)$$
    is an isomorphism.
    \item\label{item-4-prop-quasi-coherent} There is a log \'etale cover $\{\mathcal{U}_i\rightarrow \mathcal{X}\}_{i\in I}$ such that the $\cF|_{\cU_i}$ are log quasi-coherent.
\end{enumerate}
\end{proposition}
\begin{proof}
The equivalence of \eqref{item-2-prop-quasi-coherent} and \eqref{item-3-prop-quasi-coherent} is standard \cite[0GZX]{stacks}. As \eqref{item-1-prop-quasi-coherent} is a special case of \eqref{item-2-prop-quasi-coherent}, it only remains to show \eqref{item-1-prop-quasi-coherent} $\Rightarrow$ \eqref{item-2-prop-quasi-coherent}. Let $\mathcal{U}\rightarrow \mathcal{X}$ be a log \'etale morphism. By locality of conventional quasi-coherence, we may assume that $\mathcal{U}$ is qc. Moreover, since $\cF$ is a log \'etale sheaf, it follows that for any log alteration $q\colon \mathcal{U}'\rightarrow \mathcal{U}$ we have $\ul{q}{}^{\text{\'et}}_*(\mathcal{F}|_{\ul{\mathcal{U}}{}'_{\text{\'et}}} )= \mathcal{F}|_{\ul{\mathcal{U}}{}_{\text{\'et}}}$ and so we may replace $\mathcal{U}$ by a log alteration. By Proposition \ref{proposition-alteration-to-saturated-morphisms}, we can therefore assume $\mathcal{U}\rightarrow\mathcal{X}$ is the composite $\mathcal{U}\rightarrow\mathcal{X}'\rightarrow\mathcal{X}$ of a strict \'etale morphism followed by a log alteration $q\colon \cX'\rightarrow\cX$, in which case $\mathcal{F}|_{\ul{\mathcal{U}}{}_{\text{\'et}}} = (\arrowdown_{\cX'}\cF )|_{\ul{\cU}}$ is quasi-coherent as desired.

The locality statement \eqref{item-4-prop-quasi-coherent} holds if $\{\mathcal{U}_i\rightarrow \mathcal{X}\}_{i\in I}$ is a strict \'etale cover. In particular, we assume that $\mathcal{X}$ is qc and $\{\mathcal{U}_i\rightarrow \mathcal{X}\}_{i\in I}$ is a finite log \'etale cover with $\mathcal{U}_i$ qc. Thus, we may assume that there is a log alteration $\mathcal{X}'\rightarrow \mathcal{X}$ so that we can factor $\mathcal{U}_i\rightarrow\mathcal{X}'\rightarrow\mathcal{X}$, where the first morphism is strict \'etale. Thus, $\{\mathcal{U}_i\rightarrow\mathcal{X}'\}_{i\in I}$ is a strict \'etale cover and so $\mathcal{F}|_{\mathcal{X}'}=\mathcal{F}$ satisfies \eqref{item-2-prop-quasi-coherent}-\eqref{item-1-prop-quasi-coherent}.
\end{proof}
We now show that the up functor preserves quasi-coherence. For this, we first establish that quasi-coherence is stable under log \'etale sheafification.
\begin{proposition}\label{prop:nifty}
Let $\mathcal{F}$ be a presheaf of $\mathcal{O}_{\mathcal{X}_{\text{l\'et}}}^{\text{pre}}$-modules such that for a strict morphism $V\rightarrow U$ between affines in $\mathcal{X}_{\text{\'et}}$, the morphism
\begin{equation}\label{equation:quasi-coherent}
\mathcal{F}(U)\otimes_{\mathcal{O}_{\ul{U}}(\ul{U})}\mathcal{O}_{\ul{V}}(\ul{V})\longrightarrow \mathcal{F}(V)    
\end{equation}
is an isomorphism, then the sheafification $\mathcal{F}^{\text{sh}}$ is log quasi-coherent. Moreover, it is enough to check that \eqref{equation:quasi-coherent} is an isomorphism under the additional assumption that $U\rightarrow \mathcal{X}$ factors through a member of a given log \'etale cover $\{\cU_i\rightarrow \mathcal{X}\}_{i\in I}$.
\end{proposition}
\begin{proof}
The last part follows from Proposition \ref{Prop-on-quasi-coherent-sheaves} \eqref{item-4-prop-quasi-coherent}. Hence, we may assume that \eqref{equation:quasi-coherent} holds for all affine $U$. The claim now follows from Corollary \ref{Key corollary} \eqref{item-3-key-corollary} and the fact that conventional quasi-coherent sheaves are closed under colimits.
\end{proof}
\begin{corollary}\label{key-corollary-quasi-coherent}
For any DM $k$-stack $\mathcal{X}$ with a log structure we have:
\begin{enumerate}
    \item\label{cor-1} Any finitely presented $\cO_{\cX_{\text{l\'et}}}$-module is log quasi-coherent.
    \item\label{cor-2} For any log alteration $q\colon \mathcal{X}'\rightarrow \mathcal{X}$, the up functor $\Uparrow\colon \text{Mod}(\mathcal{O}_{\ul{\mathcal{X}}{}'_{\text{\'et}}})\rightarrow\text{Mod}(\mathcal{O}_{\mathcal{X}_{\text{l\'et}}})$ sends conventional quasi-coherent sheaves to log quasi-coherent sheaves.
\end{enumerate}
\end{corollary} 
\begin{proof}
Part \eqref{cor-2} follows from Proposition \ref{prop:nifty} and the explicit description of $\Uparrow$ in Proposition \ref{prop: explicit description of up and down}. Part \eqref{cor-1} follows from \eqref{cor-2} and Corollary \ref{corr: FP iff up FP}.
\end{proof}
\begin{theorem}\label{thm-category-of-quasi-coherent-sheaves}
Let $\mathcal{X}$ be any DM $k$-stack with a log structure. Then the subcategory $\logqcoh{\mathcal{X}}\subset \text{Mod}(\mathcal{O}_{\mathcal{X}_{\text{l\'et}}})$ is closed under kernels, cokernels, extensions and small colimits.
\end{theorem}
\begin{proof}
Stability under kernels follows from the fact that $\arrowdown_{\mathcal{X}'}$ is left exact for any log alteration $\mathcal{X}'\rightarrow \mathcal{X}$ and that $\text{QCoh}(\underline{\mathcal{X}}')\subset \text{Mod}(\mathcal{O}_{\ul{\mathcal{X}}{}'_{\text{\'et}}})$ is closed under kernels.

For stability under extensions, let $0\rightarrow \mathcal{F}\rightarrow\mathcal{G}\rightarrow\mathcal{H}\rightarrow 0$ be a short exact sequence of $\mathcal{O}_{\mathcal{X}_{\text{l\'et}}}$-modules so that $\mathcal{F}$ and $\mathcal{H}$ are log quasi-coherent. Furthermore, let $f\colon U'\rightarrow U$ be a strict morphism between affines in $\mathcal{X}_{\text{l\'et}}$. Denoting by $U^{\text{aff}}_{\text{l\'et}}\subset U_{\text{l\'et}}$ the full subsite spanned by the affine objects, we have the following diagram: 
\begin{equation*}
\begin{tikzcd}[column sep=0.75cm]
            & \mathcal{F}|_{U_{\text{l\'et}}^{\text{aff}}}\otimes_{\mathcal{O}_{\ul{U}}(\ul{U})}\mathcal{O}_{\ul{U}'}(\ul{U}') \arrow[r] \arrow[d, "\cong"] & \mathcal{G}|_{U_{\text{l\'et}}^{\text{aff}}}\otimes_{\mathcal{O}_{\ul{U}}(\ul{U})}\mathcal{O}_{\ul{U}'}(\ul{U}') \arrow[r] \arrow[d] & \mathcal{H}|_{U_{\text{l\'et}}^{\text{aff}}}\otimes_{\mathcal{O}_{\ul{U}}(\ul{U})}\mathcal{O}_{\ul{U}'}(\ul{U}') \arrow[r] \arrow[d, "\cong"] & 0 \\
0 \arrow[r] & {(f_{\text{l\'et}}^{\text{aff}})_*(\mathcal{F}|_{U_{\text{l\'et}}^{\text{aff}}})} \arrow[r]                                                       & {(f_{\text{l\'et}}^{\text{aff}})_*(\mathcal{G}|_{U_{\text{l\'et}}^{\text{aff}}})} \arrow[r]                                              & {(f_{\text{l\'et}}^{\text{aff}})_*(\mathcal{H}|_{U_{\text{l\'et}}^{\text{aff}}})},                                                                 &  
\end{tikzcd}
\end{equation*}
where the upper row is right exact since the section-wise tensor product ${-}\otimes_{\mathcal{O}_{\ul{U}}(\ul{U})}\mathcal{O}_{\ul{U}'}(\ul{U}')$ is right exact and the lower row is left exact as $(f_{\text{l\'et}}^{\text{aff}})_*$ is left exact. For this, also note that ${-}\otimes_{\mathcal{O}_{\ul{U}}(\ul{U})}\mathcal{O}_{\ul{U}'}(\ul{U}')$ sends sheaves to sheaves since $\mathcal{O}_{\ul{U}'}(\ul{U}')$ is flat over $\mathcal{O}_{\ul{U}}(\ul{U})$ and all covers in $U^\text{aff}_{\text{l\'et}}$ are finite. Since the left and right vertical morphisms are isomorphisms, it thus follows that both rows are in fact short exact sequences and hence
$$\mathcal{G}|_{U_{\text{l\'et}}^{\text{aff}}}\otimes_{\mathcal{O}_{\ul{U}}(\ul{U})}\mathcal{O}_{\ul{U}'}(\ul{U}')\longrightarrow (f_{\text{l\'et}}^{\text{aff}})_*(\mathcal{G}|_{U_{\text{l\'et}}^{\text{aff}}})$$
is an isomorphism. Thus, $\mathcal{G}$ is log quasi-coherent.

For stability under colimits (and hence also cokernels), we note that the assumption of Proposition \ref{Prop-on-quasi-coherent-sheaves}
holds for the section-wise colimit, which implies the claim.
\end{proof}
In conventional algebraic geometry, the category of quasi-coherent sheaves often coincides with the ind-completion to the category of finitely presented sheaves. The same is true of log quasi-coherent sheaves. 
\begin{theorem}\label{theorem-compacted-generateness-of-quasi-coherent}
Let $\mathcal{X}$ be a qcqs DM $k$-stack with a log structure, then we have
 $$\logqcoh{\mathcal{X}} = \operatorname{Ind}(\text{FP}(\LogO{\cX})).$$
 In particular, every log quasi-coherent sheaf on $\mathcal{X}$ can be written as a filtered colimit of finitely presented sheaves. Moreover, a log quasi-coherent sheaf is compact if and only if it is finitely presented. 
\end{theorem}
\begin{proof}
By \cite[0GN0, 0GMS (3)]{stacks} and \cite[Lemma 3.14 (1)]{Log-etale-cohomology-Chikara-Nakayama}, it follows that all finitely presented $\mathcal{O}_{\mathcal{X}_{\text{l\'et}}}$-modules are compact objects. Hence, it follows from \cite[Propositions I.8.5.1, I.8.7.3, and I.8.7.5 (a)]{SGA4} that the ind-completion $\text{Ind}(\text{FP}(\LogO{\cX}))\subset \logqcoh{\mathcal{X}}$ is a full subcategory that is closed under filtered colimits. 

To show that every object is contained in the ind-completion, by Corollary \ref{Key corollary} \eqref{item-1-key corollary}, it suffices to consider quasi-coherent sheaves pulled back from a log alteration of $\mathcal{X}$, for which the claim follows from \cite[Corollary 4.11]{rydh2015Noetherian}. The last claim follows from the fact that $\text{FP}(\LogO{\cX})$ is idempotent closed and \cite[Exercise I. 8.7.8]{SGA4}. 
\end{proof}
\begin{remark}
If we further assume $\mathcal{X}$ is Noetherian, $\logqcoh{\mathcal{X}}$ is the ind-completion of the small abelian category $\logcoh{\mathcal{X}}$ and is therefore a Grothendieck abelian category. Adapting Gabber's argument~\cite[077K]{stacks}, we expect this result to hold for any DM $k$-stack with a log structure. As we do not need this fact, we will not elaborate on this further.
\end{remark}
\subsection{Pushforward, pullback and their derived functors}
In this section, we show that log quasi-coherent sheaves are stable and well-behaved under derived pushforward and pullback.
\subsubsection{Derived pushforward}
\begin{proposition}\label{prop: derived push log}
    Let $f\colon \mathcal{X}\rightarrow\mathcal{Y}$ be a qcqs morphism of DM $k$-stacks with log structures.
    \begin{enumerate}
        \item\label{item: derived push is qcoh} The derived pushforwards $R^nf^\text{l\'et}_*\colon \text{Mod}(\mathcal{O}_{\mathcal{X}_{\text{l\'et}}})\rightarrow \text{Mod}(\mathcal{O}_{\mathcal{Y}_{\text{l\'et}}})$ preserve log quasi-coherence.
        \item\label{item: log coherence push forward} If $\mathcal{X}$ and $\mathcal{Y}$ are locally Noetherian, $\text{char}(k)=0$, $f$ is proper and $\mathcal{F}$ is log coherent on $\mathcal{X}$, then all derived pushforwards $R^nf^\text{l\'et}_*\mathcal{F}$ are log coherent.
        \item\label{item: computing derived pushfoward} Assume that $f$ is saturated, separated and concentrated. If $F$ is a static quasi-coherent sheaf on $\underline{\mathcal{X}}$ and $R^{i}\underline{f}{}_*F$ is static for all $i\geq n$, then we have $R^{i}f^{\text{l\'et}}_*\arrowup F =\arrowup R^{i}\underline{f}{}_*F $ for all $i\geq n-1$. Moreover, for any log flat $q\colon\mathcal{Y}'\rightarrow\mathcal{Y}$ and corresponding base change diagram 
        \begin{equation}\label{diagram: cartesian log alteration shit for pushforward fuck}
            \begin{tikzcd}
                \mathcal{X}'\dar["q'"] \rar["f'"] &\mathcal{Y}'\dar["q"]\\
                \mathcal{X}\rar["f"]&\mathcal{Y}
            \end{tikzcd}
        \end{equation}
        we have $R^{i}\underline{f}{}'_* (\underline{q}')^*F = \underline{q}^*R^{i}\underline{f}{}_*F$ for all $i\geq n-1$.
    \end{enumerate}
\end{proposition}

\begin{remark}\label{remark: how to compute higher derived push}
    \begin{enumerate}
        \item Recall that a morphism $\underline{f}\colon \underline{\mathcal{X}}\rightarrow\underline{\mathcal{Y}}$ is concentrated if it is qcqs and for any base change $\underline{f}'\colon \underline{\mathcal{X}}\times_\mathcal{\underline{Y}} \underline{Y}\rightarrow \underline{Y}$ along a morphism $\underline{Y}\rightarrow\underline{\mathcal{Y}}$ out of an affine scheme $\underline{Y}$, we have $R^{i}\underline{f}{}'_*=0$ for $i\gg 0$ \cite[Definition 2.4]{HallRydh2017}. In particular, if $\ul{f}$ is qc and separated, this last assumption holds if $\underline{f}$ is representable, $f$ is a root stack, $\text{char}(k)=0$, or $\cX$ is tame \cite[Theorem 2.1]{HallRydhWhatStacksAreConcenctrated}.
        \item Note that Proposition \ref{prop: derived push log} \eqref{item: computing derived pushfoward} gives us an explicit way to calculate the derived push-forward of a log coherent sheaf $\mathcal{F}$ along a proper morphism $f\colon \mathcal{X}\rightarrow\mathcal{Y}$ between $\mathcal{X}$ and $\mathcal{Y}$ Noetherian if $\text{char}(k)=0$. Indeed, by Proposition \ref{proposition-alteration-to-saturated-morphisms}, Theorem \ref{thm-pullback-static} and Proposition \ref{proposition on finite presentation} we may replace $\mathcal{X}$ and $\mathcal{Y}$ by log alterations so that $f$ becomes saturated and $\mathcal{F}=\arrowup F$ for some static coherent sheaf $F$ on $\underline{\mathcal{X}}$. Since $R^{i}\underline{f}{}_*F=0$ for $i>N$ for some $N\gg0$, we get $R^N f^{\text{l\'et}}_*\mathcal{F} = \arrowup R^N\underline{f}{}_*F$. Moreover, by the base change result in Proposition \ref{prop: derived push log} \eqref{item: computing derived pushfoward}, we can replace $\mathcal{Y}$ by a further log modification so that $R^{N}f^{\text{l\'et}}_*F$ becomes static, which implies $R^{N-1}f^{\text{l\'et}}_*\mathcal{F} = \arrowup R^{N-1}\underline{f}{}_*F$ and so on.
    \end{enumerate}
\end{remark}
\begin{proof}[Proof of Proposition \ref{prop: derived push log}]
    The first claim follows directly from Proposition \ref{key proposition} \eqref{item-1-key proposition} and the fact that the strict \'etale derived pushforward preserves quasi-coherence. The algorithm of Remark \ref{remark: how to compute higher derived push} also shows that \eqref{item: log coherence push forward} follows from \eqref{item: computing derived pushfoward}. To show \eqref{item: computing derived pushfoward}, it suffices to show the log flat base change claim. To see this, recall from Proposition \ref{key proposition} \eqref{item-1-key proposition} that we have 
    \begin{equation}\label{equation: shitty fuckfuck}
        R^{i}f^\text{l\'et}_*\arrowup F = \mathop{\text{colim}}_{f'\colon\mathcal{X}'\rightarrow\mathcal{Y}'} \arrowup R^{i}\underline{f}{}'_*(\underline{q}')^*F,
    \end{equation}
    where the colimit runs over all not necessarily cartesian squares
    \begin{equation*}
        \begin{tikzcd}
            \mathcal{X}'\dar["q'"]\rar["f'"] &\mathcal{Y}'\dar["q"]\\ \mathcal{X} \rar["f"]&\mathcal{Y}
        \end{tikzcd}
    \end{equation*}
    with the vertical maps being log alterations. Note that we can factor any such $f'$ into a morphism $g\colon \mathcal{X}'\rightarrow\mathcal{X}\times_\mathcal{Y}\mathcal{Y}'$, which is a log alteration by Proposition \ref{proposition: basic properties of log alterations} \eqref{item: cancellation property log alterations}, followed by a base change $\pi_{\mathcal{Y}'}\colon \mathcal{X}\times_\mathcal{Y}\mathcal{Y}'\rightarrow\mathcal{Y}'$ of $f$. Assuming that log flat base change holds in degree $i$, see that 
    \[
        R^{i}\underline{f}{}'_*(\underline{q}')^*F = \mathcal{H}^{i}(R\underline{\pi_{\mathcal{Y}'}}{}_*R\underline{g}{}_*\underline{g}^*(\underline{\pi_{\mathcal{X}}})^*F) = \mathcal{H}^{i}(R\underline{\pi_{\mathcal{Y}'}}{}_*(\underline{\pi_{\mathcal{X}}})^*F) = R^{i}\underline{\pi_{\mathcal{Y}'}}{}_*(\underline{\pi_{\mathcal{X}}})^*F = \underline{q}^*R^{i}\underline{f}{}_*F,
    \]
    where $\pi_{\mathcal{X}}\colon \mathcal{X}\times_\mathcal{Y}\mathcal{Y}'\rightarrow\mathcal{X}$ is the projection to the first factor and in the second equality we used staticity of $F$ and Proposition \ref{proposition-on-derived-pullback-vanishing-static} \eqref{item: static push pull is identity}.

    Since the log flat base change assertion is strict \'etale local in $\mathcal{Y}$, it suffices to prove it in the case that $\underline{\mathcal{Y}}= \text{Spec}(A)$ is affine and hence $\underline{\mathcal{X}}$ is qc, separated and concentrated over $\underline{\mathcal{Y}}$. We would now like to deduce the base change claim by applying Proposition \ref{prop: static complex prop} to the \v{C}ech complex associated to $F$ and an affine cover of $\ul{\cX}$. However, Proposition \ref{prop: static complex prop} does not apply as the \v{C}ech complex is not bounded on the right. Therefore, we need to use the \textit{alternating} \v{C}ech complex instead. However, since the alternating \v{C}ech complex only exists for algebraic spaces, we will first show that $\underline{\mathcal{X}}$ admits a qc and separated good moduli space on which we can carry out the desired argument. Indeed, it follows from \cite[Theorem C]{HallRydhWhatStacksAreConcenctrated} that $\underline{\mathcal{X}}$ has tame inertia. Since $\underline{\mathcal{X}}$ is concentrated, it moreover follows from \cite[Theorem D]{rydh2015Noetherian} that we can find an affine morphism $\underline{\mathcal{X}}\rightarrow\underline{\cX}{}_0$ to a DM stack $\underline{\cX}{}_0$ that is separated and of finite type over $\text{Spec}(k)$ and also has tame inertia. By \cite[Theorem 3.2]{TameStacks}, there is a coarse and good moduli space $\underline{\cX}{}_0\rightarrow \underline{X}{}_0$, where $\underline{X}{}_0$ is qc and separated by \cite[Theorem 1.1]{KeelMori}. Hence, the composite $\underline{p}\colon \underline{\cX}\rightarrow \underline{\cX}{}_0\rightarrow \underline{X}{}_0$ is cohomologically affine and $\underline{\pi}\colon \underline{\cX}\rightarrow \underline{X}$ is a good moduli space with $\underline{X}= \text{Spec}_{\underline{X}{}_0}(\underline{p}{}_*\cO_{\underline{\cX}} )$ qc and separated.
    
    Now let $F$ be a static quasi-coherent sheaf as above and  $\underline{U}\twoheadrightarrow \underline{X}$ an \'etale cover by an affine scheme $\ul{U}$. By the discussion in \cite[0721]{stacks}, we have the alternating \v{C}ech complex $\check{\mathcal{C}}^\bullet _{\text{alt}}(\underline{U}, \underline{\pi}{}_* F)$, whose components are of the shape $\check{\mathcal{C}}^{p} _{\text{alt}}(\underline{U}, \underline{\pi}{}_* F) = \Gamma(\underline{U}{}_p,\underline{\pi}{}_*F|_{\underline{U}{}_p}\otimes_{\underline{\mathbb{Z}}_{\text{\'et}}} \text{L}_p)$. Here we set $\underline{U}{}_p=\underline{W}{}_p/S_{p+1}$ with $\underline{W}{}_p \subset \underline{U}^{\times_{\underline{X}} (p+1)}$ the closed subscheme given by the complement of all diagonals and $\text{L}_p$ is an invertible $\underline{\mathbb{Z}}_{\text{\'et}}$-module on $\underline{U}{}_p$ whose construction depends only on $p$. Letting $\underline{\cU}{}_p=\underline{U}{}_p\times_{\underline{X}}\underline{\cX}$, we have that $F|_{\underline{\cU}{}_p}\otimes_{\underline{\mathbb{Z}}_{\text{\'et}}} \text{L}_p$ is static and quasi-coherent as it is locally isomorphic to $F|_{\underline{\cU}{}_p}$. By Lemma \ref{lemma-on-pushforward-along-integral-affine-preserves-log-Tor dimensions}, this implies that $\check{\mathcal{C}}^{p} _{\text{alt}}(\underline{U}, \underline{\pi}{}_* F) = \Gamma(\underline{\cU}{}_p,F|_{\underline{\cU}{}_p}\otimes_{\underline{\mathbb{Z}}_{\text{\'et}}} \text{L}_p)$ is static for all $p$ when viewed as a quasi-coherent sheaf on $\ul{\cY}$. Moreover, by \cite[03JA]{stacks}, there is a global bound on the size of the fibers of $\underline{U}\twoheadrightarrow\underline{X}$, which implies $\underline{U}{}_p = \emptyset$ for $p\gg 0$ and thus $\check{\mathcal{C}}^\bullet _{\text{alt}}(\underline{U}, \underline{\pi}{}_* F)$ is a bounded complex. This puts us in the situation of Proposition \ref{prop: static complex prop} and so it remains to show that for any log flat $\mathcal{Y'}\rightarrow\mathcal{Y}$ with $\underline{\mathcal{Y}}' = \text{Spec}(B)$
    the term-wise base change $\check{\mathcal{C}}^\bullet _{\text{alt}}(\underline{U}, \underline{\pi}{}_* F)\otimes_A B$ computes the sheaf cohomology of $(\underline{q}')^*F$ with $q'$ as in \eqref{diagram: cartesian log alteration shit for pushforward fuck}. Indeed, it follows from the construction of the alternating \v{C}ech complex and affine base change \cite[07U8]{stacks} that $\check{\mathcal{C}}^\bullet _{\text{alt}}(\underline{U}, \underline{\pi}{}_* F)\otimes_A B = \check{\mathcal{C}}^\bullet _{\text{alt}}(\underline{U}\times_{\underline{\mathcal{Y}}}\ul{\mathcal{Y}}', \pi_{\underline{X}}^*\underline{\pi}{}_* F)$, where $\pi_{\underline{X}}\colon \underline{X}\times_{\underline{\mathcal{Y}}}\underline{\cY}'\rightarrow \ul{X}$ is the natural projection. Moreover, the base change $\ul{\pi}'\colon \ul{\cX}' = 
    \underline{\mathcal{X}}\times_{\ul{\cY}} \ul{\cY}'\rightarrow \ul{X}\times_{\ul{\cY}} \ul{\cY}'$ is by \cite[Proposition 4.7]{Alper2013GoodModuli} again a good moduli space and $\pi_{\ul{X}}^*\ul{\pi}{}_*F = \ul{\pi}{}'_*(\ul{q}')^*F$. Note here that $\ul{\cX}' = \ul{\cX}\times_{\ul{\cY}}\ul{\cY}'$ holds since $f$ is saturated. Since $\underline{X}\times_{\ul{\cY}}\ul{\cY}'$ is separated and the $(\underline{U}{}\times_{\ul{\cY}}\ul{\cY}')_p$ are affine, \cite[0728]{stacks} implies that $H^{i}(\check{\mathcal{C}}^\bullet_{\text{alt}}(\underline{U}\times_{\underline{\cY}}\underline{\cY}', \underline{\pi}{}'_* (\underline{q}')^*F)) = H^{i}(\underline{X}\times_{\ul{\cY}}\ul{\cY}',\underline{\pi}{}'_* (\underline{q}')^*F) = H^{i}(\underline{\cX}',(\underline{q}')^* F)$. Here, the last equality follows from $R\underline{\pi}{}'_*(\underline{q}')^*F=\underline{\pi}{}'_*(\underline{q}')^*F$, which is true by \cite[Remark 3.5]{Alper2013GoodModuli} and the fact that $\ul{\cX}'$ and $\ul{X}\times_{\ul{\cY}}\ul{\cY}'
    $ are separated. 
\end{proof}
The following examples illustrate that the assumptions of Proposition \ref{prop: derived push log} are necessary.
\begin{example}
 Consider $f\colon X = \text{Bl}_{0}\mathbb{A}^{2}_{k}\rightarrow Y = \mathbb{A}_{k}^{2}$ with both schemes equipped with toric log structures and $F=\mathcal{O}_{E}(-2)$ with $\mathbb{P}^{1}\cong E\subset \text{Bl}_{0}\mathbb{A}^{2}_k$ the exceptional divisor. It follows by direct computation that $F$ is static and 
    \[
        R^{n}\ul{f}{}_{*}F = 
        \begin{cases}
            \cO_0 & \text{if } n =1 \\
            0 & \text{otherwise,} 
        \end{cases}
    \]
    where $\cO_0$ denotes the skyscraper sheaf supported on $0$. However, we have $R^{n}f^{\text{l\'et}}_{*}\arrowup F =0$ for $n\geq 1$ as $f^{\text{l\'et}}$ is an equivalence of ringed topoi and $\arrowdown_{X} \! f^{\text{l\'et}}_{*}\arrowup F=F\neq 0$. Thus $\arrowup R^{n}\ul{f}{}_{*}^{\text{\'et}}F\neq R^{n}f^{\text{l\'et}}_{*}\mathcal{F}$ for $n=0,1$.
\end{example}

\subsubsection{Derived pullback} Our discussion of functoriality concludes by checking that left derived pullback satisfies the expected properties. 
\begin{proposition}\label{prop: pullback}
    Let $f\colon \mathcal{X}\rightarrow\mathcal{Y}$ be a morphism of DM $k$-stacks with log structures.
    
    \begin{enumerate}
        \item\label{item: derived pull preserves qcoh} The derived pullback $L_if_\text{l\'et}^*\colon \text{Mod}(\mathcal{O}_{\mathcal{Y}_{\text{l\'et}}})\rightarrow \text{Mod}(\mathcal{O}_{\mathcal{X}_{\text{l\'et}}})$ preserves log quasi-coherence, and if $\mathcal{X}$ and $\mathcal{Y}$ are locally Noetherian, it also preserves log coherence for all $i\geq 0$.
        \item\label{item: derived pull description if static} Assume that $\mathcal{X}$ and $\mathcal{Y}$ are static and $F$ is a static quasi-coherent sheaf on $\mathcal{Y}$ so that the derived pullbacks $L_if_\text{\'et}^*F$ are static for all $i\leq n$. Then for all $i\leq n+1$ we have $\arrowup L_i \underline{f}{}_{\text{\'et}}^*F = L_i f^* \arrowup F$ and for any commutative diagram 
        \begin{equation*}
            \begin{tikzcd}
        \mathcal{X}'\dar["p_\mathcal{X}"] \rar["f'"] &\mathcal{Y}'\dar["p_\mathcal{Y}"]\\ \mathcal{X} \rar["f"]&\mathcal{Y}
            \end{tikzcd}
        \end{equation*}
        where the vertical maps are log flat, we have $ (\underline{p_\mathcal{X}})_\text{\'et}^* L_i \ul{f}{}_\text{\'et}^*F = L_i (\underline{f}')_\text{\'et}^*(\underline{p_\mathcal{Y}})_\text{\'et}^* F$.
    \end{enumerate}
\end{proposition}

\begin{proof}
    The preservation of log quasi-coherence in \eqref{item: derived pull preserves qcoh} follows from Corollary \ref{Key corollary}\eqref{item-4-key-corollary}. The claim about log coherence is that log \'etale is local in $\mathcal{Y}$ and hence follows from \eqref{item: derived pull description if static} and Theorem \ref{thm-pullback-static} by successively statifying the ordinary higher pullbacks of some static ordinary coherent sheaf that gives rise to the chosen log coherent sheaf.
    
    For \eqref{item: derived pull description if static}, we first note that Proposition \ref{proposition-on-derived-pullback-vanishing-static}\eqref{item: static all derived pulls vanish} implies $L_i(\ul{f}')_{\text{\'et}}^*(\ul{p_\mathcal{Y}})_\text{\'et}^* F = L_i(\ul{p_\mathcal{Y}}\circ \ul{f}')_\text{\'et}^* F = L_i(\ul{f}\circ \ul{p_\mathcal{X}})_\text{\'et}^* F$. Moreover, there is a Grothendieck spectral sequence $E^2_{p,q} = L_p (\ul{p_\mathcal{X}})_\text{\'et}^* L_q\ul{f}_\text{\'et}^*F\implies L_{p+q}(\ul{f}\circ \ul{p_\mathcal{X}})_{\text{\'et}}^* F$ with $E^2_{p,q}=0$ for all $p$ and $q$ with $p>0$ and $0<q\leq n$, which implies $(p_\mathcal{X})_\text{\'et}^* L_i \ul{f}_\text{\'et}^*F = L_i(\ul{f}\circ \ul{p_\mathcal{X}})_\text{\'et}^*F = L_i (\ul{f}')_\text{\'et}^*(\ul{p_\mathcal{Y}})^*_\text{\'et} F$ for all $i\leq n+1$ as desired. An analogous argument also proves the claim about $\Uparrow$.
\end{proof}

\subsection{\texorpdfstring{$\mathcal{E}xt$}{Ext} sheaves} 
In this section we study sheafy $\operatorname{\mathcal{E}xt}$: the right derived functors of $$\mathcal{H}om(\mathcal{F},-)\colon \operatorname{Mod} (\mathcal{O}_{\mathcal{X}_\text{l\'et}})\longrightarrow \operatorname{Mod} (\mathcal{O}_{\mathcal{X}_\text{l\'et}})$$ for $\mathcal{F}$ a log coherent sheaf. Since $\operatorname{Mod} (\mathcal{O}_{\mathcal{X}_\text{l\'et}})$ has enough injectives, these functors are well defined, and by \cite[TAG 06XR]{stacks} whenever there is a locally free resolution $P^\bullet \rightarrow \mathcal{F}$, there is an equality $$\operatorname{\mathcal{E}xt}^i(\mathcal{F},\mathcal{G})
=\mathcal{H}^i\!\bigl(\mathcal{H}om_{\mathcal O_X}(P^\bullet,\mathcal{G})\bigr).$$

\begin{proposition}\label{prop: ext}
    Let $\mathcal{X}$ be a locally Noetherian DM $k$-stack with log structure.
    \begin{enumerate}
        \item\label{item: sheafy Ext is qcoh} For any log coherent sheaf $\mathcal{F}$ on $\mathcal{X}$, the sheafy Ext functors $\mathcal{E}xt^{i}_{\mathcal{X}}(\mathcal{F},-)\colon \text{Mod}(\mathcal{O}_{\mathcal{X}_\text{l\'et}})\rightarrow \text{Mod}(\mathcal{O}_{\mathcal{X}_\text{l\'et}})$ preserve log (quasi-)coherence for all $i\geq 0$.
        \item\label{item: how to compute sheafy Ext} Let $\mathcal{X}$ be static, $F$ a static coherent sheaf of finite Tor dimension and $G$ a static quasi-coherent sheaf on $\mathcal{X}$ so that $\mathcal{E}xt^{i}_{\underline{\mathcal{X}}}(F,G)$ is static for all $i\geq n$. Then we have $\mathcal{E}xt^{i}_{\mathcal{X}_{\text{l\'et}}}(\arrowup F,\arrowup G) = \arrowup \mathcal{E}xt^{i}_{\underline{\mathcal{X}}}(F,G)$ and $\mathcal{E}xt^{i}_{\underline{\mathcal{X}}'}(\underline{f}^* F,\underline{f}^* G) = \underline{f}^* \mathcal{E}xt^{i}_{\underline{\mathcal{X}}}(F,G)$ for any log flat morphism $f\colon \mathcal{X}'\rightarrow\mathcal{X}$ and $i\geq n-1$.
    \end{enumerate}
\end{proposition}

\begin{proof}
    The first assertion is log \'etale local in $\mathcal{X}$. Thus we may assume that $\mathcal{X}$ is affine and static and $\mathcal{F}$ comes from a static coherent sheaf on $\mathcal{X}$, any finite free resolution of which will by Proposition \ref{prop: static complex prop} give rise to a finite free resolution $\mathcal{P}^\bullet\rightarrow\mathcal{F}$ and so $\mathcal{E}xt^{i}_{\LogO{\mathcal{X}}}(\mathcal{F},-) = \mathcal{H}^{i}(\mathcal{H}om_{\LogO{\mathcal{X}}}(\mathcal{P}^{-\bullet},-))$ preserves log (quasi-)coherence.
    
    Part \eqref{item: how to compute sheafy Ext} is strict \'etale local in $\mathcal{X}$. Indeed, there are natural morphisms 
    \begin{equation}\label{equation: random equation of doom, death and destruction}
        \underline{f}^*\mathcal{E}xt^{i}_{\underline{\mathcal{X}}}(F,G)\longrightarrow \mathcal{H}^{i}(L\underline{f}^*R\mathcal{H}om_{\underline{\mathcal{X}}}(F,G))\longrightarrow \mathcal{H}^{i}(R\mathcal{H}om_{\underline{\mathcal{X}}'}(L\underline{f}^*F,L\underline{f}^*G)) =  \mathcal{E}xt^{i}_{\underline{\mathcal{X}}'}(\underline{f}^*F,\underline{f}^*G),
    \end{equation} 
    where the first arrow is the canonical morphism from the pullback of the cohomology to the cohomology of the pullback and the second arrow comes from \cite[08JF]{stacks}. The final equality uses that $L\underline{f}^*F = F$ and $L\underline{f}^*G = G$ by Proposition \ref{proposition-on-derived-pullback-vanishing-static}. To check that this is an isomorphism is now a local problem. We may thus assume that $F$ has a bounded finite free resolution $P^\bullet \rightarrow F$, which by Proposition \ref{prop: static complex prop} implies that $\underline{f}^*P^\bullet \rightarrow \underline{f}^*F$ is also a bounded finite free resolution. As a result, we have $\mathcal{E}xt^{i}_{\underline{\mathcal{X}}'}(\underline{f}^*F,\underline{f}^*G) = \mathcal{H}^{i}(\mathcal{H}om_{\underline{\mathcal{X}}'}(\underline{f}^*P^{-\bullet},\underline{f}^*G)) = \mathcal{H}^{i}(\underline{f}^*\mathcal{H}om_{\underline{\mathcal{X}}}(P^{-\bullet},G))$. One can now check that the isomorphisms $\mathcal{E}xt^{i}_{\underline{\mathcal{X}}'}(\underline{f}^* F,\underline{f}^* G) \cong \underline{f}^* \mathcal{E}xt^{i}_{\underline{\mathcal{X}}}(F,G)$ for $i\geq n-1$ due to Proposition \ref{prop: static complex prop} applied to $E^\bullet = \mathcal{H}om_{\underline{\mathcal{X}}}(P^{-\bullet},G)$  are indeed restrictions of \eqref{equation: random equation of doom, death and destruction}. The claim about $\Uparrow$ is proven analogously.
\end{proof}
\subsection{Sheafy Tor}
\begin{proposition}\label{prop: tor}
    Let $\mathcal{X}$ be a DM $k$-stack with a log structure. Then:
    \begin{enumerate}
        \item\label{item: Tor preserves quasi-coherence} For all $i\geq 0$, the Tor functor $\mathcal{T}or_i^{\cX}(-,-)\colon \text{Mod}(\LogO{\mathcal{X}})\times\text{Mod}(\LogO{\mathcal{X}})\rightarrow \text{Mod}(\LogO{\mathcal{X}})$ preserves quasi-coherence and if $\mathcal{X}$ is locally Noetherian, then it preserves coherence as well.
        \item\label{item: compute Tor} Assume that $\mathcal{X}$ is static and $F$ and $G$ are static conventional quasi-coherent sheaves so that $\mathcal{T}or_i^{\underline{\cX}}(F,G)$ is static for all $i\geq n$. Then for any log flat $f\colon \mathcal{X}'\rightarrow \mathcal{X}$ we have $\mathcal{T}or_i^{\underline{\cX}}(\underline{f}^* F,\underline{f}^* G) = \underline{f}^*\mathcal{T}or_i^{\underline{\cX}}(F,G)$ and $\mathcal{T}or_i^{\cX}(\arrowup F,\arrowup G) = \Uparrow\!\!\mathcal{T}or_i^{\underline{\cX}}(F,G)$ for all $i\geq n-1$. 
    \end{enumerate}
\end{proposition}
\begin{proof}
    Claim \eqref{item: compute Tor} is strict \'etale local, hence we may assume that $\mathcal{X}$ is affine and thus we can find a bounded free resolution $P^\bullet\rightarrow F$, which by Proposition \ref{prop: static complex prop} implies that $\underline{f}^*P^{\bullet}\rightarrow\underline{f}^*F$ is also a free resolution and hence $$\mathcal{T}or_i^{\underline{\mathcal{X}}'}(\underline{f}^* F,\underline{f}^* G) = \mathcal{H}^{-i}(\underline{f}^*P^\bullet\otimes_{\mathcal{O}_{\underline{\mathcal{X}}'}}\underline{f}^*G) = \underline{f}^*\mathcal{H}^{-i}(P^\bullet\otimes_{\mathcal{O}_{\underline{\mathcal{X}}}}G) = \underline{f}^*\mathcal{T}or_i^{\underline{X}}(F,G)$$ for all $i\geq n-1$, where the second equality used Proposition \ref{prop: static complex prop}. The same proof also works for $\Uparrow$. Together with Theorem \ref{thm-pullback-static}, this implies the coherence claim in \eqref{item: Tor preserves quasi-coherence}. For quasi-coherence, let $\mathcal{F}$ and $\mathcal{G}$ be log quasi-coherent sheaves on $\cX$. By Corollary \ref{Key corollary}\eqref{item-1-key corollary} we may assume that $\mathcal{F} = \arrowup F$ for $F$ conventional quasi-coherent on $\underline{X}$. Since the claim is strict \'etale local in $\mathcal{X}$, we may further assume that there is a surjection $\phi\colon \cO_{\underline{\cX}}^{\oplus I}\twoheadrightarrow F$ and thus $\Uparrow\!\!\phi\colon\mathcal{O}_\cX^{\oplus I}\twoheadrightarrow \mathcal{F}$ for some set $I$. It follows that $\cT or_{i+1}^\cX(\cF,\cG) = \cT or_{i}^\cX(\text{Ker}(\arrowup\phi),\cG)$ for all $i\geq 1$ and $\cT or_{1}^\cX(\cF,\cG) = \text{Ker}(\text{Ker}(\arrowup\phi)\otimes_{\LogO{\cX}}\cG\rightarrow \cG)$, which reduces us to showing that $\logqcoh{\mathcal{X}}$ is closed under $-\otimes_{\cO_\cX}-$. This follows from the fact that $\cF\otimes_{\LogO{X}}\cG$ is the sheafification of the presheaf that assigns $\cF(\cU)\otimes_{\LogO{X}(\cU)}\cG(\cU)$ to any log \'etale neighborhood $\cU\rightarrow \cX$, which satisfies the hypotheses of Proposition \ref{prop:nifty}.
\end{proof}
\section{Logarithmic Picard spaces}\label{section:log picard groups}
We study the moduli stack of locally free rank one log coherent sheaves on a proper, vertical, saturated, and log smooth family of curves $\pi \colon C\rightarrow S$ with $S$ a Noetherian log $k$-scheme where $k$ is an algebraically closed field of characteristic zero. 
After log Quot spaces \cite{kennedy2023logarithmic}, this provides the second known example of a moduli space of log coherent sheaves.
We make the following additional assumptions for ease of exposition. 
\begin{enumerate}
    \item The base $S$ admits a global chart. 
    \item The strict locus of $\pi$,
    $$\pi^\circ\colon C^\circ\hookrightarrow C \longrightarrow S$$ is a surjective map which admits a global section. 
\end{enumerate}
We note that these are weak assumptions. Indeed, since any log scheme admits charts strict \'stale locally,  Assumption (1) always holds strict stale locally. Moreover, Assumption (2) is true strict \'etale locally on $S$ after replacing $S$ by the image of $\pi$ \cite[055U]{stacks}.

Unlike the Quot scheme, the moduli of locally free log coherent sheaves of rank one is non-separated\footnote{In the sense of \cite[Example 2.2.5]{MolchoWise}} and has everywhere positive dimensional gerbe structure. The main result of this section explains how the combinatorics of chip firing, and thus the log Picard groups studied elsewhere \cite{MolchoWise, AaronLogPic}, arise from the structure of log coherent sheaves. Whilst our discussion assumes the vertical case, we expect that related statements hold in the case studied by Slipper \cite{AaronLogPic}. 

\begin{remark}
    The hypothesis that $C$ has relative dimension one over $S$ does not play an essential role in our discussion. The higher dimensional theory of log Picard groups remains work in progress \cite{LogPicHigherDim}. 
\end{remark}

\subsection{Three flavours of logarithmic Picard} 
    In conventional algebraic geometry, the \textit{Picard stack} is the moduli stack of invertible (conventional) coherent sheaves on a curve, and it is a $\mathbb{G}_m$-gerbe over its sheaf of isomorphism classes, which we call the \textit{Picard group}.

    In the log situation, we will study three related moduli stacks. See Remark~\ref{rem:OGfunctor} for a way to think about the connection between these moduli problems. 

\subsubsection{The moduli stack of log invertible sheaves} Parallel to conventional algebraic geometry, a \emph{log invertible sheaf} is a log coherent sheaf which is locally free of rank one. 
For an $S$ log scheme $T\rightarrow S$, we will write $C_T$ for the fibre product $C\times_S T$.
\begin{definition}
    Let $\mathfrak{M}_1(C/S)$ be the CFG over the big log \'etale site of $S$ whose fibre over $T\rightarrow S$ is the groupoid of log invertible sheaves on $C_T$. We call $\mathfrak{M}_1(C/S)$ the \emph{moduli stack of logarithmic invertible sheaves} on $C$.
\end{definition}

\begin{remark}
    Let $B\mathbb{G}_m$ be the classifying stack for strict \'etale $\mathbb{G}_m$-torsors. A log invertible sheaf on $C$ is the data of a map to the log \'etale sheafification of $B\mathbb{G}_m$. Indeed, it is a line bundle on a log alteration of $C$.
\end{remark}

\begin{lemma}\label{lem:MSatisfiesDescent}
    The CFG of log invertible sheaves $\mathfrak{M}_1(C/S)$ satisfies descent on the big full log \'etale site of $S$.
\end{lemma}

\begin{proof}
    This follows from descent of $\cO$-modules, and the fact that being invertible is a local condition.
\end{proof}

\subsubsection{The logarithmic Picard stack as moduli of $\mathbb{G}_m^\mathsf{log}$-torsors}\label{sec:defnGlog} 

For $T$ a scheme with a log structure we write $M_T$ for its associated sheaf of monoids. Recall (for example, from \cite[Definition~2.2.7.1]{MolchoWise}) that the log multiplicative group, denoted $\mathbb{G}_m^\mathsf{log}$, is defined through its functor of points on the category of log schemes: $$\operatorname{Hom}(T,\mathbb{G}_m^\mathsf{log}) = \Gamma(T,M_T^\mathsf{gp}).$$
Write $H^1\bigl(T, \mathbb{G}_m^\mathsf{log}\bigr)^{\dagger}\subset H^1\bigl(T, \mathbb{G}_m^\mathsf{log}\bigr)$ for the subgroup consisting of strict \'etale $\mathbb{G}_m^\mathsf{log}$-torsors which have \textit{bounded monodromy} \cite[Section 3.5]{MolchoWise}.

\begin{remark}
    There is a canonical inclusion $$\mathbb{G}_m \hookrightarrow \mathbb{G}_m^\mathsf{log} \quad x\mapsto \alpha^{-1}(x).$$
    Here $\alpha \colon M_X\rightarrow \mathcal{O}_X$ is the data of a log structure on $X$.
\end{remark}

\begin{remark}\label{remark: bounded monodromy}
We recall from \cite{MolchoWise} two ways of thinking about the bounded monodromy condition.
\begin{enumerate}
    \item The moduli of $\mathbb{G}_m^\mathsf{log}$ torsors with bounded monodromy forms a log abelian variety, see \cite[Example 3.6.4]{MolchoWise}. The bounded monodromy condition can be expressed in the language of tropical geometry \cite[Definition 3.5.3]{MolchoWise}.
    \item\label{item: M1 to log Pic is surjective kind of} A $\mathbb{G}_m^\mathsf{log}$-torsor $\tau$ on $C/S$ satisfies bounded monodromy if and only if log \'etale locally on $S$ there exists a log alteration $C'/S'$ such that $\tau$ pulls back to a $\mathbb{G}_m$-torsor. See \cite[Corollary~4.3.1]{MolchoWise}.
\end{enumerate}
The second perspective above is built into the stack of log invertible sheaves.
\end{remark}
Recall that a logarithmic line bundle on $C_T$ is defined to be a $\mathbb{G}_m^\mathsf{log}$-torsor with bounded monodromy \cite{MolchoWise}.
\begin{definition}
    Set $\mathbf{LogPic}^\mathsf{pre}({C/S})$ the CFG over the category of logarithmic $S$-schemes whose fibre over $T/ S$ is the groupoid of logarithmic line bundles on $C_T$. The \textit{logarithmic Picard stack} $\mathbf{LogPic}({C/S})$ the associated log \'etale stackification, and the \textit{log Picard group} $\mathsf{LogPic}(C/S)$ is its corresponding (log \'etale) sheaf of isomorphism classes.
\end{definition}

We restate a result of Molcho and Wise in the language of this paper.
\begin{theorem}[Molcho--Wise]\label{thm:MolchoWise}
    Let $T\rightarrow S$ be a morphism of log schemes with $T$ static. The sheafification morphism induces an isomorphism of groupoids 
    $$\operatorname{Hom}(T/S,\mathbf{LogPic}^\mathsf{pre}({C/S})) \longrightarrow \operatorname{Hom}(T/S,\mathbf{LogPic}({C/S}))$$
    Moreover, the log Picard group is locally of finite presentation over $S$, and its functor of points can be obtained as the log \'etale sheafification of the functor of points to a $S$-scheme with log structure.
\end{theorem}
\begin{proof}
    Sheafification induces an isomorphism because Proposition \ref{proposition-on-derived-pullback-vanishing-static} (\ref{item: static push pull is identity}) ensures that $S$ satisfies the hypothesis of \cite[Corollary 8.2]{molcho2023remarks}. Local finite presentation is \cite[Proposition~4.2.2]{MolchoWise}.

    We now show that the log Picard group admits a log modification (in the sense of \cite[Definition 2.2.6.1]{MolchoWise}) which is a scheme with log structure. It suffices to work connected component by connected component. The subscheme of the logarithmic Picard group parameterising log line bundles of total degree zero is a (polariseable) log abelian variety \cite[Corollary B]{MolchoWise}. Existence of a modification which is a scheme with log structure is \cite[Theorem 1.11]{KajiwaraKatoNakayama2018}. 
    
    It remains to handle connected components of the log Picard group parameterising log line bundles of total degree $d$ for $d$ non-zero. The morphism $C^\circ \rightarrow S$ is assumed to admit a section $x$. Replacing $S$ with such an \'etale neighbourhood, tensoring with $\mathcal{O}(xd)$ gives an isomorphism from the connected component of the identity in the log Picard group to the total degree $d$ part of the log Picard group. 
\end{proof}
Recall that a site is said to be \textit{sub-canonical} if the functor of points of its objects form sheaves. The following fact is well known to experts, see for example \cite{MolchoWise,molcho2023remarks}.

\begin{example}[The big log \'etale site is not sub-canonical]
    If the big log \'etale site were sub-canonical then the functor of points of $\mathbb{A}^1$, and thus the structure sheaf of any log scheme, would form a log \'etale sheaf. This would contradict Remark \ref{remark: first instance of Max's favorite counterexample}.
\end{example}

Since the moduli problems $\mathfrak{M}_1(C/S)$ and $\operatorname{LogPic}(C/S)$ satisfy log \'etale descent, one cannot expect that they are represented by a log scheme, or even algebraic stack with log structure. The fact that the log Picard group is not represented by an algebraic space with log structure is discussed in detail in \cite{MolchoWise}. A related phenomenon occurs for logarithmic Quot spaces \cite{kennedy2023logarithmic}.

\begin{definition}\label{defn:logalgspace}
    A \emph{descending logarithmic algebraic space} is a sheaf on the big full log \'etale site of $S$ which, possibly after replacing $S$ with a strict \'etale cover, can be obtained by taking the log \'etale sheafification to the functor of points of an algebraic space with log structure.
\end{definition}

Note that descending log algebraic spaces are in particular logarithmic spaces in the second sense \cite[Section 10.1]{kajiwaraKatoNakayamaVII} whose functor of points satisfies log \'etale descent. 
Our definition of the log Picard group as a log \'etale sheafification and Theorem~\ref{thm:MolchoWise} together establish the following statement.

\begin{corollary}
    The log Picard group $\mathsf{LogPic}(C/S)$ is a descending log algebraic space of local finite presentation over $S$. 
\end{corollary}

\begin{remark}
    In the future, it may be desirable to define descending log algebraic spaces as a sheaf of sets on the big log \'etale site which admits a log \'etale cover by schemes with log structure. We expect that the proof of \cite[Theorem 6.6]{Alper2013GoodModuli} could be adapted to establish Theorem \ref{mainthm:LogPic} with this alternative definition once the theory is developed further. 
\end{remark}

\subsection{The canonical morphism} We construct the morphism $\varpi$ of Theorem \ref{mainthm:LogPic}

\subsubsection{Defining the canonical morphism}
Let $T$ be a Noetherian $S$-scheme. The data of a $T\rightarrow S$ valued point $$T\longrightarrow \mathfrak{M}_1(C/S)$$
is a log invertible sheaf, which Corollary \ref{corr: FP iff up FP} allows us to express as $\arrowup F$ for a conventional line bundle $F$ on a log alteration $C'_T$ of $C_T$.
By Proposition \ref{proposition-alteration-to-saturated-morphisms}, it is possible to replace both $C'_T$ and $T=T'$ by log alterations such that the morphism $C'_T\rightarrow T'$ is saturated.

Following \cite[Section 4.11]{MolchoWise}, the inclusion $\mathbb{G}_m\hookrightarrow \mathbb{G}_m^\mathsf{log}$ defines a map $$\mathbf{Pic}(C'_T/T')\longrightarrow \mathbf{LogPic}(C'_T/T')\cong \mathbf{LogPic}(C_T/T).$$ The isomorphism holds because $\mathbf{LogPic}(C_T/T)$ forms a sheaf in the big full log \'etale site of $S$. 

\begin{lemma}
    The assignment of the previous paragraph is well defined independent of the choice of log alteration $C'_T/T'$. Sheafifying this assignment thus defines a morphism $h$ of stacks over the big log \'etale site of $S$, $$\mathfrak{M}_1(C/S)\xlongrightarrow{h} \mathbf{LogPic}(C/S)\longrightarrow \ensuremath{\text{\sffamily\upshape LogPic}}(C/S).$$
    We write $\varpi$ for this composition.
\end{lemma}
\begin{proof}
    The only claim that must be checked is that the assignment is compatible with pullback along a log alteration $C'\rightarrow C$. This is true because $\mathbb{G}_m\rightarrow \mathbb{G}_m^\mathsf{log}$ is a morphism of (strict \'etale) sheaves on the big strict \'etale site.
\end{proof}

The following remarks orient our perspective.
\begin{remark}\label{rem:OGfunctor}
    While the fibre over $T\rightarrow S$ of the (conventional) Picard stack associated to a curve $\underline{C}/\underline{S}$ is the groupoid of line bundles on $C_T/T$, computing fibres of the picard group requires a sheafification step. This step identifies points which differ by a line bundle pulled back from $T$.   

    We compare to the logarithmic situation: The functor of points of the log Picard stack assigns to a Noetherian test object $T\rightarrow S$ the collection of equivalence classes of log invertible sheaves on $C_T$. The equivalence relation identifies log invertible sheaves that differ by (the up functor applied to) a line bundle pulled back from a subdivision to the Artin fan of $C_T$. See Lemma~\ref{lem:KernelChipFiring} for a precise statement. 
\end{remark}

\subsubsection{Fibres of the canonical morphism} We briefly recall the theory of Artin fans, and their line bundles.
An \emph{Artin fan} is an algebraic stack with a log structure which is strict \'etale over Olsson's stack of log structures.
The construction of \cite{ModStckTropCurve} defines a fully faithful embedding of the category $\mathsf{RPC}$ of rational polyhedral cones into the category of Artin fans $$a^\ast \colon \mathsf{RPC}\longrightarrow \text{Artin Fan}.$$
The functor $a^\ast$ sends a cone $\sigma$ to the stack quotient of the associated affine toric variety (with its toric log structure) by the dense torus.

For $X$ a Noetherian k-scheme with a log structure, there is an initial strict map from $X$ to an Artin fan, which we denote $\mathcal{A}_X$ and call \textit{the} Artin fan of $X$ \cite[Proposition 3.2.1]{abramovich2017boundedness}.

Our next lemma is well known to experts, see \cite[Remark 7.3]{ModStckTropCurve} and \cite[§2.5]{HolmesMolchoPandharipandePixtonSchmitt2025}, but we were unable to identify a reference which works in the necessary generality. This lemma interpolates two ways of articulating the theory of conewise (i.e. strict piecewise) linear functions; see \cite{HolmesSchwarz2022}, \cite{MolchoPandharipandeSchmitt2023}, and \cite{MolchoRanganathan2024}. 

\begin{lemma}\label{Lem:LBAfanMXgp}
    Let $X$ be a Noetherian k-scheme with a log structure with Artin fan $\mathcal{A}_X$. There is a canonical isomorphism between the following groups.
    \begin{enumerate}
        \item Isomorphism classes of line bundles on $\mathcal{A}_X$.
        \item Global sections to the groupification of the sheaf of monoids $\overline{M}_X = M_X/\mathcal{O}_X^\ast$.
    \end{enumerate}
\end{lemma}
\begin{proof}
    Define a sheaf $\mathsf{PL}$ of conewise linear functions on the small strict \'etale site of $X$ which assigns to a strict \'etale map $U\rightarrow X$ the collection of conewise linear functions on the cone complex associated to \textit{the} Artin fan $\mathcal{A}_U$. Recall that line bundles on $\cA_X$ correspond to global sections of $\mathsf{PL}$ by an argument generalising \cite[Theorem 0.1]{klyachko1989equivariant}. 
    
    To define an isomorphism of sheaves on the small strict \'etale site of $X$, $$\mathsf{PL}\longrightarrow \overline{M}_X^\mathsf{gp},$$ it will suffice to specify the morphism on a basis for the strict \'etale site. Since $X$ is Noetherian, the collection of atomic log schemes (in the sense of \cite[Definition 2.2.2.2]{MolchoWise}) form a basis for this topology \cite[Lemma 2.2.5]{AWbirational}. 

    The Artin fan of an atomic scheme with a log structure $U_i$ is of the form $a^\ast \sigma$ where $\sigma$ is the dual monoid to $\Gamma(U_i,\overline{M}_{U_i})$. By the argument of \cite[Lemma 2.2.5]{AWbirational}, it follows that $\operatorname{Hom}(\sigma,\mathbb{Z}) = \Gamma(U_i,\overline{\cM}_{U_i}^{gp})$ and the result is proved. 
\end{proof}

Lemma \ref{lem:KernelChipFiring}, which is again well known to experts, describes fibers of the canonical morphism. 

\begin{lemma}\label{lem:KernelChipFiring}
    Assume that $T$ is Noetherian, static, and atomic. Let $\mathcal{L}_1$, $\mathcal{L}_2$ be log invertible sheaves on $C_T$, specifying maps $$f_1,f_2 \in \operatorname{Hom}(T,\mathfrak{M}_1(C/S)).$$
    The following are equivalent:
    \begin{enumerate}
        \item The log line bundles $\mathcal{L}_1$ and $\mathcal{L}_2$ associated to the morphisms $f_1$ and $f_2$ are \textit{chip firing equivalent}, by which we mean the following: There exists a log alteration $C_T'\rightarrow C_T$ with Artin fan $\mathcal{A}_{C_T'}$ and conventional line bundles $L_1,L_2$ on $C_T'$ such that $$\mathcal{L}_1 = \arrowup L_1 \quad \mathcal{L}_2 = \arrowup L_2$$
        and such that $L_1\otimes L_2^{-1}$ is pulled back from a line bundle on $\mathcal{A}_{C'_T}$.
        \item The compositions
        $$\varpi\circ f_1 = \varpi\circ f_2\colon T \longrightarrow \mathsf{LogPic}(C/S)$$ coincide.
    \end{enumerate}
\end{lemma}
\begin{proof}
    Let $C_T'$ be any logarithmic alteration of $C_T$ upon which it is possible to express both $\mathcal{L}_1,\mathcal{L}_2$ as the up functor applied to a conventional invertible sheaf. Such a log alteration exists by Corollary~\ref{corr: FP iff up FP}. Consider the short exact sequence,
    $$1\longrightarrow \mathcal{O}_{C_T'}^\ast \longrightarrow M_{C_T'}^\mathsf{gp}\longrightarrow \overline{M}_{C_T'}^\mathsf{gp}\longrightarrow 1.$$ The associated long exact sequence in cohomology yields an exact sequence $$H^0(C_T',\overline{M}_{C_T'}^\mathsf{gp})\longrightarrow H^1(C_T',\mathcal{O}_{C_T'}^\ast ) \longrightarrow H^1(C_T', M_{C_T'}^\mathsf{gp}).$$ Thus two $\mathcal{O}_{C_T'}^\ast$-torsors induce the same ${M}_{C_T'}^\mathsf{gp}$-torsor if and only if their difference lies in the image of $H^0({C_T'},\overline{M}_{C_T'}^\mathsf{gp})$. Consequently Lemma \ref{Lem:LBAfanMXgp} asserts that the difference gives rise to a line bundle pulled back from $\mathcal{A}_{C_T'}$.
\end{proof}

\subsection{Universal property of the canonical morphism} In conventional algebraic geometry, the Picard stack admits a good moduli space (the Picard scheme). Whilst a logarithmic version of good moduli spaces in the sense of \cite{Alper2013GoodModuli,Alper2014Adequate,DeligneMumford1969} is not yet available, we prove that $\varpi$ satisfies a version of the universal mapping property satisfied by conventional good moduli spaces. In the vernacular of \cite{VBOlssonFan}, $\operatorname{LogPic}$ is a categorical moduli space for $\textbf{LogPic}$ in an appropriate category.

\begin{proposition}\label{prop:LogVBtoLogPic}
    Let $Y$ be a locally Noetherian\footnote{We call a descending log algebraic space locally Noetherian if its functor of points is the sheafification to the functor of points of a locally Noetherian log algebraic space with log structure.} descending log algebraic space over $S$. Every morphism from $\mathfrak{M}_1(C/S)$ to $Y$ admits a unique factorization through the canonical morphism $$\mathfrak{M}_1(C/S)\xlongrightarrow{\varpi} \mathsf{LogPic}(C/S)\longrightarrow Y.$$ Equivalently, the morphism $\varpi$ is universal for maps from $\mathfrak{M}_1(C/S)$ to locally Noetherian descending log algebraic spaces over $S$.
\end{proposition}

We will deduce Proposition \ref{prop:LogVBtoLogPic} from our next result, whose proof consists of Proposition \ref{prop:IntervalFactors} and Lemma~\ref{lem:Part2KeyProp}.

\begin{proposition}\label{prop:TheInterval}
    Let $T$ be an atomic globally charted Noetherian log scheme with locally free log structure. Then there exists an algebraic stack with a log structure $\pi\colon T_I\rightarrow T$, and two sections $\iota_1,\iota_2$ of $\pi$ which satisfy the following properties.
    \begin{enumerate}
        \item\label{prop:The interval property I} Fix any morphism $f\colon T_I\rightarrow Y$ where $Y$ is a locally Noetherian descending log algebraic space over $S$. Then $f\circ \iota_1 = f \circ \iota_2$.
        \item \label{prop: the interval property II} Let $C\rightarrow T$ be a vertical log smooth curve and let $\mathcal{L}_1$, $\mathcal{L}_2$ be chip firing equivalent log invertible sheaves. 
        Write $C_I$ for the base change of $C$ along $T_I\rightarrow T$ and denote the base changes of $\iota_1,\iota_2$ by $s_1,s_2$. There are log alterations of $C_I$ and $T_I$ such that there exists a line bundle $L'$ on $C_I'$ with the property that $$\arrowup s_1^\ast L' = \mathcal{L}_1, \quad \quad \arrowup s_2^\ast L' = \mathcal{L}_2.$$
    \end{enumerate}
\end{proposition}
In the vernacular of \cite{VBOlssonFan}, the $T_I$ we construct will be an Olsson fan over $T$. We prove (and further discuss) Proposition \ref{prop:TheInterval} in Section \ref{sec:theInterval}. 

\begin{proof}[Proof of Proposition \ref{prop:LogVBtoLogPic}]
    Let $Y$ be a locally Noetherian descending log algebraic space over $S$. In light of Remark \ref{remark: bounded monodromy} \eqref{item: M1 to log Pic is surjective kind of}, a morphism $$\mathfrak{M}_1(C/S) \longrightarrow Y$$ factors through $\mathsf{LogPic}(C/S)$ precisely when the following property holds for all test schemes $T$. Given any two elements $f_1,f_2$ of $\mathsf{Hom}(T,\mathfrak{M}_1(C/S))$ whose images in $\mathsf{Hom}(T,\mathsf{LogPic}(C/S))$ are the same, then also their images under $\varpi$ in $\mathsf{Hom}(T,Y)$ are the same.

    Since both $Y$ and $\mathsf{LogPic}({C/S})$ are locally Noetherian, it will suffice to test on $T$ valued points for $T$ a Noetherian log $S$-scheme.
    The claim is log \'etale local on $T$.
    As every Noetherian log scheme admits a cover by atomic log schemes, we are free to assume that $T$ is atomic. Since the functor of points of the same spaces satisfy log \'etale descent, and since also every Noetherian scheme with a log structure admits a static log modification, it suffices to handle the case that $T$ is Noetherian with a locally free log structure, and is both globally charted and static, and atomic. 

    Suppose $f_1,f_2$ correspond to log invertible sheaves $\arrowup L_1,\arrowup L_2$ on $C_T$. Passing to a common refinement if necessary, we may assume $L_1,L_2$ are conventional line bundles on the same log alteration $C_T'\rightarrow C_T$. 
    By Lemma \ref{lem:KernelChipFiring}, $L_1\otimes L_2^{-1}$ is a line bundle pulled back from the Artin fan of $C_T'$.
    
    Proposition \ref{prop:TheInterval} \eqref{prop: the interval property II} furnishes a map $h\colon T_I\rightarrow \mathfrak{M}_1(C/S)$ such that $h\circ \iota_1 =f_1$ and $f_2= h\circ \iota_2$.  By Proposition \ref{prop:TheInterval} \eqref{prop:The interval property I}, we deduce that $g\circ f_1=g\circ f_2$ and the result is proved.
\end{proof}

\subsection{The interval over $T$}\label{sec:theInterval} 
We call two $T$-valued points of $\mathfrak{M}_{1}(C/S)$ \textit{$\ell$-equivalent} if they are identified under $\varpi$. This section details how one can detect $\ell$-equivalence using the \textit{interval} $T_I$ over a log scheme $T$, see Proposition \ref{prop:TheInterval}. For $T$ a geometric point, $T_I$ is a replacement for the stack $\Theta = [\mathbb{A}^1/\mathbb{G}_m]$ which is used in the conventional theory to detect $S$-equivalence.


\subsubsection{Defining the interval} Let $T$ be a Noetherian globally charted atomic $k$-scheme with a locally free log structure. Thus $T$ admits a smooth Artin fan $a^\ast \sigma$, which is in particular an Artin cone.
Summing projection maps defines a morphism 
$$T\longrightarrow \mathcal{A}_T = a^\ast \mathbb{N}^k = a^\ast \sigma\longrightarrow a^\ast \mathbb{N}.$$ 
\begin{definition}
    The \textit{interval over $T$} is the $T$-stack $$\pi_I\colon T_I = T\times_{a^\ast \mathbb{N}}a^\ast \mathbb{N}^2\longrightarrow T,$$
    where the second map is the (saturated) morphism induced by summation $$\mathbb{N}^2\longrightarrow \mathbb{N}.$$ Write $\sigma_I = \sigma\times_\mathbb{N}\mathbb{N}^2$.
\end{definition} 

There are two canonical strict sections of the projection map $\pi_I\colon T_I\rightarrow T$ denoted $$\iota_1\colon T \longrightarrow T_I \quad \iota_2\colon T \longrightarrow T_I,$$ and corresponding to coordinate inclusions of $\mathbb{N}$ into $\mathbb{N}^2$. 
Proposition \ref{prop:TheInterval} shows that $T_I$ detects $\ell$-equivalence just as the stack $[\mathbb{A}^1/\mathbb{G}_m]$ detects S-equivalence in the theory of good moduli spaces.

\subsubsection{Chip firing revisited} The connection between the combinatorial operation of chip firing and families over the interval $T_I$ over $T$ is provided by our next result.

\begin{lemma}\label{lem:Part2KeyProp}
    The interval $T_I$ satisfies property \eqref{prop: the interval property II} of Proposition \ref{prop:TheInterval}.
\end{lemma}

Since $T$ is globally charted, its log structure is monodromy free and for any morphism $X\rightarrow T$ there is a commutative diagram of Artin fans,
    $$
    \begin{tikzcd}
X \arrow[r] \arrow[d] & T \arrow[d] \\
\mathcal{A}_X \arrow[r] & \mathcal{A}_T.
\end{tikzcd}
    $$

\begin{proof}[Proof of Lemma \ref{lem:Part2KeyProp}]
    Since $\mathcal{L}_1$ and $\mathcal{L}_2$ are chip firing equivalent we may make the following assumption after replacing $C$ by a log alteration $C'\rightarrow C$ whose Artin fan $\mathcal{A}_{C'}$ is smooth and of the form $a^\ast \Sigma$ with $\Sigma$ a cone complex.
    There exists a conventional line bundle $L$ pulled back from a line bundle on $\mathcal{A}_{C'} = a^\ast \Sigma$ associated to a piecewise linear function $h$ on $\Sigma$ 
    such that $\mathcal{L}_1 = \mathcal{L}_2 \otimes\arrowup L$.

    Note that $C'\times_T T_I \cong C'\times_C(C\times_{T} T_I)$ is both a log alteration of $C\times_T T_I$, and admits a strict map to $a^\ast(\Sigma \times_\sigma \sigma_I)$. Since $C$ is vertical, the cone complex $\Sigma \times_{\sigma}\sigma_I$ contains two copies $\Sigma_1,\Sigma_2$ of $\Sigma$ as the preimage to the two rays of $\mathbb{N}^2$. These copies are disjoint except at the origin. It follows that it is possible to find a conewise linear function $k$ on a subdivision of $\Sigma \times_{\sigma}\sigma_I$ whose restriction to $\Sigma_1$ is $h$ and whose restriction to $\Sigma_2$ is the zero function. Indeed, one can subdivide $\Sigma \times_{\sigma}\sigma_I$ to form a semistable cone complex, and conewise linear functions on semistable cone complexes biject with assignments of an integer to each ray. Write $C_I''$ for the resulting log alteration of $C\times_T T_I$. Let $E$ be the pullback of the line bundle corresponding to the piecewise linear function $k$ to $C_I$.

    The desired log invertible sheaf $\mathcal{L}$ is obtained by tensoring $\Uparrow\!\! E$ with the pullback of $\mathcal{L}_2$ along the map $C_I\rightarrow C$.
\end{proof}

\subsubsection{The tropical meaning of $\ell$-equivalence}\label{sec:TropicalVBmeaning}
We offer a tropical perspective on $\ell$-equivalence. More precisely, assume that $S$ admits a global chart and an associated Artin fan $a^\ast \kappa$. In \cite{Kennedy-Hunt-Poiret-Song:Log-Vector-bundle} a commutative diagram of the form 
$$
\begin{tikzcd}
\mathfrak{M}_1(C/S) \arrow[r, "\varpi"] \arrow[d] & \operatorname{LogPic}(C/S) \arrow[d] \\
a^\ast\left( \mathsf{PL}(\operatorname{Trop}(C)/\kappa)\right) \arrow[r,"\varpi^{\operatorname{trop}}"] & a^\ast \left(\operatorname{TroPic}(\operatorname{Trop}(C)/\kappa)\right)
\end{tikzcd}
$$
is constructed. Here, the space $\operatorname{TroPic}(\operatorname{Trop}(C))/\kappa$ is the tropical Picard group in the sense of \cite{MolchoWise}, and $\mathsf{PL}(\operatorname{Trop}(C)/\kappa)$ can be thought of as either the moduli space of piecewise linear functions, or the moduli space of (decorated) tropical supports: the undecorated space of tropical supports was introduced in \cite{kennedy2023logarithmic}. The situation for higher rank vector bundles is more subtle.

The morphism $\varpi^{\operatorname{trop}}$ suggests a perspective on two differing approaches to higher rank tropical vector bundles which have been explored in recent years. 
On the one hand, a notion of tropical vector bundle, closely related to the theory of matroids but unrelated to chip firing, has recently been defined \cite{KavehManonTropVB,KhanMaclagan2026} motivated by connections between buildings and toric vector bundles \cite{KavehManon2022}. We expect that the rank one case of the latter construction is related to piecewise linear functions on $\operatorname{Trop}(C)$, and thus the tropicalization of $\mathfrak{M}_1$. 

On the other hand, $\operatorname{TroPic}(\operatorname{Trop}(C)/\kappa)$ is a moduli space reflecting the combinatorial structure of chips on a graph up to chip firing, see \cite{Baker2008,BakerNorine2007,GathmannKerber2008,MikhalkinZharkov2008}. Authors have suggested studying (a notion of) vector bundles on a graph up to (an appropriately generalized) notion of chip firing \cite{Allermann2012,GrossKaurWernerTropical,GrossZakharov2022}. More generally, $G$-torsors for $G$ a reductive group can be studied in this context \cite{GrossKaurWernerTropical}.

The results of \cite{GrossKaurWernerTropical,GrossZakharov2022} support our interpretation of chip firing as a version of S-equivalence in the following sense. Recall that the moduli stack of semi-homogeneous vector bundles on an abelian variety admits a good moduli space, whereas the moduli stack of vector bundles on a scheme typically does not admit a good moduli space (even though the locus picked out by a stability condition may admit a good moduli space). When the moduli stack of vector bundles does not admit a good moduli space, the connection between algebraic geometry and the geometry of higher rank vector bundles up to chip firing is less clear. 

Finally, we imagine a map from $\sigma_I$ to some tropical moduli problem $\mathcal{T}$ is a kind of homotopy between two maps from $\sigma$ to $\mathcal{T}$.

\subsubsection{Elementary $T$-Olsson fans}\label{sec:elemolsson}
It remains to prove that $T_I$ has the first property in Proposition~\ref{prop:TheInterval}. The proof works log \'etale locally on $T_I$: the log \'etale local model for the map $T_I\rightarrow T$ is captured by the definition of an elementary $T$-Olsson fan.

Let $T$ be a globally charted log $k$-scheme which admits an Artin fan $a^*\sigma \cong a^\ast\mathbb{N}^k$.
Choose a (coordinate) ray $\rho$ in the cone $\sigma\cong \mathbb{N}^k$.
There is a strict map $$T\longrightarrow a^\ast \sigma.$$
We write $f_\rho\colon a^\ast \sigma\longrightarrow a^\ast \mathbb{N}$ for the coordinate projection to the ray $\rho$.

\begin{definition}
    The \textit{elementary $T$-Olsson fan at $\rho$} is the fiber product $$T^\rho = T\times_{a^\ast \mathbb{N}} a^\ast \mathbb{N}^2$$ where the morphism from $T$ to $a^\ast \mathbb{N}$ is obtained by composing $$T\longrightarrow a^\ast \sigma \xrightarrow{f_{\rho}} a^\ast \mathbb{N}$$
    and the morphism $\mathbb{N}^2\rightarrow \mathbb{N}$ is induced by the summation map.
\end{definition}
\begin{remark}
    Let $V$ be the quotient stack of $V(XY)$ in $\mathbb{A}^2$ by the diagonal torus action, considered as a log scheme with log structure pulled back from $\mathbb{A}^2$. If $T$ is a scheme with a divisorial log structure from a smooth divisor $D$, then $T_I\rightarrow T$ is an isomorphism away from $D$. Over $D$, the stack $T_I$ is a $V$-bundle.
\end{remark}
The morphism $T^\rho\rightarrow T$ admits two canonical sections, denoted $\iota_a,\iota_b$, which are induced by the coordinate sections to the summation map $\mathbb{N}^2\longrightarrow \mathbb{N}.$

\begin{lemma}\label{lem:ElementaryFanFactors}
    Given $Y^\text{pre}$ an algebraic space with log structure, $T$ a noetherian scheme with locally free and globally charted log structure, and a solid commutative diagram $$\begin{tikzcd}
T^\rho \arrow[r] \arrow[rr, bend left=20,"f"] & T \arrow[r, dashed] & Y^\text{pre}
\end{tikzcd}.
$$
Then there is a dotted arrow as above making the whole diagram commute. It follows that the dotted arrow agrees with $f\circ \iota_a$ and $f\circ \iota_b$, which are therefore equal. 
\end{lemma}

\begin{proof}
    Note that $T^\rho$ is an Olsson cone over $T$, and thus the map $T^\rho \rightarrow T$ is a good moduli space \cite[Proposition 5.6]{VBOlssonFan}.
    The existence of the dotted morphism on the level of underlying algebraic spaces is now \cite[Theorem 6.6]{Alper2013GoodModuli}. 

    It now remains to verify that there is also a corresponding morphism of log structures, which (by uniqueness) is a strict \'etale local question in both $Y^\text{pre}$ and $T$. Hence, we may assume that $T = \operatorname{Spec}(Q\rightarrow A)$ and $Y^\text{pre} = \operatorname{Spec}(P \rightarrow B)$ are affine and globally charted. By \cite[Lemma 2.2.5]{AWbirational} we may further assume that $Y^\text{pre}$ and $T$ are atomic with $Q = \overline{\cM}_{T,t_0}$ for some $t_0\in T$. By op.cit., this implies $Q = \Gamma(T,\overline{\cM}_T)$.

    Write $W$ for the affine scheme with a log structure $$W\colon = T\times_{a^\ast\mathbb{N}}[\mathbb{A}^2/\mathbb{G}_m] \cong T\times_{\mathbb{A}^1}\mathbb{A}^2\cong \mathsf{Spec}(A\otimes_{k[t]}k[u,v]),$$ where the morphism $\mathbb{A}^2\rightarrow \mathbb{A}^1$ sends $(u,v)$ to $uv$ and in the first expression $\mathbb{G}_m$ acts on $\mathbb{A}^2$ via the character $(1,0)$. 
    
    It is possible to present $T^\rho$ as a global quotient stack
    $$T^\rho \cong [W/\mathbb{G}_m],$$ where $\mathbb{G}_m$ acts via the cocharacter $(1,-1)$ in $\mathbb{A}^2$. 
    
    We now claim that it suffices to show that
    \begin{equation}\label{diagram: equalizer}
    \begin{tikzcd}
        \Gamma(T,\cM_T)\rar &\Gamma(W,\cM_W)\rar["m^*"'
        , shift right] \rar["\pi_2^*", shift left] &\Gamma(\mathbb{G}_m\times W,\cM_{\mathbb{G}_m\times W})
    \end{tikzcd}
    \end{equation}
    is an equalizer diagram. Indeed, we already know that $W\rightarrow Y^{\text{pre}}$ factors through $T$ on the level of underlying algebraic spaces. To show that it does so on the level of log structures, it suffices to show that the resulting $P\rightarrow \Gamma(W,\cM_W)$ factors through $\Gamma(T,\cM_T)$ in a way that is compatible with the morphism of rings $B\rightarrow A$. But since $W\rightarrow T$ is $\mathbb{G}_m$-invariant, it follows that $P$ maps into the equalizer of $m^*$ and $\pi_2^*$ and the compatibility with $B\rightarrow A$ follows from the fact that $A\rightarrow A\otimes_{k[t]}k[u,v]$ is injective.

    To show that \eqref{diagram: equalizer} is an equalizer, consider the left exact sequence\footnote{By this we mean that any two sections of the monoid $\Gamma(T,\cM_T)$ have the same image in $\Gamma(T,\overline{\cM}_T)$ if and only if they differ by a unique element in $\Gamma(T,\cO_T^\times)$.} $0\rightarrow \Gamma(T,\cO_T^\times)\rightarrow \Gamma(T,\cM_T)\rightarrow \Gamma(T,\overline{\cM}_T) = Q$. Since the chart $Q\rightarrow \Gamma(T,\cM_T)$ provides a splitting for the right hand map, we get $\Gamma(T,\cM_T) = Q\oplus \Gamma(T,\cO^\times_T)$ and by the same argument $\Gamma(W,\cM_W) = Q\oplus_{\mathbb{N}}\mathbb{N}^2\oplus \Gamma(W,\cO^\times_W)$ and $\Gamma(\mathbb{G}_m\times W,\cM_{\mathbb{G}_m\times W}) = Q\oplus_{\mathbb{N}}\mathbb{N}^2\oplus \Gamma(\mathbb{G}_m\times W,\cO^\times_{\mathbb{G}_m\times W})$. Note that there is a morphism $\phi\colon \Gamma(\mathbb{G}_m\times W,\cO^\times_{\mathbb{G}_m\times W}) = ((A\otimes_{k[t]}k[u,v])[T^\pm])^\times\rightarrow \mathbb{Z}$ given by the $T$-valuation and $\pi_2^*$ maps into the kernel of $\phi$, while $\phi\circ m^*$ restricted to $\mathbb{N}^2$ is given by $(a,b)\rightarrow a-b$. As a result, we have $\text{Eq}(m^*,\pi_2^*) \subset Q\oplus \Gamma(\mathbb{G}_m\times W,\cO^\times_{\mathbb{G}_m\times W})$. Since the diagram \eqref{diagram: equalizer} is an equalizer when we replace $\cM$ by $\cO^\times$, we further get $\text{Eq}(m^*,\pi_2^*)\subset Q\oplus \Gamma(T,\cO^\times_T) = \Gamma(T,\cM_T)$ as desired.
 \end{proof}

\subsubsection{Key property of the interval}
Given a morphism of algebraic spaces with log structure $g\colon U\rightarrow V$, we say log alterations $U',V'$ of $U,V$ respectively are compatible with $g$ if $U'$ is a log alteration of $U\times_{V} V'$. We will abuse notation and identify $g$ with the induced morphism $U'\rightarrow V'$.
\begin{proposition}\label{prop:IntervalFactors}
Let $Y$ be a Noetherian descending log algebraic space, and $T$ a globally charted atomic Noetherian log scheme. For any solid commutative diagram $$\begin{tikzcd}
T_I \arrow[r] \arrow[rr, bend left=20,"f"] & T \arrow[r, dashed] & Y
\end{tikzcd}
$$
there is a dotted arrow as above making the entire diagram commute. It follows that the dotted arrow agrees with $f\circ \iota_1$ and $f\circ \iota_2$, which are therefore equal. 
\end{proposition}

\begin{proof}
The proof proceeds by reducing to Lemma~\ref{lem:ElementaryFanFactors} in the following three steps. 

\medskip
\noindent\textbf{Step 1: strictifying a log \'etale cover.}
Since $Y$ is a Noetherian descending log algebraic space, its functor of points is the log \'etale sheafification to a Noetherian algebraic space with log structure $Y^\text{pre}$. By the same proof as Corollary \ref{Key corollary} \eqref{item-2-key-corollary}, it follows that $f$ arises from a morphism $f^{\text{pre}}\colon T'_I\rightarrow Y^{\text{pre}}$ with $T'_I\rightarrow T_I$ a log alteration.
Whilst we do not name the Artin fan of $T_I'$, certainly $T_I$ admits a strict map to $a^\ast \sigma_I$. By Lemma~\ref{lem:SubdivideConeCX}, after replacing $T_I'$ with a log alteration if necessary, we may assume that $T_I'$ admits a strict map to the image under $a^\ast$ of a subdivision $\sigma_I'$ of $\sigma_I$.

\medskip
\noindent\textbf{Step 2: compatibility with log structure on $T$.} Applying (not weak) toric semistable reduction \cite{AbramovichKaru,SemistabRedChar0,molcho2019universal} to the composition $\sigma_I'\rightarrow \sigma_I \rightarrow \sigma$, we can replace $T_I'$ and $T$ with log alterations such that $f^\text{pre}$ is saturated, and the log structure of $T_I'$ is locally free.
Our statement remains strict \'etale local on $T$, and thus we may continue to assume that $T$ is an atomic log scheme.

We have just arranged that $\sigma_I'$ is integral and saturated over $\sigma$. 
Write $\iota_0^\mathsf{trop}$ and $\iota_k^\mathsf{trop}$ for the canonical sections to the map $\sigma_I'\longrightarrow \sigma$ induced by sections $\iota_1,\iota_2$ to  
$\sigma_I = \sigma\times_\mathbb{N}\mathbb{N}^2\longrightarrow \sigma.$ Here $k$ is a positive integer whose value we assign in the next paragraph.

\medskip
\noindent\textbf{Step 3: reducing to elementary $T$-Olsson fans.} Before explaining Step 3 we illustrate the case that $\sigma$ is a two dimensional cone (i.e. a cone over an interval). Consider the following diagram which is explained by reading the following two paragraphs.
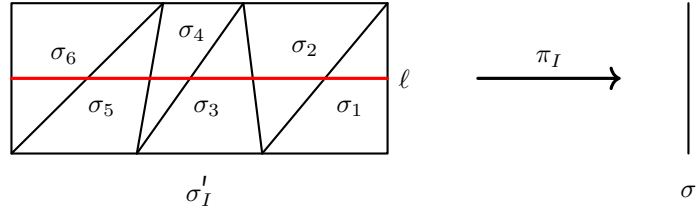
\[
\begin{tikzpicture}[scale=1, line cap=round, line join=round]
  \begin{scope}
    \coordinate (A) at (0,0);
    \coordinate (B) at (5,0);
    \coordinate (C) at (5,2);
    \coordinate (D) at (0,2);

    \draw[thick] (A)--(B)--(C)--(D)--cycle;

    \def\w{5/3}
    \def\epsone{0.35}
    \def\epstwo{-0.25}

    \coordinate (P1b) at (\w,0);
    \coordinate (P2b) at (2*\w,0);
    \coordinate (P1t) at (\w+\epsone,2);
    \coordinate (P2t) at (2*\w+\epstwo,2);

    \draw[thick] (P1b)--(P1t);
    \draw[thick] (P2b)--(P2t);

    \draw[thick] (A)--(P1t);
    \draw[thick] (P1b)--(P2t);
    \draw[thick] (P2b)--(C);

    \draw[red, very thick] (0,1)--(5,1)
      node[right, black] {\small $\ell$};

    \node at (3.9,1.4) {\small $\sigma_2$};
    \node at (4.5,0.6) {\small $\sigma_1$};

    \node at (2.6,0.6) {\small $\sigma_3$};
    \node at (2.4,1.6) {\small $\sigma_4$};

    \node at (1.2,0.6) {\small $\sigma_5$};
    \node at (0.7,1.3) {\small $\sigma_6$};

    \node at (2.5,-0.5) {\small $\sigma_I'$};
  \end{scope}

  \draw[->, very thick] (6.2,1)--(8.1,1)
    node[midway, above] {\small $\pi_I$};

  \begin{scope}[xshift=9cm]
    \coordinate (E) at (0,0);
    \coordinate (F) at (0,2);
    \draw[thick] (E)--(F);
    \node at (0,-0.5) {\small $\sigma$};
  \end{scope}
\end{tikzpicture}
\]
Since the composed morphism of cone complexes $$\sigma_I'\longrightarrow \sigma_I = \sigma\times_\mathbb{N}\mathbb{N}^2\longrightarrow \sigma$$ is integral and of relative dimension one, we can describe its maximal cones $\sigma_i$ somewhat explicitly. Let $p_0$ be the sum of primitive vectors along the rays in the image of $\iota_0^\mathsf{trop}$, and similarly define $p_k$ in the image of $\iota_k^\mathsf{trop}$. Thus $p_0$ and $p_k$ are points in the rational polyhedral cone $\sigma_I.$ The convex hull of $p_0$ and $p_k$ in $\sigma_I$ defines a line. The preimage of this line along the alteration $\sigma'_I\rightarrow \sigma_I$ defines a line $\ell$ in the cone complex $\sigma_I'$. 

Since the projection map to $\sigma$ is integral, the line $\ell$ passes through the interior of every maximal cone in $\sigma_I'$. We define an order $$\{\sigma_1,...,\sigma_k\}$$ on the maximal cones in $\sigma_I'$ according to the order in which $\ell$ intersects them, with $\sigma_1$ the cone containing $p_0$ considered first and the cone containing $\sigma_k$ containing $p_k$ considered to be the last.

The morphism $\sigma_I'\rightarrow \sigma$ admits $k+1$ sections $s_0,...,s_k$ corresponding to the $k+1$ facets meeting the line $\ell$. We adopt the convention that $s_0$ is induced by $\iota_0^\mathsf{trop}$, $s_k$ is induced by $\iota_k^\mathsf{trop}$, and the order of the other $s_i$ reflects the order in which facets meet $\ell$. 
The maximal cones $\sigma_i$ define $k$ elementary $T$-Olsson fans $T^{\rho_1}_1,...,T^{\rho_k}_k$ which jointly form a Zariski open cover of $T_I'$. Note that the tropical sections $s_{i}$ and $s_{i+1}$ induce algebro--geometric sections $s_i^{+}$ and $s_{i+1}^-$ of the morphism $T^{\rho_k}_k\rightarrow T.$ The sections $s_i^{+}$ and $s_{i+1}^-$ are $\iota_a,\iota_b$ in the notation of Section \ref{sec:elemolsson}.

Since $s_1^-$ coincides with $\iota_1$ and $s_k^+$ coincides with $\iota_2$, the result follows from Lemma \ref{lem:ElementaryFanFactors}.
\end{proof}

Finally, we establish the required technical lemma.
\begin{lemma}\label{lem:SubdivideConeCX}
    Let $X'\rightarrow X$ be a log modification of qcqs log schemes, and assume that $X$ admits a strict map to a (qc) Artin fan $a^\ast \Sigma$ with $\Sigma$ a cone stack.
    Then there exists a subdivision $\Sigma''\rightarrow \Sigma$ such that there is a factorization of log alterations,
    $$X'' = X\times_{a^\ast \Sigma} {a^\ast \Sigma''}\longrightarrow X'\longrightarrow X.$$
\end{lemma}
\begin{proof}
    After subdivision and base change the problem, we may replace $\Sigma$ by a smooth cone complex, which we may identify with the underlying cone complex of a fan with associated toric variety $Y_\Sigma$. We now base change the morphism $$X'\longrightarrow X \longrightarrow a^\ast \Sigma$$ along the log smooth cover $Y_\Sigma\rightarrow a^\ast \Sigma$ and apply Proposition \ref{proposition-alteration-to-saturated-morphisms} to the resulting composition. This defines the required subdivision $\Sigma''$ of $\Sigma$. We obtain a log alteration $X'' = X'\times_{a^\ast \Sigma} a^\ast \Sigma' \rightarrow X'$. We claim that the monomorphism $$X'\times_{a^\ast \Sigma} a^\ast \Sigma'\hookrightarrow X\times_{a^\ast \Sigma} a^\ast \Sigma'$$ is an isomorphism. The claim is strict \'etale local on $X$, and so by definition of log modification we can assume that we can write $X' = X\times_{a^\ast \Sigma} a^\ast \hat{\Sigma}$ since it is locally the (global) base change of a torus equivariant map of toric varieties. Such a map of toric varieties specifies the required subdivision of cone complexes $\hat{\Sigma}\rightarrow \Sigma$. Since $X'\rightarrow a^\ast \Sigma'$ is strict, we know that the subdivision $\Sigma''$ refines $\hat{\Sigma}$ and thus $$X'\times_{a^\ast \Sigma} a^\ast \Sigma'\cong X\times_{a^\ast \Sigma} \left(a^\ast \hat{\Sigma}\times_{a^\ast \Sigma}a^\ast \Sigma'\right)\cong X\times_{a^\ast \Sigma} a^\ast \Sigma'.$$
\end{proof}

\bibliographystyle{alpha}
\bibliography{log.bib}

\end{document}